\def\todaysdate{8\textsuperscript{th} December 2025}
\documentclass[reqno,a4paper]{article}
\usepackage{etex}
\usepackage[T1]{fontenc}
\usepackage[utf8]{inputenc}
\usepackage{verbatim}
\usepackage{lmodern}
\usepackage{subcaption}

\usepackage{amssymb,amsmath,amsthm}
\usepackage{thmtools,xcolor}
\usepackage{units}
\usepackage{graphicx}
\usepackage[all]{xy}
\usepackage{tikz}
\usetikzlibrary{arrows,calc,positioning,decorations.pathreplacing,cd}
\usepackage{relsize}
\usepackage{mhsetup}
\usepackage{mathtools}
\usepackage{float}
\usepackage{stmaryrd}
\usepackage{tocloft}
\usepackage{paralist} 
\usepackage[left=3cm,right=3cm,top=3cm,bottom=3cm]{geometry} 
\usepackage[bottom]{footmisc}
\usepackage[colorlinks,citecolor=blue,linkcolor=blue,urlcolor=blue,filecolor=blue,breaklinks]{hyperref}
\usepackage{footnotebackref}
\usepackage[format=plain,indention=1em,labelsep=quad,labelfont={up},textfont={normalsize},margin=2em]{caption}
\usepackage[mathscr]{euscript}
\usepackage{accents}

\newcommand{\doublearrow}[1]{\xrightarrow[]{#1}\mathrel{\mkern-14mu}\rightarrow}

\definecolor{lightblue}{rgb}{0.8,0.8,1}
\definecolor{carmine}{rgb}{0.59, 0.0, 0.09}

\usepackage{enumitem}
\setlist[enumerate]{nosep}
\setlist[itemize]{nosep}

\numberwithin{equation}{section}
\numberwithin{figure}{section}

\definecolor{vdarkred}{rgb}{0.7,0,0}
\declaretheoremstyle[
  spaceabove=\topsep,
  spacebelow=\topsep,
  headpunct=,
  numbered=no,
  postheadspace=1ex,
  headfont=\color{vdarkred}\normalfont\bfseries,
  bodyfont=\normalfont\itshape,
]{colored}
\declaretheoremstyle[
  spaceabove=\topsep,
  spacebelow=\topsep,
  headpunct=,
  numbered=no,
  postheadspace=1ex,
  headfont=\normalfont\bfseries,
  bodyfont=\normalfont\itshape,
]{italic}
\declaretheoremstyle[
  spaceabove=\topsep,
  spacebelow=\topsep,
  headpunct=,
  numbered=no,
  postheadspace=1ex,
  headfont=\normalfont\bfseries,
  bodyfont=\normalfont\upshape,
]{upright}

\declaretheorem[style=italic,name=Theorem,numbered=yes,numberwithin=section]{thm}
\declaretheorem[style=italic,name=Lemma,numbered=yes,numberlike=thm]{lem}
\declaretheorem[style=italic,name=Proposition,numbered=yes,numberlike=thm]{prop}

\declaretheorem[style=italic,name=Conjecture,numbered=yes,numberlike=thm]{conjecture}

\declaretheorem[style=upright,name=Definition,numbered=yes,numberlike=thm]{defn}
\declaretheorem[style=upright,name=Remark,numbered=yes,numberlike=thm]{rmk}

\declaretheorem[style=upright,name=Notation,numbered=yes,numberlike=thm]{notation}
\declaretheorem[style=upright,name=Convention,numbered=yes,numberlike=thm]{convention}

\makeatletter
\renewcommand*{\@seccntformat}[1]{\upshape\csname the#1\endcsname.\hspace{1ex}}
\renewcommand*{\section}{\@startsection{section}{1}{\z@}%
	{2.5ex \@plus 1ex \@minus 0.2ex}%
	{1.5ex \@plus 0.2ex}%
	{\normalfont\normalsize\bfseries}}
\renewcommand*{\subsection}{\@startsection{subsection}{2}{\z@}%
	{2.5ex \@plus 1ex \@minus 0.2ex}%
	{-1.5ex \@plus -0.2ex}%
	{\normalfont\normalsize\bfseries}}
\renewcommand*{\subsubsection}{\@startsection{subsubsection}{3}{\z@}%
	{2.5ex \@plus 1ex \@minus 0.2ex}%
	{-1.5ex \@plus -0.2ex}%
	{\normalfont\normalsize\bfseries}}
\renewcommand*{\paragraph}{\@startsection{paragraph}{4}{\z@}%
	{2.5ex \@plus 1ex \@minus 0.2ex}%
	{-1.5ex \@plus -0.2ex}%
	{\normalfont\normalsize\bfseries}}
\renewcommand*{\subparagraph}{\@startsection{subparagraph}{5}{\z@}%
	{2.5ex \@plus 1ex \@minus 0.2ex}%
	{-1.5ex \@plus -0.2ex}%
	{\normalfont\normalsize\slshape}}
\makeatother

\newcommand{\la}{\bar{\lambda}}
\newcommand{\FA}{\mathscr F_{\bar{i},\cN}}
\newcommand{\LA}{\mathscr L_{\bar{i},\cN}}
\newcommand{\FAnn}{\mathscr F_{\bar{i},\cN+1}}
\newcommand{\LAnn}{\mathscr L_{\bar{i},\cN+1}}
\newcommand{\FAM}{\mathscr F_{\bar{i},\cM}}
\newcommand{\LAM}{\mathscr L_{\bar{i},\cM}}
\newcommand{\FFA}{ F_{\bar{i},\cN}}
\newcommand{\LLA}{ L_{\bar{i},\cN}}
\newcommand{\FFAM}{F_{\bar{i},\cM}}
\newcommand{\LLAM}{ L_{\bar{i},\cM}}

\newcommand{\PJ}{J{ \Gamma}^{\bar{N}}(\beta_n)}
\newcommand{\PA}{\Gamma^{\cN}(\beta_n)}
\newcommand{\PAM}{\Gamma^{\mathcal M}(\beta_n)}

\newcommand{\PAA}{\hat{\Gamma}^{\cN}(K)}
\newcommand{\PAAJ}{\hat{\Gamma}^{\cN,J}(K)}
\newcommand{\PAAW}{\hat{\Gamma}^{\cN,W}(K)}

\newcommand{\PAANJ}{\hat{\Gamma}^{\cN+1,J}(K)}

\newcommand{\PAAr}{\hat{\Gamma}^{R,\cN}(\beta_n)}

\newcommand{\LL}{\mathbb L}
\newcommand{\LLh}{\hat{\mathbb L}}
\newcommand{\LLhJ}{\hat{\mathbb L}^J}
\newcommand{\LLhW}{\hat{\mathbb L}^W}
\newcommand{\LLH}{\mathbb L^H}
\newcommand{\LLr}{\mathbb L^R}

\newcommand{\LLNJ}{\hat{\mathbb L}^{J}_{\mathcal N}}
\newcommand{\LLNW}{\hat{\mathbb L}^{W}_{\mathcal N}}

\newcommand{\LLNNJ}{\hat{\mathbb L}^{J}_{\mathcal N+1}}
\newcommand{\LLNNW}{\hat{\mathbb L}^{W}_{\mathcal N+1}}

\newcommand{\LaJ}{\hat{\Lambda}^{J}}

\newcommand{\LNJ}{\Lambda^{J}_{\mathcal N}}

\newcommand{\LMJ}{\Lambda_{\mathcal M}^{J}}
\newcommand{\LNt}{\hat{\Lambda}_{\mathcal N}}
\newcommand{\LNtJ}{\hat{\Lambda}_{\mathcal N}^{J}}

\newcommand{\LNNtJ}{\hat{\Lambda}_{\mathcal N+1}^{J}}

\newcommand{\LNNJ}{\Lambda_{\mathcal N+1}^{J}}

\newcommand{\I}{{\large \hat{{\Huge {\Gamma}}}}(K)}
\newcommand{\IR}{{\large \hat{\Huge {\Gamma^R}}}}

\newcommand{\AM}{\Phi^{\mathcal M}(K)}
\newcommand{\AN}{\Phi^{\mathcal N}(K)}

\newcommand{\ANtJ}{\hat{J}^{\mathcal N}(K)}

\newcommand{\ANNtJ}{\hat{J}^{\mathcal N+1}(K)}

\newcommand{\AaJ}{\hat{J}(K)}

\newcommand{\IJJ}{{\large \hat{{\Huge {\Gamma}}}^{J}}(K)}
\newcommand{\IW}{{\large \hat{{\Huge {\Gamma}}}^{W}}(K)}

\newcommand{\AMJ}{J_{\mathcal M}(K)}
\newcommand{\ANJ}{J_{\mathcal N}(K)}

\newcommand{\CCC}{\ddot{\C}}

\newcommand{\usR}{\psi^{R}_{\mathcal N}}

\newcommand{\suJ}{\hat{\psi}^J}

\newcommand{\usn}{\psi^{u}_{\mathcal N}}
\newcommand{\usnJ}{\psi^{u,J}_{\mathcal N}}

\newcommand{\usnnJ}{\psi^{u,J}_{\mathcal N+1}}

\newcommand{\usmJ}{\psi^{u,J}_{\mathcal M}}

\newcommand{\snJ}{\psi^{J}_{\mathcal N}}

\newcommand{\smJ}{\psi^{J}_{\mathcal M}}

\newcommand{\sntJ}{\tilde{\psi}^{J}_{\mathcal N}}

\newcommand{\snntJ}{\tilde{\psi}^{J}_{\mathcal N+1}}

\newcommand{\snbJ}{\bar{\psi}^{\mathcal N,J}_{\mathcal N}}

\newcommand{\smnbJt}{\hat{\psi}^{\mathcal N,J}_{\mathcal M}}
\newcommand{\snbJt}{\hat{\psi}^{\mathcal N,J}_{\mathcal N}}

\newcommand{\snbbbJ}{\bar{\psi}^{J}_{\mathcal N+1}}

\newcommand{\snbb}{\bar{\bar{\psi}}_{\mathcal N}}
\newcommand{\snbbJ}{\bar{\bar{\psi}}^{J}_{\mathcal N}}

\newcommand{\snnbbJ}{\bar{\bar{\psi}}^{J}_{\mathcal N+1}}

\newcommand{\PPNJ}{\text{pr}_{\mathcal{N}}^{J}}

\newcommand{\PMNJ}{p^{\mathcal{N},J}_{\mathcal{M}}}

\newcommand{\PMNNJ}{p^{\mathcal{N},J}_{\mathcal{M}}}

\newcommand{\PNNJ}{p^{\mathcal{N},J}_{\mathcal{N}}}

\newcommand{\PhM}{\hat{p}_{\mathcal{M}}}

\newcommand{\PhNJ}{\hat{p}_{\mathcal{N}}^{J}}

\newcommand{\N}{\mathbb N}
\newcommand{\Z}{\mathbb Z}

\newcommand{\Imm}{\mathrm{Im}}

\newcommand{\Conf}{\mathrm{Conf}}

\definecolor{dgreen}{RGB}{0,150,0}

\newcommand{\incl}[3][right]%
{%
\draw[<-,>=#1 hook] #2 to ($ #2!0.5!#3 $);
\draw[->,>=stealth'] ($ #2!0.5!#3 $) to #3;%
}
\newcommand{\inclusion}[5][right]%
{%
\draw[<-,>=#1 hook] #4 to ($ #4!0.5!#5 $) node[#2,font=\small]{#3};
\draw[->,>=stealth'] ($ #4!0.5!#5 $) to #5;%
}

{\begin{compactitem}

}%
{\end{compactitem}}
{\begin{compactitem}[#1]

}%
{\end{compactitem}}
{\begin{compactdesc}

}%
{\end{compactdesc}}

\newcommand{\bl}{\bar{l}}

\newcommand{\cD}{\mathcal{D}}

\newcommand{\cM}{\mathcal{M}}
\newcommand{\cN}{\mathcal{N}}

\newcommand{\C}{\mathbb{C}}

\newcommand{\Clm}{C_{n,m}^{\bl}}
\newcommand{\ClmN}{C_{n,m}^{\bl,\bar{\cN}}}

\newcommand{\ccN}{N^C}
\newcommand{\ccl}{\lambda^C}

\newcommand{\FJ}{\mathscr F_{\bar{i},\bar{N}}}
\newcommand{\LJ}{\mathscr L_{\bar{i},\bar{N}}}
\newcommand{\HJ}{H^{\bar{\cN}(\bar{N})}_{n,2+\sum_{i=2}^{n} (\ccN_i-1)}}
\newcommand{\HJd}{H^{\bar{\cN}(\bar{N}),\partial}_{n,2+\sum_{i=2}^{n} (\ccN_i-1)}}
\newcommand{\HJi}{H^{\bar{\cN}(\bar{N})}_{n,m_J(\bar{N})}}
\newcommand{\HJdi}{H^{\bar{\cN}(\bar{N}),\partial}_{n,m_J(\bar{N})}}

\newcommand{\FJJ}{\mathscr F'_{\bar{i},\bar{N}}}
\newcommand{\LJJ}{\mathscr L'_{\bar{i},\bar{N}}}
\newcommand{\HJJ}{H^{\bar{\cN}(\bar{N})}_{n,1+\sum_{i=2}^{n} (\ccN_i-1)}}
\newcommand{\HJJd}{H^{\bar{\cN}(\bar{N}),\partial}_{n,1+\sum_{i=2}^{n} (\ccN_i-1)}}

\newcommand{\FJJJ}{\mathscr F''_{\bar{i},\bar{N}}}
\newcommand{\LJJJ}{\mathscr L''_{\bar{i},\bar{N}}}
\newcommand{\HJJJ}{H^{\bar{\cN}(\bar{N})}_{n,\sum_{i=2}^{n} (\ccN_i-1)}}
\newcommand{\HJJJd}{H^{\bar{\cN}(\bar{N}),\partial}_{n,\sum_{i=2}^{n} (\ccN_i-1)}}

\newcommand{\sJ}{\psi_{q,\bar{N},\overline{[N]}}}
\newcommand{\sA}{\psi_{\xi_{\cN},\bar{\lambda},\overline{d(\lambda)}}}
\newcommand{\sJt}{\psi_{\bar{N}}}
\newcommand{\sAt}{\psi^{\bar{\lambda}}_{\cN}}
\newcommand{\sJtt}{\psi^{1}_{q,\bar{N},\overline{[N]}}}
\newcommand{\sAtt}{\psi^{1-\cN}_{\xi_{\cN},\bar{\lambda},\overline{d(\lambda)}}}
\newcommand{\sAtu}{\psi^{\bar{\lambda}}_{\cM}}
\newcommand{\sAtuJ}{\psi^{J}_{\cM}}

\newcommand{\HAi}{H^{\bar{\cN}}_{n,m_A(\cN)}}
\newcommand{\HAdi}{H^{\bar{\cN},\partial}_{n,m_A(\cN)}}

\newcommand{\HA}{H^{\bar{\cN}}_{n,2+(n-1)(\cN-1)}}
\newcommand{\HAd}{H^{\bar{\cN},\partial}_{n,2+(n-1)(\cN-1)}}

\newcommand{\FAA}{\mathscr F'_{\bar{i},\cN}}
\newcommand{\LAA}{\mathscr L'_{\bar{i},\cN}}
\newcommand{\LAV}{\mathscr F^{c}_{\bar{i},\cN}}
\newcommand{\HAA}{H^{\bar{\cN}}_{n,1+(n-1)(\cN-1)}}
\newcommand{\HAAd}{H^{\bar{\cN},\partial}_{n,1+(n-1)(\cN-1)}}

\newcommand{\FAAA}{\mathscr F''_{\bar{i},\cN}}
\newcommand{\LAAA}{\mathscr L''_{\bar{i},\cN}}
\newcommand{\HAAA}{H^{\bar{\cN}}_{n,(n-1)(\cN-1)}}
\newcommand{\HAAAd}{H^{\bar{\cN},\partial}_{n,(n-1)(\cN-1)}}

\renewcommand{\geq}{\geqslant}
\renewcommand{\leq}{\leqslant}
\renewcommand{\footnoterule}{%
  \kern -3pt
  \hrule width \textwidth height 0.4pt
  \kern 2.6pt
}

\definecolor{dgreen}{RGB}{0,150,0}

\begin{document}
\title{\vspace{-14mm} \Large\bfseries Geometric universal Jones invariant from configurations on ovals in the disc }
\author{ \small Cristina Anghel \quad $/\!\!/$\quad \todaysdate\vspace{-3ex}}
\date{}
\maketitle
{
\makeatletter
\renewcommand*{\BHFN@OldMakefntext}{}
\makeatother
}
\vspace{-7mm}
\begin{abstract}

We construct geometrically a {\bf \em universal Jones invariant} as a limit of {invariants given by graded intersections in configuration spaces}. For any fixed level $\cN$, we define a new knot invariant, called {\em``$\cN^{th}$ Unified Jones invariant''} globalising topologically all coloured Jones polynomials at levels less than $\cN$. It is defined via the intersection points between {Lagrangian submanifolds} supported on {\em arcs and ovals} in the disc. The geometry of these Lagrangians is novel: previous topological models involved immersed submanifolds rather than embedded ones. We do this by defining a new local system that refines the Lawrence representation, and depends of the distribution of multiplicities of points in the configuration space on the ovals.


On the algebraic side, Habiro's famous invariant for knots \cite{H3} is a universal invariant globalising the family of coloured Jones polynomials. He conjectured that this universal invariant recovers also the ADO invariant divided by the Alexander polynomials, which was proved by Willetts \cite{W} in a version of Habiro's ring (\cite{H2}). The universal Jones invariant that we construct belongs to a different ring that comes with a map to Habiro's ring \cite{H2}. 

We prove that our invariant recovers this version of Habiro's invariant. The difference is that our invariant is given as a limit of new knot invariants, the $\cN^{th}$ unified Jones invariants. These invariants in turn provide a geometrical understanding of sets of all coloured Jones polynomials of bounded colour, collecting more information as we increase the colour. 

 
 \end{abstract}
 
 \vspace{-1mm}

\section{Introduction}\label{introduction}

\vspace{-3mm} Coloured Jones and coloured Alexander polynomials are quantum link invariants coming from representation theory (\cite{J}, \cite{RT}, \cite{Witt}). The geometry encoded by these invariants is an important open problem in quantum topology. Physics predicts that their asymptotics encode rich geometrical information such as the Volume Conjecture (Kashaev \cite{K}, \cite{M2}) and Gukov-Manolescu's Conjecture (\cite{GM}). At the asymptotic level, Habiro (\cite{H3}) introduced his celebrated universal invariant for knots which is power series that unifies and recovers all coloured Jones polynomials of a knot.
The story for quantum link invariants is completely different. Until now, there were no universal invariants defined for the link case. Our aim is to study quantum invariants from a new topological perspective, given by `` topological models'', which are graded intersections in configuration spaces. The aim of this paper is two sided: first, we construct a universal Jones invariant of a purely topological nature. Secondly, we define the first topological models for quantum link invariants, providing a set-up for a sequel article \cite{CrI} where we answer the open problem of constructing universal geometrical link invariants.

{\bf Problem:} {\bf \em Unification of all coloured Jones polynomials  of a bounded level}

 For a fixed $\cN$, can one construct an invariant that recovers all  coloured Jones polynomials of colour bounded by $\cN$? Can we do this in a purely topological manner? We will answer this problem.
\vspace{-1mm}
\begin{thm}[{\bf\em  New level $\cN$ invariants}] For each $\cN$, we construct a knot invariant $\PAAJ$ via Lagrangian intersections in a fixed configuration space in the disc. We call $\PAAJ$ the \\ {\color{blue} $\cN^{th}$ unified Jones invariant} and prove that it unifies all coloured Jones polynomials up to level $\cN$:
\begin{equation}
 \PAAJ|_{\smnbJt}=J_{\cM}(K), \ \ \forall \cM \leq \cN.
\end{equation}
\end{thm}
Then we show that the sequence of $\cN^{th}$ unified Jones invariants has a good geometrical behaviour with respect to the change of colour and we globalise then in a {\em universal Jones invariant}, as below.

\vspace{-1mm}
\begin{thm}[\bf\em Universal geometrical Jones invariant]\label{THEOREMUnivJ}
We define a knot invariant $\IJJ$ taking values in a completed ring  $\LLhJ$ that recovers all coloured Jones polynomials. 
This invariant $\IJJ$ is a limit of the {unified Jones invariants} defined in configuration spaces:
\begin{equation}
\IJJ:= \underset{\longleftarrow}{\mathrm{lim}} \ \PAAJ \in \LLhJ.
\end{equation}
\end{thm}
The construction of the universal invariant is done in three main steps, as below.
\begin{enumerate}
\item  We provide the {\em  first topological model} for {\em coloured Alexander polynomials and coloured Jones polynomials for coloured links} (Theorem \ref{THEOREMA})
\item Then we construct  new geometric knot invariants at level $\cN$:\\ {\color{blue} \em $\cN^{th}$ unified Jones invariant} (Theorem \ref{INV})
\item Then we obtain the  {\em Geometric Universal knot invariant from configurations on ovals:}\\  {\color{blue} \em Universal geometrical Jones invariant} (Theorem  \ref{THJ}).
\end{enumerate}
{\bf Construction} The geometrical part of the construction is done in four main steps. First, for a fixed level $\cN$ and a braid $\beta_n$, we define a state sum of Lagrangian intersections: $$\PA \in \Z[u^{\pm1},x^{\pm1}, d^{\pm1}]$$ in the {configuration space of $(n-1)(\cN-1)+2$} points in the  punctured disc. This is given by the set of Lagrangian intersections $\{\langle (\beta_{n} \cup {\mathbb I}_{n+2} ) \ { \color{red} \FA}, {\color{dgreen} \LA} \rangle\}$ from Figure \ref{ColouredAlex}. We prove the following. 
 
\begin{thm}[\bf \em Interpolating all coloured Jones polynomials of bounded level and the $\cN^{th}$ ADO invariant] For any fixed level $\cN \in \N$ and $K=\hat{\beta_n}$ a knot seen as a closure of a braid $\beta_n$, the graded intersection $\PA$ recovers all coloured Jones polynomials up to level $\cN$ and also the $\cN^{th}$ coloured Alexander polynomial via specialisation of coefficients (as in Figure \ref{FigIntro}).
\end{thm}
\vspace{-2.8mm}
\begin{figure}[H]
\begin{center}
\hspace*{-4mm}\begin{tikzpicture}
[x=1.2mm,y=1.4mm]
\node (Jl)               at (-61,10)    {$ \LLNJ \ni \PAAJ$};

\node (Jn)               at (-40,-2)    {$J_{\cN}(K)\ \ \ \cdots$};
\node (Jn-1)               at (-65,-2)    {$J_{\cN-1}(K)$};
\node (Jm)               at (-85,-2)    {$J_{2}(K)$};


\node (An)               at (20,-2)    {${\Phi^{\cN}}(K) \ \ \cdots$};
\node (An-1)               at (2,-2)    {$\Phi^{\cN-1}(K)$};
\node (Am)               at (-17,-2)    {${\Phi^{2}}(K) $};

\node(IJ)[draw,rectangle,inner sep=3pt,color=red] at (-78,15) {$\cN^{th}$ unified Jones invariant};

\node(LI)[draw,rectangle,inner sep=3pt,color=blue, minimum height = 2cm] at (-27,18) {\begin{tabular}{l}
 \ \ \ Lagrangian intersection\\    ${\color{blue} \ \text{in} \ \Conf_{2+(n-1)(\cN-1)}\left(\mathbb D\right)}$ \\ $\phantom{!}$ \\ $\color{black} \PA \in \Z[u^{\pm1},x^{\pm1}, d^{\pm1}]$
\end{tabular}};

\node(UJ)[draw,rectangle,inner sep=3pt,color=red] at (-60,35) { \ \ Universal Coloured Jones invariant \ \ };

\node(Q)[draw,rectangle,inner sep=3pt,color=red, minimum width = 6.4cm, 
    minimum height = 1cm] at (-65,-3) {$\phantom{A}$};
\node(Q)[draw,rectangle,inner sep=3pt,color=dgreen, minimum width = 1.5cm, 
    minimum height = 1cm] at (17,-3) {$\phantom{A}$};

\node (J)               at (-47,30)    {$\IJJ \in \LLhJ$};

\draw[->]             (Jl)      to node[right,xshift=2mm,font=\large]{}   (Jn);
\draw[->]             (Jl)      to node[right,xshift=2mm,font=\large]{}   (Jn-1);
\draw[->]             (Jl)      to node[right,xshift=2mm,font=\large]{}   (Jm);

\draw[->]             (J)      to node[right,xshift=2mm,font=\large]{}   (Jl);
\draw[->]             (LI)      to node[right,xshift=2mm,font=\large]{}   (An);



\node(C) at (-70,-10) {\color{red} Coloured Jones polynomials };
\node(C) at (-70,-13) {\color{red} with colors bounded by $\cN$};
\node(C) at (-70,-10) {\color{red} Coloured Jones polynomials };
\node(C) at (5,-10) {\color{dgreen} $\cN^{th}$ Coloured Alexander polynomial};

\draw[->,dashed, red]             (IJ)      to node[right,xshift=2mm,font=\large]{$\underset{\longleftarrow}{\mathrm{lim}}$}   (UJ);

\draw[->, blue]             (LI)      to node[right,xshift=2mm,font=\large]{$ $}   (Jl);

\end{tikzpicture}
\end{center}
\vspace{-4mm}
\caption{Geometrical universal Jones invariant for knots}\label{FigIntro}
\end{figure}
\vspace{-4mm}
{\bf Universal ring} For the second step, we look at the set of coloured Jones polynomials. In this case we identify $u$ and $x$, and consider $\LL=\Z[x^{\pm 1},d^{\pm 1}]$. We define a sequence of nested ideals in $\LL$, given by:
$\tilde{I}_{\cN}^{J}:=\langle \prod_{i=1}^{\cN} (xd^{i-1}-1) \rangle \subseteq \LL$.
 We obtain a sequence of quotient rings, with maps between them and their projective limit $\LLhJ$:
\begin{equation}
\begin{aligned}
&\hspace{13mm} l^J_{\cN} \hspace{11mm} l^J_{\cN+1}\\
&\cdots \LLNJ \longleftarrow \LLNNJ \longleftarrow \cdots \ \ \ \ \ \ \ \ \ \ \ \ \ \ \LLhJ:= \underset{\longleftarrow}{ \mathrm{lim}} \ \LLNJ.
\end{aligned}
\end{equation}
{\bf Unified invariant} In the third step, we define the $\cN^{th}$ Unified Jones invariant $\PAAJ$ as the image of $\PA$ in the quotient ring $\LLNJ$. It is a well-defined knot invariant recovering all coloured Jones polynomials up to level $\cN$ (see Theorem \ref{INV}).

{\bf Universal Jones invariant} Finally, in the fourth step we prove that we can pass to the limit, and define the universal geometrical Jones invariant
\begin{equation}
\begin{aligned}
&\IJJ\in \LLhJ = \underset{\longleftarrow}{\mathrm{lim}} \ \Z[x^{\pm1}, d^{\pm 1}] /\langle \prod_{i=1}^{\cN} (xd^{i-1}-1)\\
&\IJJ:= \underset{\longleftarrow}{\mathrm{lim}} \ \PAAJ.
\end{aligned}
\end{equation}

\subsection{Main Results}

As we have seen, the story for quantum link invariants is different to the knot case. Such coloured versions of link invariants are building blocks for $3-$manifold invariants: the Witten-Reshetikhin-Turaev (\cite{Witt}) and the Costantino-Geer-Patureau invariants (\cite{CGP14}). These are powerful invariants which detect lens spaces (\cite{BCGP}). For knot invariants, topological models appeared in \cite{Big}, \cite{Law},  \cite{Cr3}, \cite{Crsym}, \cite{Cr2}. In \cite{Cr2}, we provided a topological model for coloured Jones polynomials via configurations on {\em figure eights} in the disc. Then, in \cite{WRT} we proved the {\em first topological model} for $3$-manifold invariants, for Witten-Reshetikhin-Turaev's invariant, providing a new framework for the study of these invariants for which a complete $3$-dimensional topological description and categorifications are open questions. As a parallel, there are no topological models known for the Costantino-Geer-Patureau invariants.   Motivated by Habiro's unifications and problems regarding the topology of non-semisimple $3$-manifold invariants we pose four questions.

\begin{itemize}
\item{\bf Question 1} Construct {\em topological models} for coloured Jones and coloured Alexander invariants for {\em coloured links rather than for knots}.\\ Up to this moment no such model was known for the coloured Alexander invariants. 
\item{\bf Question 2} Having in mind Floer type categorifications, can we provide topological models with {\em embedded topological supports}, given by {\em ovals rather than figure eights}? 
\item{\bf Question 3} Is it possible to construct {\em universal invariants for links rather than for knots}?  
\item{\bf Question 4} Are there {\em topological models for CGP invariants}? 
\end{itemize}
 In this paper we answer the first two questions. The answer to {\em Question 2} makes possible the definition of our geometrical universal Jones invariant.
The response to {\em Question 1} gives the first topological model for non-semisimple link invariants. This is the building block for two sequels. In the first sequel article we answer {\em Question 3} providing universal invariants for links. The other paper answers {\em Question 4} on topological models for non-semisimple $3$-manifold invariants.
\vspace*{-2mm}
\begin{figure}[H]
\begin{equation*}
\begin{aligned}
&\hspace{20mm}\LLhJ= \underset{\longleftarrow}{\mathrm{lim}} \ \LLNJ  \\
& \LLNJ= \Z[x^{\pm1}, d^{\pm 1}] /\langle \prod_{i=1}^{\cN} (xd^{i-1}-1) \rangle.
\end{aligned}
\end{equation*}
\begin{center}
\hspace*{-30mm}\begin{tikzpicture}
[x=1.2mm,y=1.1mm]
\node (Jl)               at (-80,10)    {$ \LLNJ \ni \PAAJ$};

\node (Wn)               at (-85,-30)    {$\LLNW$};

\node (Jn)               at (-70,0)    {$J_{\cN}(K)$};
\node (Jn-1)               at (-70,-10)    {$J_{\cN-1}(K)$};
\node (Jm)               at (-70,-20)    {$J_{\cM}(K)$};

\node (A'n)               at (-20,0)    {$\frac{\AN}{\Delta(x^{2\cN})}$};
\node (A'n-1)               at (-20,-10)    {$\frac{\Phi^{\cN-1}(K)}{\Delta(x^{2(\cN-1})}$};
\node (A'm)               at (-20,-20)    {$\frac{\AM}{\Delta(x^{2\cM})}$};

\node(IJ)[draw,rectangle,inner sep=3pt,color=red] at (-85,15) { \ \ $\cN^{th}$ unified Jones invariant \ \ };

\node(IJI)[draw,rectangle,inner sep=3pt,color=red] at (-90,-5) {\begin{tabular}{l} Coloured Jones \\ \ \ polynomials \end{tabular}};

\node(IJI)[draw,rectangle,inner sep=3pt,color=dgreen,text width=3.6cm ] at (3,-5) { \ \ \  \ \ \ \ ADO over \ \ \ \\  Alexander polynomials};

\node(UJ)[draw,rectangle,inner sep=3pt,color=red] at (-55,30) { \ \ Universal Geometrical Jones invariant \ \ };
\node[draw,rectangle,inner sep=3pt,color=blue] at (-52,-35) { \ Habiro universal invariant \cite{H2},\cite{W} \ };

\node (J)               at (-50,25)    {$\IJJ \in \LLhJ$};
\node (W)               at (-50,-30)    {$\IW \in \LLhW$};

\draw[->,dashed]             (Wn)      to node[right,xshift=2mm,font=\large]{}   (Jn);
\draw[->,dashed]             (Wn)      to node[right,xshift=2mm,font=\large]{}   (Jn-1);
\draw[->,dashed]             (Wn)      to node[below right,xshift=-1mm,font=\Large]{\color{blue}$\nexists$}   (Jm);

\draw[->]             (Jl)      to node[right,xshift=2mm,font=\large]{}   (Jn);
\draw[->]             (Jl)      to node[right,xshift=2mm,font=\large]{}   (Jn-1);
\draw[->]             (Jl)      to node[right,xshift=2mm,font=\large]{}   (Jm);

\draw[->]             (J)      to node[right,xshift=2mm,font=\large]{}   (Jn);
\draw[->]             (W)      to node[right,xshift=2mm,font=\large]{}   (Jn);

\draw[->]             (J)      to node[right,xshift=2mm,font=\large]{}   (A'n);
\draw[->]             (W)      to node[right,xshift=2mm,font=\large]{}   (A'n);


\draw[->]             (J)      to node[right,xshift=2mm,font=\large]{$\pi$}   (W);
\draw[->,dashed, red]             (IJ)      to node[right,xshift=-12mm,font=\large]{$\underset{\longleftarrow}{\mathrm{lim}}$}   (UJ);

\end{tikzpicture}
\hspace*{-20mm}\begin{equation}\label{WR}
\begin{aligned}
& \LLhW=\underset{\longleftarrow}{ \mathrm{lim }} \ \LLNW  \\ 
&\hspace{-40mm}\LLNW= \Z[x^{\pm1}, d^{\pm 1}] / \langle \prod_{i=1}^{k-1} (xd^{i-1}-1) \prod_{i=k-1}^{\cN-1} (d^i-1) \mid 1\leq k \leq \cN-1 \rangle.\\
\end{aligned}
\end{equation}
\end{center}
\vspace{-5mm}
\caption{\hspace*{-2mm} The universal Jones invariant globalises sequences of coloured Jones invariants}\label{re}
\end{figure}

\subsection{Recovering Habiro's invariant \cite{H2}, \cite{W}} 
In this part we analyse similarities and differences between our invariant and the Habiro's universal invariant seen in the universal ring from \cite{H2}.
In \cite{H3}, Habiro defined a two-variable power series constructed in his Universal Ring for knots, which unifies all coloured Jones polynomials, and conjectured that at roots of unity this recovers coloured Alexander polynomials over the Alexander polynomials. 
In \cite{W}, Willetts studied the image of Habiro's universal invariant in the ring defined by Habiro in \cite{H2}. More specifically, Willetts used the sequence of ideals and associated quotient rings $\LLNW$, as in relation \eqref{WR}. They give in the limit a universal ring  $\LLhW:= \underset{\longleftarrow}{ \mathrm{lim}} \ \LLNW.$ In this manner, Willetts used Habiro Universal invariant $\IW$ seen in the ring $\LLhW$, and proved that it recovers coloured Jones polynomials through generic specialisations of coefficients and the ADO invariants divided by the Alexander polynomial through specialisations at roots of unity (as stated in Theorem \ref{Will}). However, this invariant is constructed directly in the limit ring: $\IW \in \LLhW$ and it comes from Lawrence's universal construction.

We show that our universal invariant recovers Habiro's invariant (Theorem \ref{RELATION}).
\begin{thm}[\textbf{Projection onto Habiro's Universal invariant}] There is a natural map between the universal rings $\LLhJ$ and $\LLhW$ which sends our universal Jones invariant to Habiro's universal invariant:
$$\pi: \LLhJ \rightarrow \  \LLhW$$
\begin{equation}
\pi(\IJJ)=\IW. 
\end{equation} 
\end{thm}
As a consequence, we obtain the following.
\begin{thm}[\textbf{Geometric unification of coloured Jones and Alexander invariants }\cite{W}] The universal geometrical Jones invariant that we construct recovers coloured Jones polynomials and the ADO invariants divided by the Alexander polynomial:
\begin{equation}
\begin{aligned}
&\IJJ \Bigm| _{d=\xi_{\cN}^{-1}}  =~ \frac{\Phi^{\cN}(K,x)}{\Delta(K,x^{2\cN})}\\
&\IJJ \Bigm| _{x=d^{\cN-1}}  =~ J_{\cN}(K). 
\end{aligned}
\end{equation} 
\end{thm}

\subsubsection{Differences: the geometrical Jones invariant and Habiro's invariant} In this part, we investigate the components of the two universal invariants that we have: 
$$\IJJ:= \underset{\longleftarrow}{\mathrm{lim}} \ \PAAJ \in \LLhJ \ \ \ \longrightarrow \ \ \ \IW:= \underset{\longleftarrow}{\mathrm{lim}} \ \PAAW \in \LLhW.$$
Structurally, these two universal invariants come from totally different perspectives. The geometric universal Jones invariant is a limit of invariants that see more and more coloured Jones polynomials:
\begin{equation}
\PAAJ \Bigm| _{\color{red}x=d^{\cM-1}}  =~ {\color{red} J_{\cM}(K)}, \ \  \forall \cM\leq \cN. 
\end{equation}
 
On the other hand, the structure of the Habiro invariant is different, being naturally constructed in the limit ring. In Section \ref{S:rel}, we discuss the structure of the ring used by Willetts and we investigate its components. Since the ideals that he considered are much larger than ours, its intermediary rings are smaller. Because of this, the specialisations at natural numbers or roots of unity are impossible from the $\cN^{th}$ component of Habiro's invariant $\PAAW$:
\begin{equation}
\begin{aligned}
&\PAAW \Bigm| _{\color{blue}{x=d^{\cM-1}}} \text{not well-defined}, \\ 
&\PAAW \Bigm| _{\color{blue}{d=\xi_{\cM}^{-1}}} \text{     not well-defined}, \ \ \ \ \ \forall \cM\leq \cN.
\end{aligned}
\end{equation}

\begin{rmk}[Unifying intermediary levels of the coloured Jones sequence]
So the question of having a knot invariant at the intermediary levels used by Willetts that sees either coloured Jones invariants or coloured Alexander invariants is impossible even to state. On the other hand, in our ring, which is larger, we are able to globalise more and more coloured Jones polynomials and in the limit we recover also the ADO invariant over the Alexander polynomial, as summarised in Figure \ref{re}.

\end{rmk}

\subsection{Future directions}
\subsubsection{Geometry of the universal Jones invariant via infinite configuration~spaces} \label{asym}\phantom{A}\\
As we have seen, the universal Jones invariant $\IJJ$ is the limit of the $\cN^{th}$ unified Jones invariants $\{  \PAAJ \in \LLNJ \}$.
For a fixed colour $\cN$, $\PAAJ$ is given by graded intersections in the configuration space of $(n-1)(\cN-1)+2$ points in the disc :
$$\PAAJ \longleftrightarrow \lbrace (\beta_{n} \cup {\mathbb I}_{n+2} ) \ { \color{red} \FA} \cap {\color{dgreen} \LA} \rbrace, \text{ for } {\bar{i} \in \{\bar{0},\dots,\overline{\cN-1}\}}.$$

\vspace{-8mm}

\begin{figure}[H]
\begin{equation}
\centering
\begin{split}
\begin{tikzpicture}
[x=1mm,y=1.6mm]
\node (tl) at (50,68) {$\LLhJ$};
\node (ml) at (50,35) {$\LLNNJ$};
\node (mr) at (80,35) {$\LNNJ$};
\node (bl) at (50,-10) {$\LLNJ$};
\node (br) at (80,-10) {$\LNJ$};
\node (mlei) at (ml.north east) [anchor=south west,color=black] {\phantom{A}\vspace*{20mm}\hspace{-110mm}${\bar{i} \in \{\bar{0},\dots,\overline{\cN}\}}: \lbrace (\beta_{n} \cup {\mathbb I}_{n+2} ) \ { \color{red} \FAnn} \cap {\color{dgreen} \LAnn} \rbrace$};
\node (blei) at (bl.north east) [anchor=south west,color=black] {\phantom{A}\vspace*{20mm}\hspace{-105mm}${\bar{i} \in \{\bar{0},\dots,\overline{\cN-1}\}}: \lbrace (\beta_{n} \cup {\mathbb I}_{n+2} ) \ { \color{red} \FA} \cap {\color{dgreen} \LA} \rbrace$};
\node (tle) at (tl.south) [anchor=north,color=red] {$\IJJ$};
\node (mle) at (ml.north) [anchor=south,color=red] {$\PAANJ$};
\node (ble) at (bl.north) [anchor=south,color=red] {$\PAAJ$};
\node (mre) at (mr.north west) [anchor=south west,color=blue] {$J_{\cN +1}(K)$};
\node (bre) at (br.north west) [anchor=south west,color=blue] {$J_{\cN}(K)$};
\draw[->] (mle) to node[below,font=\small]{} (mre);
\draw[->,dashed] (ble) to node[below,font=\small,yshift=7mm]{{\hspace{-50mm}\color{dgreen}Add one particle on each circle}} (ml);
\draw[->,dashed] (ble) to node[below,font=\small,yshift=-26mm]{{\hspace{-50mm}\color{carmine}Add arcs to the right punctures}} (ml);
\draw[->,dashed] (ble) to node[below,font=\small,yshift=17mm]{\hspace{40mm}Level $\cN+1$} (ml);
\draw[->,dashed] (ble) to node[below,font=\small,yshift=-16mm]{\hspace{40mm}Level $\cN$} (ml);
\draw[->] (ble) to node[below,font=\small]{} (bre);
\draw[->] (tle) to node[above,yshift=5mm,font=\small]{$\usnnJ$} (mre);
\node[draw,rectangle,inner sep=3pt] at (-5,69) {Limit};
\node[draw,rectangle,inner sep=3pt,color=red] at (-5,65) {Universal Jones invariant};
\end{tikzpicture}
\end{split}
\end{equation}

\end{figure}
\vspace{-130mm}

\begin{figure}[H]
\centering

{\phantom{A} \ \ \ \ \ \ \ \ \ \ \ \ \ \ \ \ \ \ \ \ \ \ \ \ \ \ \ \ \ \ \ \  \ \ \ \ \ \ \ \ \ } \hspace*{-140mm}\includegraphics[scale=0.35]{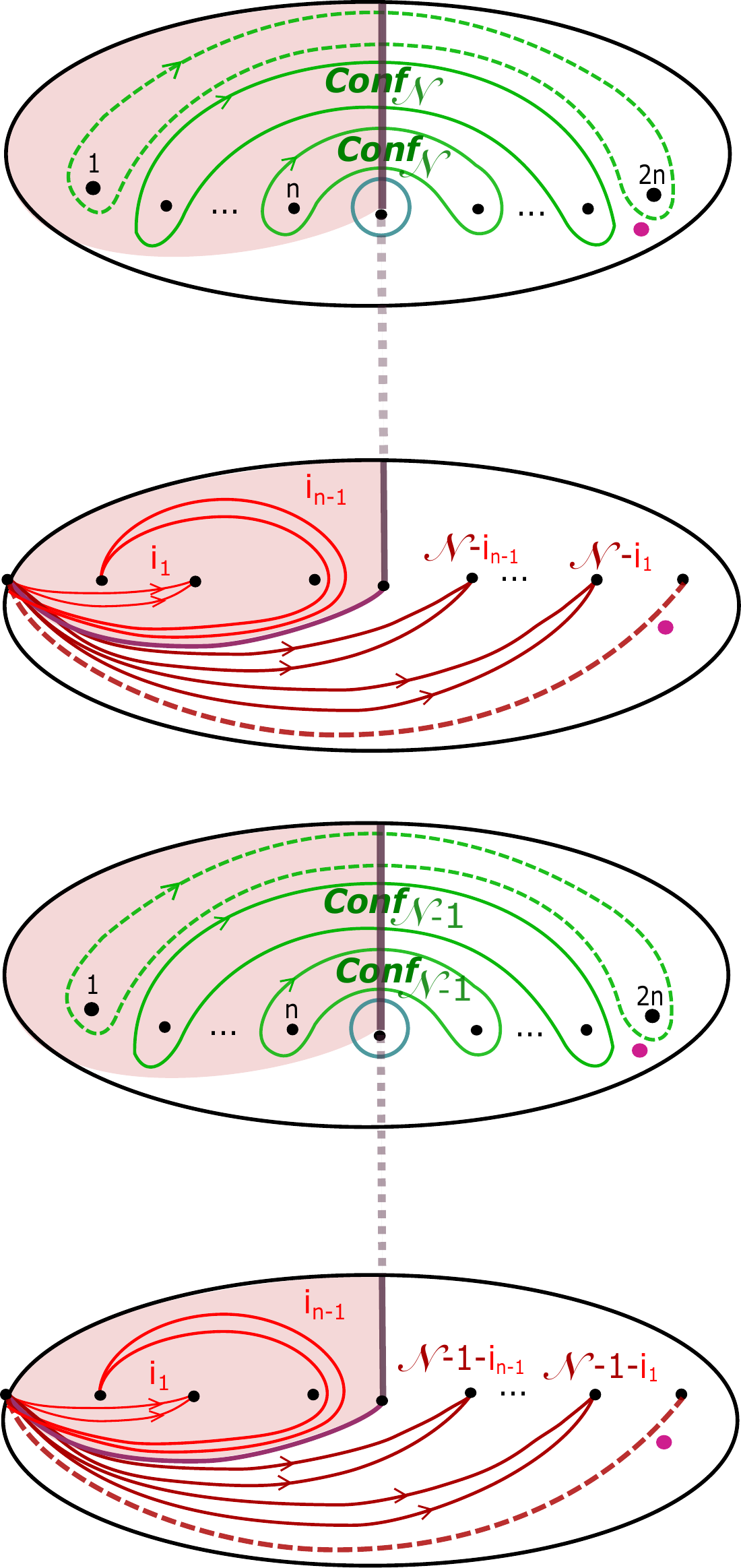}
\vspace{-1mm}
\caption{\normalsize Universal Jones invariant as the limit of intersections in configurations}
\label{ColouredAlex}
\end{figure}
\vspace{-3mm}

Passing from level $\cN$ to level $\cN+1$, we have an explicitly understandable  behaviour of the homology classes, as shown in Figure \ref{ColouredAlex}. 
The above argument suggests that the behaviour of the knot invariants $\PAA$ with respect to the change of level has a stability phenomenon. 
Based on this we are interested in defining the universal Jones invariant $\IJJ$ directly as an intersection in an infinite configuration space. 
\clearpage
\vspace{-3mm}
\subsubsection{Sequel paper I: universal link invariants} 

This article is paired with a sequel \cite{CrI}. 

{\bf \em Open problem:}
Can one construct a unification of coloured Alexander invariants (which are non semi-simple) for coloured links? Our second paper provides an answer to this open problem.\\
 The parallel question about semi-simple invariants for knots is the subject of Habiro’s famous universal knot invariant, for which we presented a geometric perspective in the current paper.\\
 In \cite{CrI} we construct geometrically two {\bf universal link invariants}, as limits of invariants defined in configuration spaces. 
 More specifically, we define a universal ADO link invariant as a limit of invariants given by graded intersections in configuration spaces. Moreover, we construct a universal Jones invariant for links, which is a limit of invariants globalising all coloured Jones link invariants at bounded levels. These models uses in an essential manner the geometrical set-up from the current paper: the topological model for coloured Alexander link invariants from Theorem \ref{THEOREMA}. However, it contains new geometric data which permits globalisations for the link case. 
  \vspace{-4mm}
\subsubsection{Sequel paper II: Topological models for $WRT$ and $CGP$ invariants} 
An active research area studies categorifications of link invariants (\cite{K,Ras,OZ,SM,M1}). On the other hand, the existence of categorifications for $WRT$ and $CGP$ invariants are open questions.
For this purpose, a topological model with embedded Lagrangians would bring a new approach to these problems. 
In a different direction, passing from knot to $3$-manifold invariants, Habiro \cite{H3} showed a beautiful unification of all $WRT$ invariants.

{\bf \em Problem-} Is it possible to unify the $CGP$ invariants, as a parallel to Habiro's celebrated unification of the WRT invariants (\cite{H3})? As a step towards this, we shall prove the following models.

Using the constructions from this paper for coloured Jones and Alexander invariants for links, we pass to $3$-manifolds. Our sequel result describes both $WRT$ and $CGP$ invariants from a set of Lagrangian submanifolds in a fixed configuration space. The geometric support of these Lagrangians will be based on ovals, which is important for the geometric study of these invariants. 
\vspace{-8mm}
\subsection*{Structure of the paper} This paper has seven parts. In Section \ref{S3} we define homological tools: local systems and homology groups. Then, in Section \ref{S:4} we show that for good choices of covering spaces, we can lift submanifolds supported on ovals. Section \ref{S:J} has the topological model for coloured Jones polynomials. Section \ref{S:A} is dedicated to the topological model for coloured Alexander polynomials for links. Then, in Section \ref{SAu} we show the globalisation property of  the models for coloured Jones polynomials at bounded level. Based on this, we define our unified Jones invariants in Section \ref{S:UN}. In Section \ref{univ} we construct the universal Jones invariant $\IJJ$, a knot invariant recovering all coloured Jones polynomials. In Section \ref{S:rel} we show that our invariant recovers Habiro's invariant by a change of coefficients and discuss the relation between these two constructions.
\vspace{-4mm} 
\subsection*{Acknowledgements} 
 I acknowledge the support of SwissMAP, a National Centre of Competence in Research funded by the Swiss National Science Foundation as well as the Romanian Ministry of Education and Research, CNCS-UEFISCDI, project number PN-III-P4-ID-PCE-2020-2798. I also acknowledge the support of the ANR grant ANR-24-CPJ1-0026-01 at Universit\'e Clermont Auvergne - LMBP and partial support by grants of the Ministry of Research, Innovation and Digitization, CNCS - UEFISCDI, project numbers  PN-IV-P2-2.1-TE-2023-2040 and PN-IV-P1-PCE-2023-2001, within PNCDI IV. I would like give special thanks to Christian Blanchet, Jun Murakami and Emmanuel Wagner for insightful discussions. Also, I thank Kazuo Habiro, Rinat Kashaev, Christine Lee and  Martin Palmer for useful conversations. 
 \vspace{-4mm} 

{\tableofcontents \vspace{-2mm}}

\section{Summary of the two topological models} In this section we will describe the construction of the geometrical set-up for our models as well as the main statements, and we will present the proofs in the later sections. 
\subsection{Geometrical set-up} Suppose that $L$ link, and let $\beta_n \in B_n$ be a braid such that $L= \widehat{\beta_n}$ by braid closure. Let $\mathbb D_{2n+2}$ be the $(2n+2)$-punctured disc where we split the set of punctures:
\begin{itemize}
\item[•]$2n$ punctures placed horizontally, called $p$-punctures (denoted by $\{1,..,2n\}$)
\item[•]$1$ puncture called $q$-puncture (labeled by $\{0\}$)
\item[•]$1$ puncture placed as in Figure \ref{ColouredAlex} (below the last $p$-puncture) called $s$-puncture.
\end{itemize}
Now, for $m\in \N$, let $C_{n,m}:=\Conf_{m}(\mathbb D_{2n+2})$ be the unordered configuration space of $m$ points in the punctured disc. A subtle part of our construction is the definition of a suitable covering space.
Suppose that we fix $\bar{\cN}=(\cN_1,\cN_2,...,\cN_n)$, called ``(multi-)level'', with $\cN_1+\cN_2+...+\cN_n=m-1$. 

{\bf Difficulty} Let $L(\bar{\cN}) \subseteq \Conf_{m}(\mathbb D_{2n+2})$ be the submanifold given by {\em configurations on symmetric ovals} with multiplicities prescribed by $\bar{\cN}$, as in Figure \ref{Pictureci}. We want to construct a covering of $C_{n,m}$ where $L(\bar{\cN})$ has a well-defined lift. In the literature there are two well known local systems given by the Lawrence representation \cite{Law} obtained by evaluating symmetric punctures by the same or opposite variables. Our  $L(\bar{\cN})$ does {\bf not} have a lift in the two  associated {\bf Lawrence covering spaces}. 
\begin{defn}[{\bf {Level $\bar{\cN}$} covering space}]
In order to be able to have {\em well-defined lifts} of $L(\bar{\cN})$, we construct a well-chosen covering that depends on the level $\bar{\cN}$. We do so by defining a finer local system $\Phi^{\bar{\cN}}$ on $\Conf_{m}(\mathbb D_{2n+2})$ (Definition \ref{localsystem}), associated to the {\em level $\bar{\cN}$}. This evaluates monodromies around symmetric punctures $i$ and $2n-i$ by different variables, depending on the multiplicity on the $i^{th}$ oval: $\cN_i$. Next, we use the homologies of the associated covering space: $H^{\bar{\cN}}_{n,m}$ and $H^{\bar{\cN},\partial}_{n,m}$ (they are versions of Borel-Moore homology introduced in \cite{CrM}, see Diagram \ref{diags}). 

\end{defn}
\begin{defn}[{Intersection pairing}] There is a Poincaré-type duality relating these homologies: 
$\left\langle ~,~ \right\rangle:  H^{\bar{\cN}}_{n,m} \otimes H^{\bar{\cN},\partial}_{n,m} \rightarrow \C[x_1^{\pm 1},...,x_l^{\pm 1},y^{\pm 1}, d^{\pm 1}] \ \  (\text{Proposition } \ref{P:3'})$ and moreover there is a braid group action on $H^{\bar{\cN}}_{n,m}$ (as in Proposition \ref{colbr}).
\end{defn}
\begin{rmk}[{\em \bf New framework: Topological models from configurations on ovals}] The submanifold $L(\bar{\cN})$ has a well-defined lift in the level $\bar{\cN}$ covering, and gives a well defined homology class.  This lifting result is definitely false if we take the local system that naturally arises from the Lawrence representation, and do not quotient further using the level $\bar{\cN}$.
This is the reason that prevented us from having a model with configurations on ovals for coloured Jones polynomials and coloured Alexander polynomials in \cite{Cr2} and for WRT invariants in \cite{WRT}. 
\end{rmk}

\subsection{First Topological model for coloured Alexander polynomials for links}\label{ssA} 
As we have seen, for a fixed level $\cN \in \N$ ($\cN\geq 2$), the quantum group at the root of unity $\xi_N=e^{\frac{2 \pi i}{2\cN}}$ gives non-semisimple quantum link invariants. They associate to a link $L$ coloured with generic colours $\lambda_1,...,\lambda_l\in \C$ an invariant $\Phi^{\cN}_{\bar{\lambda}}(L)\in \C$, called coloured Alexander polynomial \cite{ADO} (here $\bar{\lambda}=(\lambda_1,...,\lambda_l)$).
Let us fix such a level $\cN$ and $L$ an oriented framed link with framings $f_1,...,f_l\in \Z$ and $\beta_n \in B_n$ such that $L= \widehat{\beta_n}$. For the coloured Alexander polynomial $\Phi^{\cN}_{\la}$ we use the above homological set-up associated to the data:
\begin{itemize}
\item Number of particles $m_A(\cN):=2+(n-1)(\cN-1)$, Space $C_{n,m_A(\cN)}:=\Conf_{2+(n-1)(\cN-1)}\left(\mathbb D_{2n+2}\right)$
\item Multi-level: $\bar{\cN}:=(\cN_1,...,\cN_n)=(1,\cN-1,...,\cN-1)$; \ \ Local system: $\Phi^{\bar{\cN}}$
\item Specialisation of coefficients: $\sAt$ (Notation \ref{pA1}, Definition \ref{pA2}).
\item Indexing set: $\{\bar{0},\dots,\overline{\cN-1}\}= \big\{ \bar{i}=(i_1,...,i_{n-1})\in \N^{n-1} \mid 0\leq i_k \leq \cN-1, \  \forall 1 \leq k \leq n-1 \big\}.$
\end{itemize}
\begin{defn}(Coloured Homology classes) We use the homology classes of lifts of Lagrangian submanifolds from the base space. They are prescribed by collection of arcs/ ovals in the disc as in Notation \ref{paths}.
With this recipe, for  $\bar{i}\in \{\bar{0},\dots,\overline{\cN-1}\}$ we define two classes:
\vspace{-7mm}
\begin{figure}[H]
\centering
$${\color{red} \mathscr F_{\bar{i},\cN} \in \HAi} \ \ \ \ \ \ \ \ \text{ and }\ \ \ \ \ \ \ \ \  {\color{dgreen} \mathscr L_{\bar{i},\cN}\in \HAdi}.$$
\vspace{-3mm}

\includegraphics[scale=0.33]{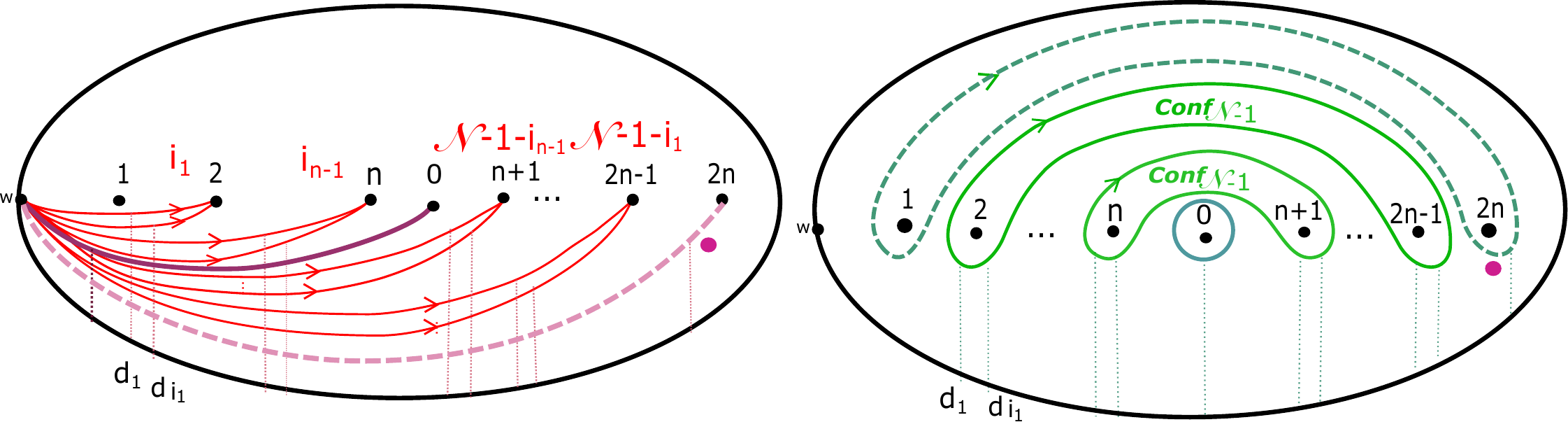}
\caption{\normalsize Coloured Alexander Homology Classes}
\label{ColouredAlex0}
\end{figure}
\end{defn}
\vspace{-6mm}
We look at the link as being the closure of a braid with $n$ strands. This induces a colouring $C$ of $2n$ points with $l$ colours, which we denote by:
\begin{equation}\label{eq:col}
C:\{1,...,2n\}\rightarrow \{1,...,l\}.
\end{equation}
\begin{defn}[{\bf \em Lagrangian intersection for coloured Alexander polynomials}] We define the following Lagrangian intersection in ${\color{blue} \Conf_{2+(n-1)(\cN-1)}\left(\mathbb D_{2n+2}\right)}$:
\begin{equation} 
\begin{aligned}
& \PA \in \LL=\Z[u_1^{\pm 1},...,u_l^{\pm 1},x_1^{\pm 1},...,x_l^{\pm 1},y^{\pm 1}, d^{\pm 1}]\\
& \PA:=\prod_{i=1}^l u_{i}^{ \left(f_i-\sum_{j \neq {i}} lk_{i,j} \right)} \cdot \prod_{i=2}^{n} u^{-1}_{C(i)} \cdot  \sum_{\bar{i}=\bar{0}}^{\overline{\cN -1}}  \left\langle(\beta_{n} \cup {\mathbb I}_{n+2} ) \ { \color{red} \FA}, {\color{dgreen} \LA}\right\rangle. 
   \end{aligned}
\end{equation} 
\end{defn}

\begin{thm}[{\bf \em Topological model for coloured Alexander polynomials of coloured links}]\label{THEOREMA}
The graded intersection $\PA$ recovers the $\cN^{th}$ coloured Alexander polynomial of $L$ as below:
\begin{equation}
\begin{aligned}
&\Phi^{\cN}_{\la}(L) =~ \PA \Bigm| _{\sAt}  \ \ \  (\sAt \text{ is a specialisation of variables}).
\end{aligned}
\end{equation} 
\end{thm}
This is the {\bf \em first topological model for these non-semisimple link invariants} coloured with different colours. For uni-coloured links we constructed topological models in \cite{Cr3,Cr2}. From the representation theory perspective, the change from the {\bf \em single to multi-colour case} is subtle due to the fact that the non-semisimple quantum invariants for coloured links strongly rely on the notion of modified traces and modified quantum dimensions.
 
\vspace{-4mm}
\subsubsection{Encoding modified quantum dimensions via topological tools}\label{ssmodif}


We  encode the modified dimensions using punctures and local systems on configuration spaces. Our idea is to add an additional puncture, the {\bf \em $s$-puncture} which {\bf \em  encodes modified quantum dimensions}. We do this via the local system,  by recording the monodromy around this puncture by a variable $y$. Then, we define our graded intersection $\PA$ which also has a contribution in $y$. At the very end, we specialise $y$ to the modified dimension through the function $\sAt$.
\begin{rmk}({\bf \em Role of the variables}) The pairing $\PA$ has $4$ types of variables: $x_i, y, d, u_i$. The first three variables come from the geometry of the local system, while the last one, $u_i$, encodes algebraic data provided by the pivotal structure coming from the quantum group.
\begin{enumerate}
\item[•][{\color{blue} \bf \em Winding around the link}] Variables {\color{blue}\bf \em $x_i$} encode {\color{blue}linking numbers with the link}, as winding numbers around the {\color{blue}$p$-punctures}
\item[•][{\color{dgreen} \bf \em Relative twisting}] The power of {\color{dgreen}$d$} records a relative twisting of the submanifolds, given by the winding around the {\color{dgreen}diagonal of the symmetric power}.
\item[•][{\color{carmine} \bf \em Modified dimension}] The power of {\color{carmine}$y$} of records the winding number around the {\color{carmine} $s$-puncture}
\item[•][{\color{teal} \bf \em Pivotal structure}] {\color{teal} \bf \em $u_i$} captures the {\color{teal} semisimplicity/ non-semisimplicity} of the invariant, which is given by a power of ${\color{blue} x_i}$, so a {\color{teal} power of the linking number with the link}.
\end{enumerate}
\end{rmk}

\subsection{Topological model- coloured Jones polynomials}
As we have discussed, coloured Jones polynomials come from the quantum group $U_q(sl(2))$ and to each $l$-component link $L$ and a choice of colours $N_1,...,N_l \in \N\setminus \{0,1\}$ they associate a polynomial $J_{N_1,...,N_l}(L,q)$. Let $N_1,...,N_l\in \N$ be a set of colours for our link and $\bar{N}:=(N_1,...,N_l).$ Using the colouring $C$ defined above, we denote $\ccN_i:=N_{C(i)}.$ 
{\bf \em Parameters for $J_{N_1,...,N_l}$} This time, we use the homological set-up associated to the data:
\begin{itemize}
\item Number of particles: $m_J(\bar{N}):=2+\sum_{i=2}^{n} N^C_i$
\item Configuration space: $C_{n,m_J(\bar{N})}:=\Conf_{2+\sum_{i=2}^{n} N^C_i}\left(\mathbb D_{2n+2}\right)$ 
\item Multi-level: $\bar{\cN}(\bar{N}):=(1,\ccN_2-1,...,\ccN_n-1)$, Local system: $\bar{\Phi}^{\bar{\cN}(\bar{N})}$ ({\bf \em  depend on the colours})
\item Specialisation of coefficients: $\sJt$ (Notation \ref{pJ1}, Definition \ref{pJ2}).
\end{itemize}
We remark that the number of particles $m_J(\bar{N})$, the multi-level $\bar{\cN}(\bar{N})$ and local system $\bar{\Phi}^{\bar{\cN}(\bar{N})}$ depend on $\bar{N}$. So, the good covering in this case depends on the colouring. Now, let us define:
\begin{equation}\label{stJi}
C(\bar{N}):= \big\{ \bar{i}=(i_1,...,i_{n-1})\in \N^{n-1} \mid 0\leq i_k \leq \ccN_{k+1}-1, \  \forall k\in \{1,...,n-1\} \big\}.
\end{equation}
For $\bar{i}\in C(\bar{N})$ consider the two homology classes given by the geometric supports from Figure \ref{Picture0i}:
\vspace*{-10mm}
 \begin{figure}[H]
\centering
$${\color{red} \FJ \in \HJi} \ \ \ \ \ \ \ \ \ \ \text{ and } \ \ \ \ \ \ \ \ \ \ \ \ \ {\color{dgreen} \LJ \in \HJdi} .$$
\vspace{-5mm}
\includegraphics[scale=0.4]{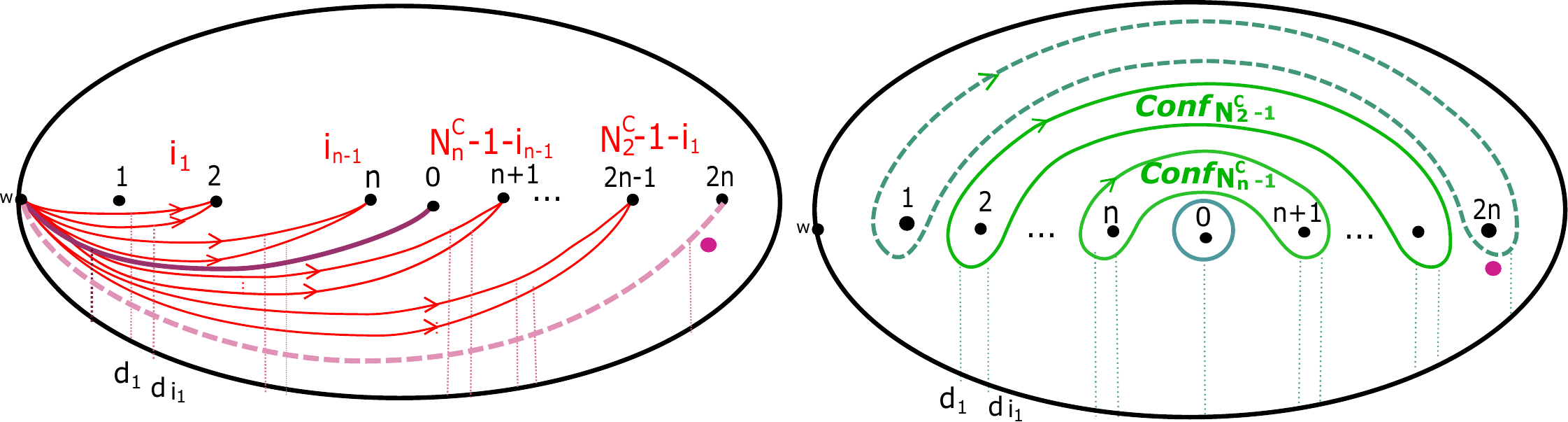}
\caption{ \normalsize Lagrangians for Coloured Jones link invariants coloured with diferent colours}
\label{Picture0i}
\end{figure}
\vspace*{-3mm}
\begin{thm}[{\bf \em Topological model for coloured Jones polynomials for coloured links}]\label{THEOREMJ}
Let us consider the Lagrangian intersection $\PJ \in \LL$, given by:
\begin{equation}
\PJ:=\prod_{i=1}^l u_{i}^{ \left(f_i-\sum_{j \neq {i}} lk_{i,j} \right)} \cdot \prod_{i=2}^{n} x^{-1}_{C(i)} \cdot \sum_{\bar{i}\in C(\bar{N})} \left\langle(\beta_{n} \cup {\mathbb I}_{n+2} ) \ { \color{red} \FJ}, {\color{dgreen} \LJ}\right\rangle. 
\end{equation} 
Then $\PJ$ recovers the coloured Jones polynomial of $L$ through specialisations:
\begin{equation}
 J_{N_1,...,N_l}(L,q)= \PJ \Bigm| _{\sJt}.
\end{equation} 
\end{thm}
We present the proof of this topological model in Section \ref{modelJ}.
\subsubsection{Knot closures} 
From now on we will be looking at the case of knots $K$, seen as closure of braids $\beta_n$, in the context described above (for $l=1$). In this situation, the intersection pairing for coloured Jones polynomials and coloured Alexander polynomials with colour $\cN$ coincide:
\begin{equation}
\PA=J\PA \in \Z[u^{\pm 1},x^{\pm 1},d^{\pm 1}].
\end{equation}
From Theorem \ref{THEOREMJ} and Theorem \ref{THEOREMA} we know that this intersection recovers the two invariants at level $\cN$, as below:
\begin{equation}
\begin{aligned}
&J_{\bar{\cM}}(K) = \PA \Bigm| _{\psi_{\bar{\cM}}}\\
&\Phi^{\cN}_{\la}(K) = \PA \Bigm| _{\sAt}.
\end{aligned}
\end{equation} 

\begin{thm}[{\bf \em Unifying all coloured Jones polynomials of bounded level}]\label{THEOREMAU}
Let us fix $\cN\in \N$. Then, $\PA$ recovers all coloured Jones polynomials of colours $\cM\leq \cN$, as below:
\begin{equation}
\begin{aligned}
J_{\cM}(K)& =~ \PA \Bigm| _{\psi_{\bar{\cM}}}.
\end{aligned}
\end{equation} 
\end{thm}
\subsubsection{New invariants at level $\cN$} \label{ssAu}

\begin{convention}\label{convention}
Let us fix the ring $\LL=\Z[x^{\pm 1},d^{\pm 1}]$. For the following part, having in mind that we work with knots, for the simplicity of the construction we set $y=1$. If we do so, our intersection form $\PA$ recovers the normalised versions of the $\cN^{th}$ coloured Jones and Alexander polynomials. We denote by $J_{\cN}$ the ${\cN}^{th}$ normalised coloured Jones polynomial (in the above section we had $J_{\bar{\cN}}$ the un-normalised version of this invariant). 

Further, by setting $u=x$ we can look at the intersection $$\PA \in \LL=\Z[x^{\pm 1},d^{\pm 1}]$$ and all the specialisations associated to the coloured Jones polynomials induce well-defined specialisations from $\LL$, which we denote by $\snJ$, as in formula \eqref{u1J}:
\begin{equation}
\snJ(x)= d^{1-\cN}.
\end{equation}
\end{convention}
Now for a fixed level $\cN$, we have defined our state sum of Lagrangian intersections $\PA \in \LL$ in the {\bf configuration space of $(n-1)(\cN-1)+2$} points in the  disc. This is given by the set of Lagrangian intersections $\{\langle (\beta_{n} \cup {\mathbb I}_{n+2} ) \ { \color{red} \FA}, {\color{dgreen} \LA} \rangle\}_{\bar{i}\in C(\cN)}.$ For the next step, we define a sequence of nested ideals in $\LL$, denoted by $\cdots \supseteq \tilde{I}^J_{\cN} \supseteq \tilde{I}^J_{\cN+1} \supseteq \cdots$, given by the formulas:
\begin{equation}
\tilde{I}_{\cN}^{J}:=\langle \prod_{i=1}^{\cN} (xd^{i-1}-1) \rangle \subseteq \LL.\\
\end{equation}
 In this manner, we obtain a sequence of associated quotient rings, with maps between them:

\begin{equation}
\begin{aligned}
&\hspace{13mm} l^J_{\cN} \hspace{11mm} l^J_{\cN+1}\\
&\cdots \LLNJ \longleftarrow \LLNNJ \longleftarrow \cdots \ \ \ \ \ \ \ \ \ \ \ \ \ \ \LLhJ.
\end{aligned}
\end{equation}
 In Theorem \ref{levN} we prove the following. 

\begin{thm}[{\bf \em $\cN^{th}$ Unified Jones invariant}] \label{INV} Let $\PAAJ$ be the image of the intersection $\PA$ in the $\cN ^{th}$ quotient ring $\LLNJ$. Then $\PAAJ$ is a well-defined knot invariant recovering all coloured Jones polynomials up to level $\cN$:
$$  \PAAJ|_{\smnbJt}=J_{\cM}(K), \ \ \ \ \ \forall \cM \leq \cN.$$
Here, $\smnbJt$ is the change of coefficients from relation \eqref{quotmn}. We call $\PAAJ$ the $\cN^{th}$ Unified Jones invariant.
\end{thm}
\subsection{Construction of the Geometric Universal Invariant} \label{ssU} 
In this part we show that the level $\cN$ unified Jones invariants have a good asymptotic behaviour and they give knot invariants in the associated projective limit ring. The key result that makes this unification possible is Theorem \ref{THEOREMAU}.
\begin{defn}[Limit ring] We denote the projective limit of this sequence of rings as follows:
\begin{equation}
\LLhJ:= \underset{\longleftarrow}{\mathrm{lim}} \ \LLNJ.
\end{equation}
\end{defn}
\begin{thm}[{\bf \em Universal Coloured Jones Invariant}] \label{THJ}We construct a well-defined knot invariant $\IJJ \in \LLhJ$ as a limit of the {graded intersections via configuration spaces on ovals}, recovering all coloured Jones polynomials:
\begin{equation}\label{eq10:4}
J_{\cN}(K)=  \IJJ|_{\usnJ}.
\end{equation}
In this expression $\usnJ$ is a ring homomorphism (see  Definition \ref{pA4}).
\end{thm}

We prove this topological model in Section \ref{modelA}.

\subsection{Geometry of the models for coloured Jones versus coloured Alexander invariants for the link case} (Figure \ref{CJA})

\

Let us recall the set-up for the non-semisimple invariants, which we saw in Subsection \ref{ssA}. 

{\bf \em Parameters for the topological model for $\Phi^{\cN}_{\la}$} 
\begin{itemize}
\item Number of particles: $m_A(\cN):=2+(n-1)(\cN-1)$
\item Configuration space: $C_{n,m_A(\cN)}:=\Conf_{2+(n-1)(\cN-1)}\left(\mathbb D_{2n+2}\right)$
\item Multi-level: $\bar{\cN}:=(\cN_1,...,\cN_n)=(1,\cN-1,...,\cN-1)$; \ \ Local system: $\bar{\Phi}^{\bar{\cN}}$
\item Specialisation of coefficients: $\sAt$ (Notation \ref{pA1}, Definition \ref{pA2}).
\end{itemize}
\begin{defn}(Set of states at level $\cN$) In this case, our indexing set will be given by:
 $$C(\cN):=\{\bar{0},\dots,\overline{\cN-1}\}= \big\{ \bar{i}=(i_1,...,i_{n-1})\in \N^{n-1} \mid 0\leq i_k \leq \cN-1, \  \forall k\in \{1,...,n-1\} \big\}.$$ 
\end{defn}

There are three parameters used for $\PA$ which are intrinsic and do not depend on the choice of the colour $\bar{\lambda}$:
\begin{enumerate}
\item The set of states $C(\cN)$ (in contrast to the case of coloured Jones polynomials- Subsection \ref{stJi})
\item The multi-level $\bar{\cN}$, the local system $\bar{\Phi}^{\bar{\cN}}$ and so the covering space which depend just on ${\cN}$
\item The modified dimension is encoded by the variable $y$, being intrinsic, not depending on $\bar{\lambda}$.
\end{enumerate} 

In this part we compare the main geometric differences between the models for $\Phi^{\cN}_{\bar{\lambda}}$ and $J_{N_1,...,N_l}(L,q)$.
\vspace{-10mm}
\begin{center}
\begin{tikzpicture}
[x=1mm,y=0.8mm]

\node (b2) [color=black] at (-45,0)   {Coloured Alexander polyn.};
\node (b1)               at (30,0)    {Coloured Jones polyn.};
\node (b4) [color=black] at (-45,-10)   {Level $\cN \in \N$};
\node (b3) [color=black] at (40,-10)   {Colours $N_1,...,N_l \in \N \rightarrow N^C_1,...,N^C_n$ };
\node (b3) [color=black] at (38,-15)   {$\bar{N}:=(N_1,...,N_l)$};

\node (a2)               at (-5,-20)    {Number of particles in};
\node (a2)               at (-5,-25)    {the configuration space};

\node (b3) [color=black] at (-50,-25)   {$m_A(\cN):=2+(n-1)(\cN-1)$};
\node (b4) [color=black] at (50,-25)   {$m_J(\bar{N}):=2+\sum_{i=2}^{n} N^C_i$};

\node (a2)               at (-5,-35)    {Coloured Multi-level};

\node (b3) [color=black] at (-50,-35)   {$\bar{\cN}:=(1,\cN-1,...,\cN-1)$};
\node (b4) [color=black] at (50,-35)   {$\bar{\cN}(\bar{N}):=(1,\ccN_2-1,...,\ccN_n-1)$};

\end{tikzpicture}
\end{center}
We remark that for coloured Jones polynomials the set of states depends on the colours, but for coloured Alexander polynomials $C(\cN)$ it is intrinsic. Also, the multi-level $\bar{\cN}$ and the associated covering space from Theorem \ref{THEOREMAU} depend just on the level of the root of unity $\cN$ but not on the individual colours $\bar{\lambda}$ (see Figure \ref{CJA}). 
 \begin{figure}[H]
\centering
\vspace{-7mm}
$$\tilde{C}_{n,m_A(\cN)} \hspace{15mm} \text{ Lawrence covering space } \hspace{15mm} \tilde{C}_{n,m_J(\bar{N})}$$
\vspace*{-7mm}
$$\hspace{5mm}\downarrow \ \nexists \ \ \text{lifts} $$
$$\bar{C}^{\bar{\cN}}_{n,m_A(\cN)} \hspace{15mm} \text{ Level } \bar{\cN}  \text { covering space } \hspace{18mm} \bar{C}^{\bar{\cN}(\bar{N})}_{n,m_J(\bar{N})}$$
$$\text{Independent    }  \hspace{28mm} \normalsize \tilde{L}(\bar{\cN}) \hspace{30mm} \text{ Depends    } $$
$$ \ \ \ \ \text{ \ \ \ \ \ \ \ \ on the colours} \ \ \bar{\lambda} \hspace{22mm}\downarrow \text{ well-defined lifts \ \ \ \ \ \ \  on the colours } \bar{N} \ \ \ \ $$
\vspace{-4mm}

\includegraphics[scale=0.38]{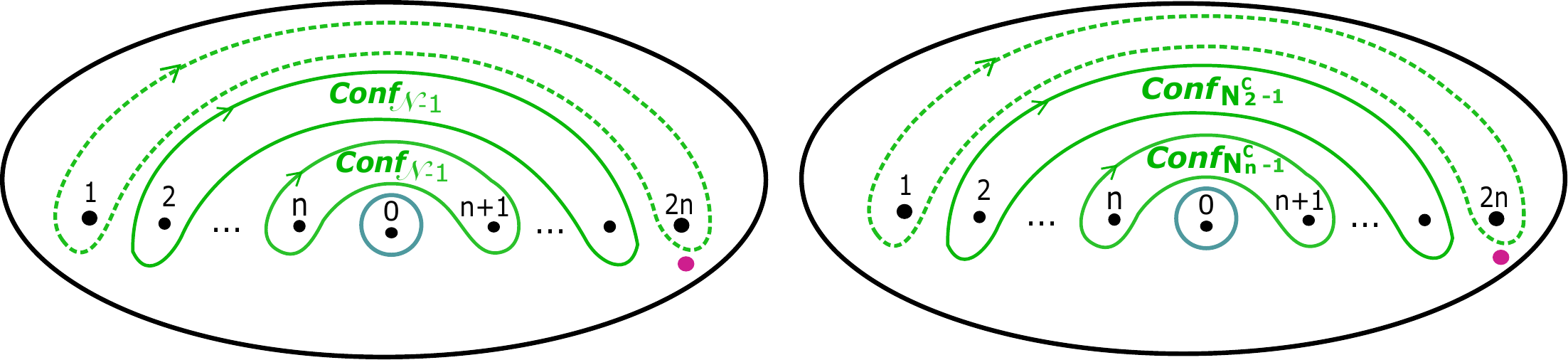}
\vspace*{-40mm}

$$\normalsize L(\bar{\cN})$$
\vspace*{25mm}

$C_{n,m_A(\cN)}:=\Conf_{2+(n-1)(\cN-1)}\left(\mathbb D_{2n+2}\right)\hspace{20mm} C_{n,m_J(\bar{N})}:=\Conf_{1+\sum_{i=2}^{n} N^C_i}\left(\mathbb D_{2n+2}\right)$
\caption{ \normalsize Covering spaces for coloured Alexander versus coloured Jones polynomials}\label{CJA}
\vspace*{1mm}
\end{figure}

\section{Notations} \label{S:not}
\begin{notation}(Specialisations of modules)\label{N:spec}\\ 
Let $N$ be a module over a ring $R$. Let $R'$ be another ring and suppose that we have a specialisation of the coefficients, meaning a morphism:
$\psi: R \rightarrow R'.$
We denote by  
$N|_{\psi}:=N \otimes_{R} R'$
the specialisation of the module $N$ by the function $\psi$. 
\end{notation}
\begin{notation}[Quantum numbers and modified dimensions]
$$ \{ x \}_q :=q^x-q^{-x} \ \ \ \ [x]_{q}:= \frac{q^x-q^{-x}}{q-q^{-1}}.$$
For $N\in \N$, $\xi_N=e^{\frac{2\pi i}{2N}}$ and $\lambda \in \C$, the associated modified dimension is given by:
$$ d(\lambda):= \frac{\{\lambda\}_{\xi_N}}{\{N\lambda\}_{\xi_N}}.$$
\end{notation}
\begin{defn}(Left $\slash$ right hand side of the disc) We separate the punctured disc into two halves, as follows: 
\begin{itemize}
\item (Left hand side of the punctured disc) This will be the half of the disc from Figure \ref{CJA} that contains the first $n$ $p$-punctures and passes though the puncture labeled by $0$.
\item (Right hand side of the punctured disc) This will be the other half of the disc from Figure \ref{CJA} that contains the rest of the $p$-punctures.
\end{itemize}

\end{defn}

\begin{defn}[Multi-indices at level $\bar{\cM}$ and $\bar{\cN}$]

We will use the following sets of multi-indices, which will parametrise the states for our models:
\begin{equation}\label{mi}
\begin{aligned}
&\{\bar{0},\dots,\overline{\cN-1}\}:= \big\{ \bar{i}=(i_1,...,i_{n-1})\in \N^{n-1} \mid 0\leq i_k \leq \cN-1, \  \forall k\in \{1,...,n-1\} \big\}\\
&\{\bar{\cM},\dots,\overline{\cN-1}\}:= \big\{ \bar{i}=(i_1,...,i_{n-1})\in \N^{n-1} \mid 0\leq i_k \leq \cN-1, \  \forall k\in \{1,...,n-1\} \text{ and }\\
& \hspace{90mm} \exists j\in \{1,...,n-1\}, \  \cM\leq i_j \big\}
\end{aligned}
\end{equation}

\end{defn}
\begin{defn}[Multi-indices]
For a choice of parameters $l \in \N$, $l\geq1$ and $\alpha_1,..,\alpha_l, \eta,..,\eta_l \in \mathbb C$, $N_1,...,N_l \in \N$ we denote the following associated multi-indices:
\begin{equation}
\begin{cases}
\bar{\alpha}:=(\alpha_1,..,\alpha_l)\\
\bar{\eta}:=(\eta_1,..,\eta_l)\\
\bar{N}:=(N_1,...,N_l).
\end{cases}
\end{equation} 
\end{defn}
\subsection{Specialisations for the homology groups- first specialisation and globalised specialisation}
\begin{defn}[Specialisations for the homology groups] Let us consider two multi-indices: $\bar{\alpha}\in \C^{l}$ and $\bar{\eta}\in (\C[q^{\pm 1}])^{\bar{l}}$. We denote the following specialisation of coefficients:
$$ \psi_{q,\bar{\alpha},\bar{\eta}}: \Z[x_1^{\pm 1},...,x_{l}^{\pm 1},y_1^{\pm 1},...,y_{\bar{l}}^{\pm 1}, d^{\pm 1}] \rightarrow \C[q^{\pm \frac{1}{2}},q^{\pm \frac{\alpha_1}{2}},...,q^{\pm \frac{\alpha_l}{2}}]$$
\begin{equation}\label{notlk}
\begin{cases}
&\psi_{q,\bar{\alpha},\bar{\eta}}(x_j)=q^{\alpha_j}, \ j\in \{1,...,l\}\\
&\psi_{q,\bar{\alpha},\bar{\eta}}(y_j)=\eta_j,j\in \{1,...,l\}\\
&\psi_{q,\bar{\alpha},\bar{\eta}}(d)= q^{-1}. 
\end{cases}
\end{equation} 
\end{defn}

\begin{defn}[Globalised specialisations] Let us consider an integer number $t$, and two multi-indices of $l$ colours: $\bar{\alpha}\in \C^{l}$ and $\bar{\eta}\in \C^{l}$. Associated to this data, we define the specialisation of coefficients:
$$ \psi^{t}_{q,\bar{\alpha},\bar{\eta}}: \Z[u_1^{\pm 1},...,u_{l}^{\pm 1},x_1^{\pm 1},...,x_{l}^{\pm 1},y_1^{\pm 1},...,y_{l}^{\pm 1}, d^{\pm 1}] \rightarrow \C[q^{\pm \frac{1}{2}},q^{\pm \frac{\alpha_1}{2}},...,q^{\pm \frac{\alpha_l}{2}}]$$
\begin{equation}\label{not}
\begin{cases}
&\psi^{t}_{q,\bar{\alpha},\bar{\eta}}(u_j)=\left(\psi^{t}_{q,\bar{\alpha},\bar{\eta}}(x_j)\right)^t=q^{t\alpha_j}\\
&\psi^{t}_{q,\bar{\alpha},\bar{\eta}}(x_j)=q^{\alpha_j}, \ j\in \{1,...,l\}\\
&\psi^{t}_{q,\bar{\alpha},\bar{\eta}}(y_j)=\eta_j\\
&\psi^{t}_{q,\bar{\alpha},\bar{\eta}}(d)=q^{-1}. 
\end{cases}
\end{equation} 
\end{defn}
\begin{defn}(Specialisations of coefficients)\label{spec}\\
For a colouring $C:\{1,...,n\}\rightarrow \{1,...,l\}$ we consider the specialisation of coefficients as below:
\begin{center}
\begin{tikzpicture}
[x=1.2mm,y=1.4mm]
\node (b1)               at (-27,0)    {$\Z[x_1^{\pm 1},...,x_n^{\pm 1},y^{\pm 1}, d^{\pm 1}]$};
\node (b2)   at (27,0)   {$\Z[x_1^{\pm 1},...,x_l^{\pm 1},y^{\pm 1}, d^{\pm 1}]$}; 
\node (b3)   at (0,-20)   {$\C[q^{\pm \frac{1}{2}},q^{\pm \frac{\alpha_1}{2}},...,q^{\pm \frac{\alpha_l}{2}}]$};
,\draw[->]   (b1)      to node[xshift=1mm,yshift=5mm,font=\large]{$f_C$ \eqref{eq:8}}                           (b2);
\draw[->]             (b2)      to node[right,xshift=2mm,font=\large]{$\psi^{t}_{q,\bar{\alpha},\bar{\eta}}$\ \eqref{not}}   (b3);
\draw[->,thick,dotted]             (b1)      to node[left,font=\large]{}   (b3);
\end{tikzpicture}
\end{center}
\end{defn}
\begin{notation}\label{not'}
In the formulas from the paper we denote by $N^C_i:=N_{C(i)}$.
\end{notation}
\begin{defn}(Our setting: specialisation corresponding to a braid closure) \label{introcol}\\
We will use this change of coefficients in the situation where $n$ is replaced by $2n$ and these $2n$ points inherit a colouring with $l$ colours coming from a braid closure of a braid with $n$ strands:
\begin{equation}\label{eq:col}
C:\{1,...,2n\}\rightarrow \{1,...,l\}.
\end{equation}
\end{defn}

\subsection{Coloured Jones polynomials- generic parameters}
\begin{defn}[Specialisation for the classes associated to coloured Jones polynomials] Let us fix the multi-indices:
\begin{equation}
\begin{cases}
\bar{N}:=(N_1-1,...,N_l-1)\\
\overline{[N]}:=([N^C_1]_{q})\\
\end{cases}
\end{equation}
Then, we obtain the following specialisation associated to $\bar{\alpha}=\bar{N}$ and $\eta=\overline{[N]}$:
$$ \sJ: \C[x_1^{\pm 1},...,x_{l}^{\pm 1},y^{\pm1},d^{\pm 1}] \rightarrow \C[q^{\pm 1}]$$
\begin{equation}\label{sJc}
\begin{cases}
&\sJ(x_i)=q^{N^C_i-1}, \ i\in \{1,...,l\}\\
&\sJ(y)=[N^C_1]_{q},\\
&\sJ(d)= q^{-1}.
\end{cases}
\end{equation}
\end{defn}
For the case of the link invariants, we will use the globalised specialisation associated to the multi-indices:
\begin{equation}
\begin{cases}
\bar{N}:=(N_1-1,...,N_l-1)\\
\overline{[N]}:=([N^C_1]_{q})\\
t=1.
\end{cases}
\end{equation} 
\begin{notation}[Specialisation at generic $q$]\label{pJ1}
Let us denote the specialisation associated to the above parameters as:
\begin{equation}
\sJt:=\sJtt.
\end{equation}
\end{notation}

\begin{defn}[Specialisation for coloured Jones polynomials]\label{pJ2}
The associated specialisation of coefficients is given by:
$$ \sJt: \C[u_1^{\pm 1},...,u_{l}^{\pm 1},x_1^{\pm 1},...,x_{l}^{\pm 1},y^{\pm 1},d^{\pm 1}] \rightarrow \C[q^{\pm 1}]$$
\begin{equation}\label{eq:8''''} 
\begin{cases}
&\sJt(u_j)=\left(\sJt(x_j)\right)^t=q^{N_i-1}\\
&\sJt(x_i)=q^{N_i-1}, \ i\in \{1,...,l\}\\
&\sJt(y)=[N^C_1]_{q},\\
&\sJt(d)= q^{-1}.
\end{cases}
\end{equation}
\end{defn}
\subsection{Coloured Alexander polynomials- parameters at roots of unity}

\begin{defn}[Specialisation for the classes associated to coloured Alexander polynomials] Let us fix the multi-indices:
\begin{equation}
\begin{cases}
\bar{\lambda}:=(\lambda_1,...,\lambda_l)\\
\overline{d(\lambda)}:=([\lambda_{C(1)}]_{\xi_{\cN}})\\
\end{cases}
\end{equation}
Then, we consider the following specialisation, which is associated to the multi-indices $\bar{\alpha}=\bar{\lambda}$ and $\eta=\overline{d(\lambda)}$:
$$ \sA: \C[x_1^{\pm 1},...,x_{l}^{\pm 1},y^{\pm1},d^{\pm 1}] \rightarrow \C[\xi_{\cN}^{\pm \frac{1}{2}},\xi_{\cN}^{\pm \frac{\lambda_1}{2}},...,\xi_{\cN}^{\pm \frac{\lambda_l}{2}}]$$
\begin{equation}
\begin{cases}
&\sA(x_i)=\xi_{\cN}^{\lambda_i}, \ i\in \{1,...,l\}\\
&\sA(y)=([\lambda_{C(1)}]_{\xi_{\cN}}),\\
&\sA(d)= \xi_{\cN}^{-1}.
\end{cases}
\end{equation}
\end{defn}
For this case of the link invariants, we use the globalised specialisation associated to the multi-indices:
\begin{equation}
\begin{cases}
\bar{\lambda}:=(\lambda_1,...,\lambda_l)\\
\overline{d(\lambda)}:=([\lambda_{C(1)}]_{\xi_{\cN}})\\
t=1-\cN.
\end{cases}
\end{equation} 

\begin{notation}[Specialisation at roots of unity]\label{pA1}
Let us denote the specialisation associated to the above parameters as:
\begin{equation}
\sAt:=\sAtt.
\end{equation}
\end{notation}
\begin{defn}[Specialisation for coloured Alexander polynomials]\label{pA2}
The associated specialisation of coefficients is given by:
$$ \sAt: \C[u_1^{\pm 1},...,u_{l}^{\pm 1},x_1^{\pm 1},...,x_{l}^{\pm 1},y^{\pm 1}, d^{\pm 1}] \rightarrow \C[\xi_{\cN}^{\pm \frac{1}{2}},\xi_{\cN}^{\pm \frac{\lambda_1}{2}},...,\xi_{\cN}^{\pm \frac{\lambda_l}{2}}]$$
\begin{equation}\label{eq:8''''} 
\begin{cases}
&\sAt(u_j)=\left(\sJt(x_j)\right)^{1-\cN}=\xi_{\cN}^{(1-\cN)\lambda_i}\\
&\sAt(x_i)=\xi_{\cN}^{\lambda_i}, \ i\in \{1,...,l\}\\
&\sAt(y)=([\lambda_{C(1)}]_{\xi_{\cN}}),\\
&\sAt(d)= \xi_{\cN}^{-1}.
\end{cases}
\end{equation}
\end{defn}
\subsection{Specialisations of coefficients for universal invariants}
\begin{defn}[Rings for the universal invariant] Let us define the following rings:
 \begin{equation}
 \begin{cases}
 \LL:=\Z[x^{\pm 1},d^{\pm 1}]\\
\LNJ=\Z[d^{\pm 1}] .
 \end{cases}
 \end{equation}
\end{defn}
\begin{defn}[Specialisation for the universal invariant] 

\

We change the ring of coefficients as above, and let us define the associated globalised specialisation of coefficients as below:
$$ \snJ: \LL=\Z[x^{\pm 1},d^{\pm 1}] \rightarrow \LNJ$$
\begin{equation}\label{u1J} 
\snJ(x)= d^{1-\cN}.
\end{equation}
\end{defn}

\begin{defn}[Universal specialisation map, as in \eqref{eq10:limit}] \label{pA4}We have the projective limit of the sequence of rings:
\begin{equation*}
\LLhJ:= \underset{\longleftarrow}{\mathrm{lim}} \ \LLNJ.
\end{equation*}
Then, we have a well-defined induced universal specialisation map, which we denote:
\begin{equation}
\usnJ: \LLhJ \rightarrow \LNJ.
  \end{equation}
\end{defn}

\subsection{Summary: Diagram with all the specialisations of coefficients for link invariants}
We summarize the specialisations of coefficients and homology groups in Figure \ref{diagsl}. They use also the following definition, which has an important role for the proof of the lifting property of submanifolds given by configuration spaces based on ovals presented in Section \ref{S:4}.

\begin{defn}[$\bar{\cN}$-change of variables for lifting configurations on ovals]
Let 
\begin{equation}
\begin{aligned}
\tilde{f}^{\bar{\cN}}: \ &\Z[x_1^{\pm 1},...,x_n^{\pm 1},\bar{x}_1^{\pm 1},...,\bar{x}_n^{\pm 1},y_1^{\pm 1}..., y_{\bar{l}}^{\pm 1}, \bar{d'}^{\pm 1}]\rightarrow \Z[x_1^{\pm 1},...,x_n^{\pm 1},y_1^{\pm 1}..., y_{\bar{l}}^{\pm 1}, d'^{\pm 1}]\\
\end{aligned}
\end{equation}
given by:
\begin{equation}
\begin{cases}
&\tilde{f}^{\bar{\cN}}(x_i)=x_i,\\    
&\tilde{f}^{\bar{\cN}}(\bar{x_i})=x_i\cdot d'^{(\cN_i-1)},\\ 
&\tilde{f}^{\bar{\cN}}(y_j)=y_j,\\ 
&\tilde{f}^{\bar{\cN}}(\bar{d'})=d'^{2}.
\end{cases}
\end{equation}

\end{defn}

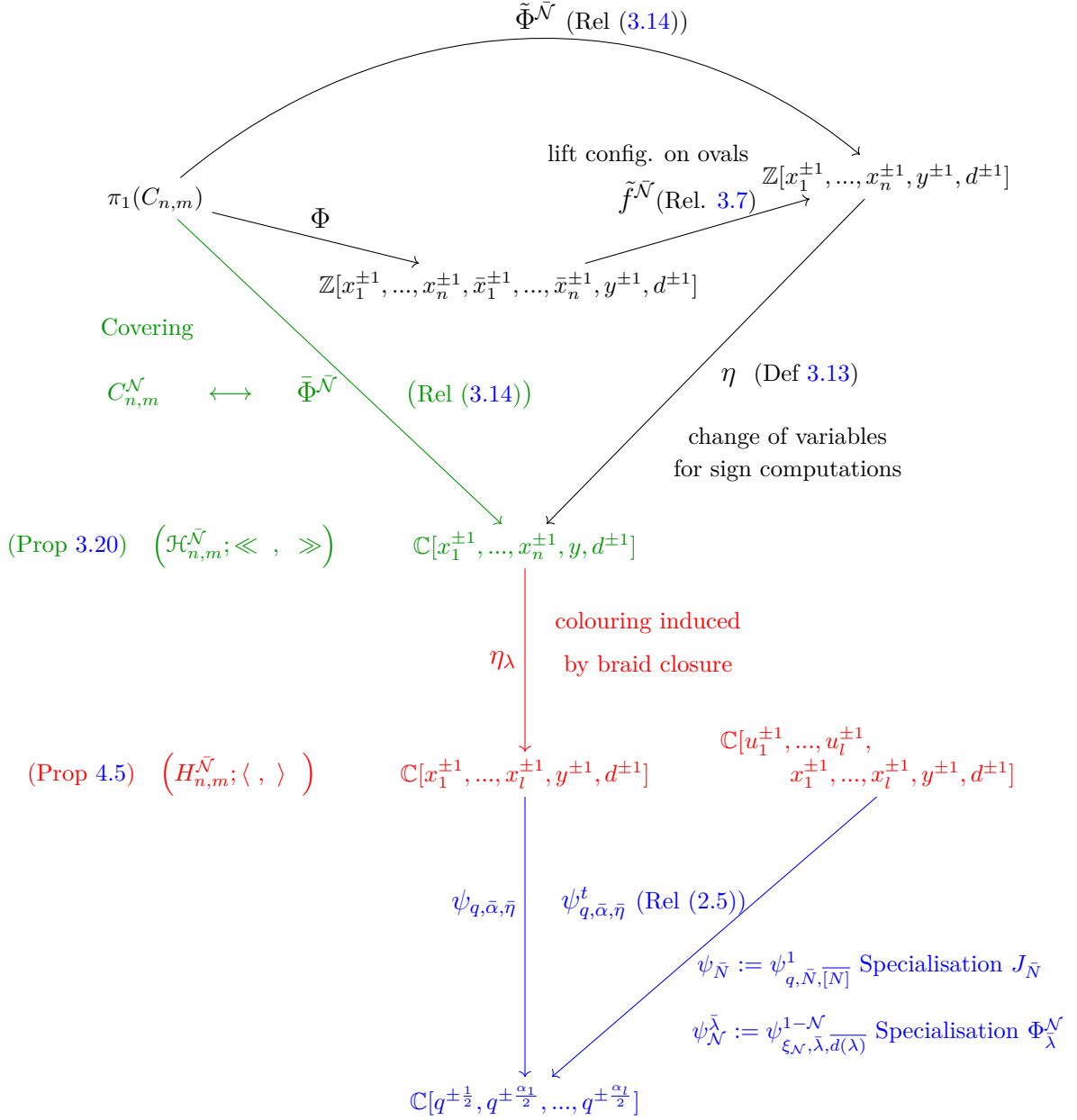
\begin{figure}[H]
\begin{center}
\begin{tikzpicture}
[x=1.2mm,y=1.6mm]

\node (O)               at (65,54)    {lift config. on ovals};
\node (Or)               at (82,28)    {change of variables};
\node (Or)              at (82,25)    {for sign computations};

\node (A) [color=red]              at (65,11)    {colouring induced};
\node (A) [color=red]              at (65,7)    {by braid closure};

\node (A) [color=blue]              at (92,-21)    {$\sJt:=\sJtt$ Specialisation $J_{\bar{N}}$};
\node (A) [color=blue]              at (93,-27)    {$\sAt:=\sAtt$ Specialisation $\Phi^{\cN}_{\bar{\lambda}}$};

\node (b1)               at (5,50)    {$\pi_1(C_{n,m})$};
\node (b20) [color=dgreen] at (4,38)   {Covering};
\node (b200) [color=dgreen] at (8,32)   {$C_{n,m}^{\cN}$ $ \ \ \ \ \ \longleftrightarrow$};
\node (b2) [color=dgreen] at (50,18)   {$\C[x_1^{\pm 1},...,x_n^{\pm 1},y^{\pm 1}, d^{\pm 1}]$};
\node (b22) [color=dgreen] at (7,18)   {(Prop \ref{P:3'''}) \ \ $\left(\mathscr H^{\bar{\cN}}_{n,m}; \ll ~,~ \gg  \right)$};
\node (b3) [color=red]   at (50,-3)   {$\C[x_1^{\pm 1},...,x_l^{\pm 1},y^{\pm 1}, d^{\pm 1}]$};
\node (b33) [color=red] at (7,-3)   {(Prop \ref{P:3'}) \ \ $\left( H^{\bar{\cN}}_{n,m}; \left\langle ~,~ \right\rangle \ \ \right)$};
\node (b3') [color=red]   at (83,0)   {$\C[u_1^{\pm 1},...,u_l^{\pm 1},$};
\node (b33') [color=red]   at (96,-3)   {$x_1^{\pm 1},...,x_l^{\pm 1},y^{\pm 1}, d^{\pm 1}]$};
\node (b4) [color=blue]  at (50,-33)   {$\C[q^{\pm \frac{1}{2}},q^{\pm \frac{\alpha_1}{2}},...,q^{\pm \frac{\alpha_l}{2}}]$};
\node (b1') [color=black]  at (48,42)   {$\Z[x_1^{\pm 1},...,x_n^{\pm 1},\bar{x}_1^{\pm 1},...,\bar{x}_n^{\pm 1},y^{\pm 1}, d^{\pm 1}]$};
\node (b1'') [color=black]  at (94,52)   {$\Z[x_1^{\pm 1},...,x_n^{\pm 1},y^{\pm 1}, d^{\pm 1}]$};

\draw[->,color=dgreen]   (b1)      to node[yshift=-3mm,xshift=11mm,font=\large]{$\bar{\Phi}^{\bar{\cN}} \hspace{10mm} \normalsize{(\text{Rel } \eqref{sgr})}$}                           (b2);
\draw[->,color=red]             (b2)      to node[left,font=\large]{$\eta_{\lambda} $}   (b3);
\draw[->,color=black]             (b1')      to node[left,xshift=10mm,yshift=5mm,font=\large]{$\tilde{f}^{\bar{\cN}}$\normalsize{(Rel. \ref{N-chh})}}   (b1'');
\draw[->,color=blue,yshift=5mm]             (b3)      to node[left,font=\large,yshift=5mm]{$\psi_{q,\bar{\alpha},\bar{\eta}}$}   (b4);
\draw[->,color=blue]             (b33')      to node[left,font=\large,xshift=6mm,yshift=5mm]{$\psi^t_{q,\bar{\alpha},\bar{\eta}}$ \normalsize{(Rel \eqref{not})}}   (b4);
\draw[->,color=black]   (b1)      to node[right,xshift=-2mm,yshift=3mm,font=\large]{$\Phi$}                        (b1');
\draw[->,color=black]   (b1)      to [out=40,in=140] node[right,xshift=-2mm,yshift=3mm,font=\large]{$\tilde{\Phi}^{\bar{\cN}}$ {\normalsize{(Rel \eqref{sgr})}}}                        (b1'');
\draw[->,color=black]   (b1'')      to node[right,yshift=-2mm,xshift=1mm,font=\large]{$\eta$ \normalsize{ (Def \ref{groupring'})}}                        (b2);
\end{tikzpicture}
\end{center}
\vspace{-3mm}
\caption{\normalsize Specialisations of coefficients: Coloured link invariants}\label{diagsl}
\end{figure}

\section{Homological set-up} \label{S3}
Let us consider $n,\bl \in \N$. 
We denote by $\mathscr D_{2n+1,\bl}$ the $(2n+\bl+1)$-punctured disc where the set of punctures is spit as follows:

\begin{itemize}
\setlength\itemsep{-0.2em}
\item[•]$2n$ punctures placed horizontally, called $p$-punctures (denoted by $\{1,..,2n\}$)
\item[•]$1$ puncture called $q$-puncture (labeled by $\{0\}$)
\item[•]$\bar{l}$ punctures placed as in Figure \ref{Localsystt}, called $s$-punctures (labeled by $\it{\{1,...,\overline{l}\}}$ ).
\end{itemize}
\begin{convention}If $\bar{l}=1$, which will be the case for our models from this paper, we will remove it from the indices of our configuration spaces and homology groups.
\end{convention}

\subsection{Configuration space of the punctured disc}

Now, for $m\in \N$ we consider the unordered configuration space of $m$ points in the punctured disc $\mathscr D_{2n+1,\bl}$, and denoted it by:
 $$\Clm:=\Conf_{m}(\mathscr D_{2n+1,\bl}).$$ 
We fix a base point of this configuration space, by choosing $d_1,..d_m \in \partial \hspace{0.5mm}\mathscr D_{2n+1,\bl}$ and consider ${\bf d}=(d_1,...,d_m)$ to be the associated point in the configuration space. In the following part we construct a local system on the space $\Clm$. 

We assume that $m \geq 2$. Then, let us start with the abelianisation to the first homology group of this configuration space, which has the following form.
\begin{prop}[Abelianisation to the homology group]
 Let $[ \ ]: \pi_1(\Clm) \rightarrow H_1\left( \Clm\right)$ be the abelianisation of the fundamental group of our configuration space. Its homology has the structure presented below:
\begin{equation*}
\begin{aligned}
H_1\left( \Clm \right)  \simeq \ \ \ \  \ & \Z^{n+1} \ \ \ \ \oplus \ \ \ \  \Z^{n} \ \ \ \ \ \oplus \ \ \ \ \Z^{\bl} \ \ \ \  \oplus \ \ \ \ \Z\\
&\langle [\sigma_i] \rangle \ \ \ \ \ \ \ \ \ \langle [\bar{\sigma}_{i'}] \rangle \ \ \ \ \ \ \ \ \ \langle [\gamma_j] \rangle \ \ \ \ \ \ \ \ \langle [\delta]\rangle,  \ \ \ {i\in \{0,...,n\}}, j\in \{1,...,\bl\},\\
& \hspace{69mm} {i'\in \{1,...,n\}}.
\end{aligned}
\end{equation*}
The five types of generators are presented in the picture below. 
\begin{figure}[H]
\centering
\includegraphics[scale=0.26]{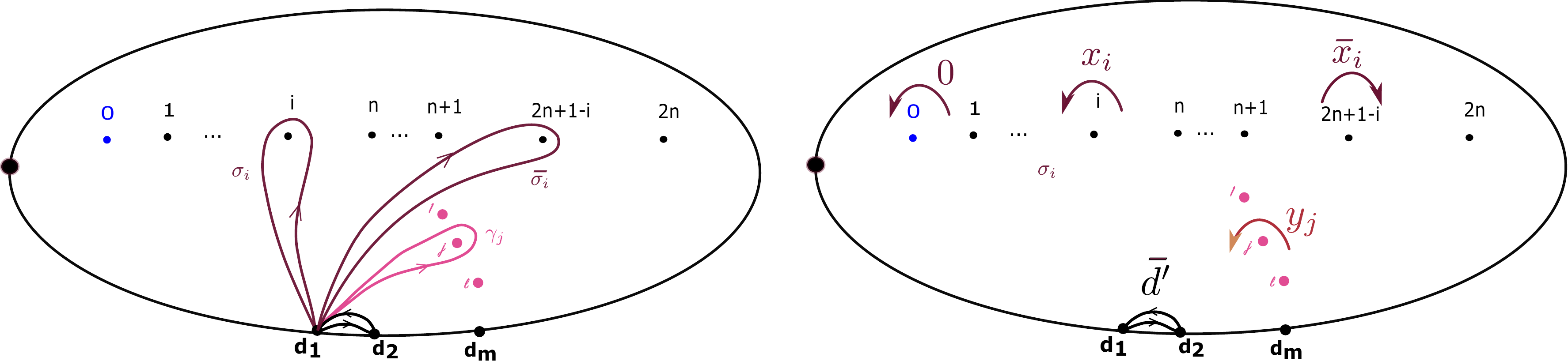}
\caption{\normalsize Local system}
\label{Localsystt}
\end{figure}
\end{prop}
\subsection{Local system and covering space at level $\bar{\cN}$}
\begin{defn}[Augmentation map] After this first step, we consider the following augmentation map $$\nu: H_1\left(\Clm\right)\rightarrow \Z^n \oplus \Z^n \oplus \Z^{\bl} \oplus \Z$$ 
$$ \hspace{28mm} \langle x_i \rangle \ \ \langle \bar{x_i} \rangle \ \ \langle y_j\rangle \ \ \langle \bar{d'} \rangle$$ defined by the formulas:
\begin{equation}
\begin{cases}
&\nu(\sigma_0)=0\\
&\nu(\sigma_i)=2x_i,\\ 
&\nu(\bar{\sigma}_i)=2\bar{x}_i, i\in \{1,...,n\}\\
&\nu(\gamma_j)=y_j, j\in \{1,...,\bl\}\\
&\nu(\delta)=\bar{d'}.
\end{cases}
\end{equation}
\end{defn}
\begin{defn}(The first local system)\label{localsystem0}
Let us consider the local system that is given by the composition of the two morphisms from above:
\begin{equation}
\begin{aligned}
&\Phi: \pi_1(\Clm) \rightarrow \Z^n \oplus \Z^n \oplus \Z^{\bar{l}} \oplus \Z\\
&\hspace{24mm} \langle x_i \rangle \ \ \langle \bar{x_i} \rangle \ \ \langle y_j\rangle \ \ \langle \bar{d'} \rangle, \ i \in \{1,...,n\}, \ j \in \{1,...,\bl\}\\
&\Phi= \nu \circ [ \ ]. \ \ \ \ \ \ \ \ \ \ \ \ \ \ \ \ \ \ \ \ 
\end{aligned}
\end{equation}
\end{defn}
In the next part we continue with a smaller local system, which will depend on a choice of a sequence of ``levels''. This will be used in order to make sure that our submanifolds which have geometric supports given by configuration spaces on circles lift to the covering associated to this local system. 
\begin{defn}[Multi-level]
Let us fix a sequence of levels $\cN_1,...,\cN_n \in \N$ and consider a ``multi-level'' associated to this set, given by:
\begin{equation}
\bar{\cN}:=(\cN_1,...,\cN_n).
\end{equation}
\end{defn}

\begin{defn}($\bar{\cN}$-change of variables)\label{N-ch}
For a multi-level $\bar{\cN}$, let us consider the morphism 
\begin{equation}
\begin{aligned}
f^{\bar{\cN}}: \ &\Z^n \oplus \Z^n \oplus \Z^{\bar{l}} \oplus \Z \rightarrow \Z^n \oplus  \Z^{\bar{l}} \oplus \Z\\
&\langle x_i \rangle \ \ \langle \bar{x_i} \rangle \ \ \langle y_j\rangle \ \ \langle \bar{d'} \rangle \hspace{2mm} \langle x_i \rangle \ \ \langle y_j\rangle \ \ \langle d' \rangle, \ i \in \{1,...,n\}, \ j \in \{1,...,\bl\}\\
\end{aligned}
\end{equation}
given by:
\begin{equation}
\begin{cases}
&f^{\bar{\cN}}(x_i)=x_i,\\    
&f^{\bar{\cN}}(\bar{x_i})=x_i+(\cN_i-1)d',\\ 
&f^{\bar{\cN}}(y_j)=y_j,\\ 
&f^{\bar{\cN}}(\bar{d'})=2d'.
\end{cases}
\end{equation}
\end{defn}

\begin{defn}($\bar{\cN}$-change of variables at the level of the group ring)\label{N-chh}
Let us define the morphism induced by $f^{\bar{\cN}}$ at the level of group rings:
\begin{equation}
\begin{aligned}
\tilde{f}^{\bar{\cN}}: \ &\Z[x_1^{\pm 1},...,x_n^{\pm 1},\bar{x}_1^{\pm 1},...,\bar{x}_n^{\pm 1},y_1^{\pm 1}..., y_{\bar{l}}^{\pm 1}, \bar{d'}^{\pm 1}]\rightarrow \Z[x_1^{\pm 1},...,x_n^{\pm 1},y_1^{\pm 1}..., y_{\bar{l}}^{\pm 1}, d'^{\pm 1}]\\
\end{aligned}
\end{equation}
given by:
\begin{equation}
\begin{cases}
&\tilde{f}^{\bar{\cN}}(x_i)=x_i,\\    
&\tilde{f}^{\bar{\cN}}(\bar{x_i})=x_i\cdot d'^{(\cN_i-1)},\\ 
&\tilde{f}^{\bar{\cN}}(y_j)=y_j,\\ 
&\tilde{f}^{\bar{\cN}}(\bar{d'})=d'^{2}.
\end{cases}
\end{equation}

\end{defn}
\begin{defn}(Local system associated to a multi-level $\bar{\cN}$)\label{localsystem}
Let us consider the local system that is given by the composition of the two morphisms from above:
\begin{equation}
\begin{aligned}
&\Phi^{\bar{\cN}}: \pi_1(\Clm) \rightarrow \Z^n \oplus \Z^{\bar{l}} \oplus \Z\\
&\hspace{26mm} \langle x_i \rangle \ \ \langle y_j\rangle \ \ \langle d' \rangle, \ i \in \{1,...,n\}, \ j \in \{1,...,\bl\}\\
&\Phi^{\bar{\cN}}= f^{\bar{\cN}} \circ \Phi. \ \ \ \ \ \ \ \ \ \ \ \ \ \ \ \ \ \ \ \ 
\end{aligned}
\end{equation}
\end{defn}

\begin{defn}[Covering space at multi-level $\bar{\cN}$]\label{localsystemc}
Let $\ClmN$ be the covering of the configuration space $\Clm$ which is associated to the level $\bar{\cN}$ local system $\Phi^{\bar{\cN}}$.
\end{defn}
\begin{notation}[Base point]

\

For the next steps we also fix a base point $\tilde{{\bf d}}$ which belongs to the fiber over ${\bf d}$ in $\ClmN$.
\end{notation}

\subsection{Passing from groups to group rings} We will work with the homologies of the level $\bar{\cN}$ covering space. More precisely, our tools will use a Poincaré-Lefschetz duality between two homology groups of $\ClmN$. In order to define this duality, we pass from groups to group rings of coefficients, as follows. 

We do this in two steps. First, we will enlarge the group associated to the image of the local system $\Phi^{\bar{\cN}}$. Secondly, we will use a change of variables which will play an important role for computations. 

We remark that the group of deck transformations of the covering space $\ClmN$ is:
$$\Imm\left(\Phi^{\bar{\cN}}\right)=(2\Z)^{n}\oplus \Z^{\bar{l}}\oplus (2\Z) \subseteq \Z^{n}\oplus \Z^{\bar{l}}\oplus \Z.$$
This means that the homology of this covering space $\ClmN$ is a module over the associated group ring:
$$\Z[x_1^{\pm 2},...,x_n^{\pm 2},y_1^{\pm 1}..., y_{\bar{l}}^{\pm 1}, d'^{\pm 2}].$$ 
\begin{defn}[Inclusion of group rings] Let us denote the following inclusion map:
$$\iota: \Z[x_1^{\pm 2},...,x_n^{\pm 2},y_1^{\pm 1}..., y_{\bar{l}}^{\pm 1}, d'^{\pm 4}] \subseteq \Z[x_1^{\pm 1},...,x_n^{\pm 1},y_1^{\pm 1}..., y_{\bar{l}}^{\pm 1}, d'^{\pm 1}].$$
\end{defn}

We consider the homology of the level $\bar{\cN}$ covering tensored over $\iota$ with the group ring $$\Z[x_1^{\pm 1},...,x_n^{\pm 1},y_1^{\pm 1}..., y_{\bar{l}}^{\pm 1}, d'^{\pm 1}].$$
Using this change of coefficients, we have homology groups which become modules over: 
$$\Z[x_1^{\pm 1},...,x_n^{\pm 1},y_1^{\pm 1}..., y_{\bar{l}}^{\pm 1}, d'^{\pm 1}].$$
For the next part, we want to change further the variables by multiplying with the $4^{th}$ root of unity  $\xi_2=i=e^{\frac{2\pi i}{4}}$. This will play an important role in the actual computations, which we explain below.
\begin{notation}[Variables used for computations encoding orientations]
Let us consider a new variable:
\begin{equation}
d:=\xi_2d'.
\end{equation}
 Accordingly, via the function $\tilde{f}^{\bar{\cN}}$, this notation corresponds to the operation given by rising to the $2^{nd}$ power, so this corresponds to the change of variables:
\begin{equation}
\bar{d}:=-\bar{d'}.
\end{equation} 
From now we will use this variable, having in mind certain computations of the intersection pairing. 
\end{notation}

{\bf Explanation} More specifically, our intersection pairing will be computed for the situation where the homology classes are given by submanifolds in the configuration space in the punctured disc. For the actual computation we will use a property that allows us to encode the sign of the geometric intersections in the configuration space by signs of intersections in the punctured disc, by replacing $\bar{d'}$ by the variable $-\bar{d'}$ (which we denote by $\bar{d}$, and we refer to \cite[Remark 3.4.3]{Crsym} and \cite[Section 3]{Big} for a precise explanation of the sign formula). Correspondingly, after passing from $\bar{d'}$ to $\bar{d}$ this means to replace $d'$ by $$d:=-id'$$
in the computations for the intersection pairing which we will do in the next sections. 
\begin{rmk}[Computing the intersection pairing--relative twisting]\label{rktw}

\

With these notations, when computing the intersection pairing for the two topological models from the next sections (see relations \eqref{d1}, \eqref{d1}) we will count $d^2$ for the contribution of the relative twisting.

\end{rmk}
This motivates the next change of variables that we introduce below. 
\begin{defn}[Change of variables used for orientation purposes] \label{groupring'} Let us consider:
\begin{equation} 
\begin{aligned}
&\eta: \Z[x_1^{\pm 1},...,x_n^{\pm 1},y_1^{\pm 1}..., y_{\bar{l}}^{\pm 1}, d'^{\pm 1}]\rightarrow \C[x_1^{\pm 1},...,x_n^{\pm 1},y_1^{\pm 1}..., y_{\bar{l}}^{\pm 1}, d^{\pm 1}]\\
&\begin{cases}
\eta(x_i)=x_i\\
\eta(y_j)=y_j\\
\eta(d')=id.
\end{cases}
\end{aligned}
\end{equation}
\end{defn}

\begin{rmk}[Structure of the homology of the level $\cN$ covering space]\label{homcov}
Using this change of coefficients, the homology groups of the level $\cN$ covering space $\ClmN$ become modules over the following ring:
\begin{equation} \label{groupring}
 \C[x_1^{\pm 1},...,x_n^{\pm 1},y_1^{\pm 1}..., y_{\bar{l}}^{\pm 1}, d^{\pm 1}].
\end{equation}
\end{rmk}

\subsection{Homology groups associated to a level $\cN$}\label{hlgy}
Our homological tools will use the relative homology of the level $\bar{\cN}$ covering space in the middle dimension. The definitions from this subsection rely on the construction of certain homology groups coming from a particular choice of splitting the boundary of coverings of configuration spaces from \cite{CrM}. However, as we will see in the next subsection, we will use a very concrete way of computing intersection pairings in these homologies, which we explain in detail in Subsection \ref{comp}. 

\begin{notation}
a) Let $S^{-}\subseteq \partial \mathscr D_{2n+1,\bl}$ be the semicircle on the boundary of the punctured disc containing the points with negative $x$-coordinate. We also fix a point on the boundary of the disc, denoted:
 $$w \in S^{-} \subseteq \partial \mathscr D_{2n+1,\bl}.$$
b) Let us denote by $C^{-}$ the subspace in the boundary of the configuration space $\Clm$  given by configurations where at least one point is in the set $S^{-}$. 

c) Also, let $P^{-} \subseteq \partial \ClmN$ be the part of the boundary of level $\bar{\cN}$-covering given by the fiber over $C^{-}$.

\end{notation}

In the next part, we consider two homology groups which will involve submodules in the Borel-Moore homology of the covering space $\ClmN$, given by the twisted Borel-Moore homology of the base space $\Clm$ (twisted by the local system $\Phi^{\bar{\cN}}$).

\begin{defn}
\label{T2}
For the definition of the two homology groups, we split the infinity part of the configuration space in two parts. The precise procedure for this construction is described in \cite[Remark 7.5]{CrM} (for the case of the $2$-dimensional disc minus half of its boundary and minus $2n+{\bar{l}}+1$ open discs with pairwise disjoint closures).

Using this splitting, we define two homology groups as follows.

\begin{equation*}
\begin{aligned}
& \bullet \text{ We consider the homology group } H^{\text{lf},\infty,-}_m(\ClmN,P^{-}; \Z) \text{ which is given}\\
& \text{by the homology relative to part of the infinity that is the open boundary of}\\
& \ClmN \text{ consisting in configurations that project to a multipoint in } \Clm\\
&  \text{ that touches a puncture from } \mathscr D_{2n+1,\bl} \text { and also relative to the boundary } P^{-}.\\
&\bullet \text{Also, let us define } H^{lf, \Delta}_{m}(\ClmN, \partial; \Z) \text{ to be the homolgy relative to }\\
& \text{ the boundary of the covering space which is not in } P^{-} \text{and Borel-Moore}\\
& \text{ with respect to collisions of particles from the configuration space}.
\end{aligned}
\end{equation*}
\end{defn}
\begin{rmk}
 In the next part we will see that the Borel-Moore homology of a covering space is different from the twisted Borel-Moore homology of the base space.
For our construction, we will use the homology groups of the covering space rather than the twisted homology of the base space. For this purpose, we use the following properties. 
\end{rmk}
\begin{prop}[\cite{CrM}, Theorem E]\label{P:5}
Let $\mathscr L_{\Phi^{\bar{\cN}}}$ be the rank $1$ local system associated to $\Phi^{\bar{\cN}}$ (as in \cite[Definition 2.7]{CrM}). Then there are natural injective maps between the following homologies:
\begin{equation}
\begin{aligned}
& \iota: H^{\text{lf},\infty,-}_m(\Clm, C^{-}; \mathscr L_{\Phi^{\bar{\cN}}})\rightarrow H^{\text{lf},\infty,-}_m(\ClmN,P^{-}; \Z)\\
& \iota^{\partial}:H^{\text{lf},\Delta}_m(\Clm, \partial; \mathscr L_{\Phi^{\bar{\cN}}})\rightarrow H^{lf, \Delta}_{m}(\ClmN, \partial; \Z).
\end{aligned}
\end{equation}
\end{prop}
We will use the images of the twisted homologies of the base configuration space from Proposition \ref{P:5}, seen in the homologies of the level $\bar{\cN}$ covering $\ClmN$. 
\begin{defn}[Homology of the level $\bar{\cN}$ covering]

We denote the submodules in the homologies of the level $\bar{\cN}$ covering space that are the images of the inclusions $\iota$ and $\iota^{\partial}$ respectively:
\begin{enumerate}
 \item[$\bullet$]  $\mathscr H^{\bar{\cN}}_{n,m,{\bar{l}}}\subseteq H^{\text{lf},\infty,-}_m(\ClmN, P^{-1};\Z)$ and 
 \item[$\bullet$]  $\mathscr H^{\bar{\cN},\partial}_{n,m,{\bar{l}}} \subseteq H^{\text{lf},\Delta}_m(\ClmN,\partial;\Z)$. 
\end{enumerate}
As we have seen in \eqref{groupring}, all these homologies are $\C[x_1^{\pm 1},...,x_n^{\pm 1},y_1^{\pm 1}..., y_{\bar{l}}^{\pm 1}, d^{\pm 1}]$-modules.
\end{defn}

These two homologies are related by a geometric intersection pairing that comes from a  Poincaré-Lefschetz type duality for twisted homology (see \cite[Proposition 3.2]{CrM}] and \cite[Lemma 3.3]{CrM}).
\begin{prop}(\cite[Proposition 7.6]{CrM})\label{P:3'''}
There is a topological intersection pairing between the following homology groups:
$$\ll ~,~ \gg: \mathscr H^{\bar{\cN}}_{n,m,{\bar{l}}} \otimes \mathscr H^{\bar{\cN},\partial}_{n,m,{\bar{l}}} \rightarrow\C[x_1^{\pm 1},...,x_n^{\pm 1},y_1^{\pm 1}..., y_{\bar{l}}^{\pm 1}, d^{\pm 1}].$$
\end{prop}
\subsection{Computation of the geometric intersection pairing}\label{comp}
We will need to perform precise computations of this intersection pairing, so in this subsection we sketch the main steps needed for its formula (that is described in \cite[Section 7]{CrM}). We will work with the coefficients that belong to the group ring from equation \eqref{groupring}. For this, we introduce the following notation.
\begin{notation}[Passing to the group ring and encode orientations]\phantom{A}\\ 
Let $\tilde{\Phi}^{\bar{\cN}}$ be the morphism induced by the level $\bar{\cN}$ local system $\Phi^{\bar{\cN}}$, that takes values in the group ring of $\Z^n \oplus \Z^{\bar{l}} \oplus \Z$:
\begin{equation}\label{sgr}
\tilde{\Phi}^{\bar{\cN}}: \pi_1(\Clm) \rightarrow \C[x_1^{\pm 1},...,x_n^{\pm 1},y_1^{\pm 1}..., y_{\bar{l}}^{\pm 1}, d'^{\pm 1}].
\end{equation}
Then, using the change of variables $\eta$, we consider:
\begin{equation}\label{sor}
\begin{aligned}
&\bar{\Phi}^{\bar{\cN}}: \pi_1(\Clm) \rightarrow \C[x_1^{\pm 1},...,x_n^{\pm 1},y_1^{\pm 1}..., y_{\bar{l}}^{\pm 1}, d^{\pm 1}]\\
&\bar{\Phi}^{\bar{\cN}}=\eta \circ \tilde{\Phi}^{\bar{\cN}}. \ \ \ \ \ \ \ \ \ \ \ \ \ \ \ \ \ \ \ \ 
\end{aligned}
\end{equation}
\end{notation}
In the next part we present the formula for the intersection pairing, which will use this morphism. 
Let $H_1 \in H^{\bar{\cN}}_{n,m,{\bar{l}}}$ and $H_2 \in H^{\bar{\cN},\partial}_{n,m,{\bar{l}}}$ be two homology classes. Moreover, let us suppose that these two classes are given by two lifts $\tilde{X}_1, \tilde{X}_2$ of immersed submanifolds in the base space, denoted $X_1,X_2 \subseteq \Clm$. We assume that $X_1$ and $X_2$ intersect transversely, in a finite number of points. 
  
The intersection pairing will be computed following two steps. 

{\bf 1) Loop associated to an intersection point} First, we associate to each intersection point $x \in X_1 \cap X_2$ a loop in the configuration space $l_x \subseteq \Clm$. After that, we will grade this via the local system and the morphism $\bar{\Phi}^{\bar{\cN}}$.
  
\begin{defn}[Construction of $l_x$]
For $x \in X_1 \cap X_2$, we suppose that we have the paths $\gamma_{X_1}, \gamma_{X_2}$ starting in $\bf d$, ending on $X_1$,$X_2$ respectively such that  
$\tilde{\gamma}_{X_1}(1) \in \tilde{X}_1$ and $ \tilde{\gamma}_{X_2}(1) \in \tilde{X}_2$. Here, $\tilde{\gamma}_{X_i}$ is the unique lift of ${\gamma}_{X_i}$ through $\tilde{\bf d}$.
Let us choose $\nu_{X_1}, \nu_{X_2}:[0,1]\rightarrow \Clm$ two paths with the properties:
\begin{equation}
\begin{cases}
Im(\nu_{X_1})\subseteq X_1; \nu_{X_1}(0)=\gamma_{X_1}(1);  \nu_{X_1}(1)=x\\
Im(\nu_{X_2})\subseteq X_2; \nu_{X_2}(0)=\gamma_{X_2}(1);  \nu_{x_2}(1)=x.
\end{cases}
\end{equation}
The composition of these paths gives us the loop:
$$l_x=\gamma_{X_1}\circ\nu_{X_1}\circ \nu_{X_2}^{-1}\circ \gamma_{X_2}^{-1}.$$
\end{defn}
{\bf 2) Grade the family of loops using the local system}
\begin{prop}[Intersection pairing from graded intersections in the base space]\label{P:3}
The intersection pairing can be computed from the set of loops $l_x$ and the local system:
\begin{equation}\label{eq:1}  
\ll H_1,H_2 \gg= \eta \left( \sum_{x \in X_1 \cap X_2}  \alpha_x \cdot \Phi^{\bar{\cN}}(l_x) \right) \in \C[x_1^{\pm 1},...,x_n^{\pm 1},y_1^{\pm 1}..., y_{\bar{l}}^{\pm 1}, d^{\pm 1}]
\end{equation}
where $\alpha_x$ is the sign of the geometric intersection between the submanifolds $M_1$ and $M_2$ at the point $x$, in the configuration space $\Clm$.

\end{prop}
\begin{rmk}[Submanifolds coming from products of manifolds of dimension one] \label{orientd}

If the homology classes come from products of one dimensional submanifolds in the punctured disc, quotiented to the unordered configuration space, we can compute directly the intersection $\ll H_1,H_2 \gg$ from the sum presented in equation \eqref{eq:1} without the change of coefficients given by $\eta$.

In this situation we can replace the local system $\Phi^{\bar{\cN}}$ by $\bar{\Phi}^{\bar{\cN}}$ in the previous formula, counting the sign contribution given just by the product of local orientations in the disc around each component of the intersection point $x$ (instead of computing the orientation in the configuration space given by $\alpha_x$).

We will use this type of computations in Section \ref{S:J} and Section \ref{S:A}.
\end{rmk}

\subsection{Specialisations given by colorings}
So far, the construction of the two homology groups is intrinsec and does not depend on a braid representative that gives a link. In this part we will prepare our homological tools for the situation where we have a braid with $n$-strands that gives a link with $l$ components by braid closure. This will induce a colouring, as below. 
\begin{defn}[Colouring of the punctures $C$]
Let $C$ be a coloring of the $2n$ $p$-punctures of the disc $\{1,...,2n\}$ by $l$ colours, as below:
\begin{equation}
C:\{1,...,2n\}\rightarrow \{1,...,l\}.
\end{equation} 
\end{defn}

There are two special situations that occur in our models, namely ${\bar{l}}=1$ or ${\bar{l}}=l$. In this paper, we have the case ${\bar{l}}=1$. In the sequel paper for $3$-manifold invariants we will use ${\bar{l}}=l$. 

We fix the last $p$-puncture and denote it by $\bar{p}=2n$. We will position the $\bar{l}=1$ $s$-puncture underneath this special $p$-puncture labeled by $\bar{p}$. Using this, we define the colour of the $s$-punctures as being the same as the colour of $\bar{p}$ (which is induced by the colouring $C$).


\begin{defn}[Change of coefficients $f_{C}$]
We change the variables associated to the punctures of the punctured disc using the colouring $C$. More specifically, let us change the first $n+{\bar{l}}$ variables $x_1,...,x_n,y_1,...,y_{\bar{l}}$ from the ring $\C[x_1^{\pm 1},...,x_n^{\pm 1},y_1^{\pm 1}..., y_{\bar{l}}^{\pm 1}, d^{\pm 1}]$ to $l+{\bar{l}}$ variables, which we denote by $x_1,..,x_l,y_1,...,y_{\bar{l}}$,  as below:
$$ f_C: \C[x_1^{\pm 1},...,x_n^{\pm 1},y_1^{\pm 1}..., y_{\bar{l}}^{\pm 1}, d^{\pm 1}] \rightarrow \C[x_1^{\pm 1},...,x_l^{\pm 1},y_1^{\pm 1}..., y_{\bar{l}}^{\pm 1}, d^{\pm 1}]$$
\begin{equation}\label{eq:8} 
\begin{cases}
&f_C(x_i)=x_{C(i)}, \ i\in \{1,...,n\}\\
&f_C(y_j)=y_{C(\bar{p}_j)}, \ j\in \{1,...,{\bar{l}}\}.
\end{cases}
\end{equation}
\end{defn}
Now, we look at this change of coefficients at the level of the homology groups, via the function $f_C$.
\begin{defn}(Homology groups)\label{D:4} We consider the two homologies over the ring associated to the new coefficients:
\begin{enumerate}
 \item[$\bullet$]  $H^{\bar{\cN}}_{n,m,{\bar{l}}}:=\mathscr H^{\bar{\cN}}_{n,m,{\bar{l}}}|_{f_C}$ 
 \item[$\bullet$]  $H^{\bar{\cN},\partial}_{n,m,{\bar{l}}}:=\mathscr H^{\bar{\cN},\partial}_{n,m,{\bar{l}}}|_{f_C}.$
\end{enumerate}
We remark that these homology groups are modules over $\C[x_1^{\pm 1},...,x_l^{\pm 1},y_1^{\pm 1}..., y_{\bar{l}}^{\pm 1}, d^{\pm 1}]$.
\end{defn}
\begin{notation}[Cases when we remove $\bar{l}$ from notations]
Since for the models for link invariants we have $\bar{l}=1$, we remove this from the notation of the homology groups, and also replace $y_1$ by $y$ as variable, and use the following notations.
\end{notation}
\begin{notation}[Homology groups for link invariants] 
We have defined above the homologies of the level $\bar{\cN}$ covering $\ClmN$, which we denote by:
$$H^{\bar{\cN}}_{n,m}, H^{\bar{\cN}, \partial}_{n,m}$$ and which are modules over the ring
$\C[x_1^{\pm 1},...,x_l^{\pm 1},y^{\pm 1}, d^{\pm 1}]$.
\end{notation}
\clearpage
\subsection{Diagram with the specialisations of coefficients}\phantom{A}\\
As a summary, we have the following specialisations of coefficients, from Figure \ref{diags}.
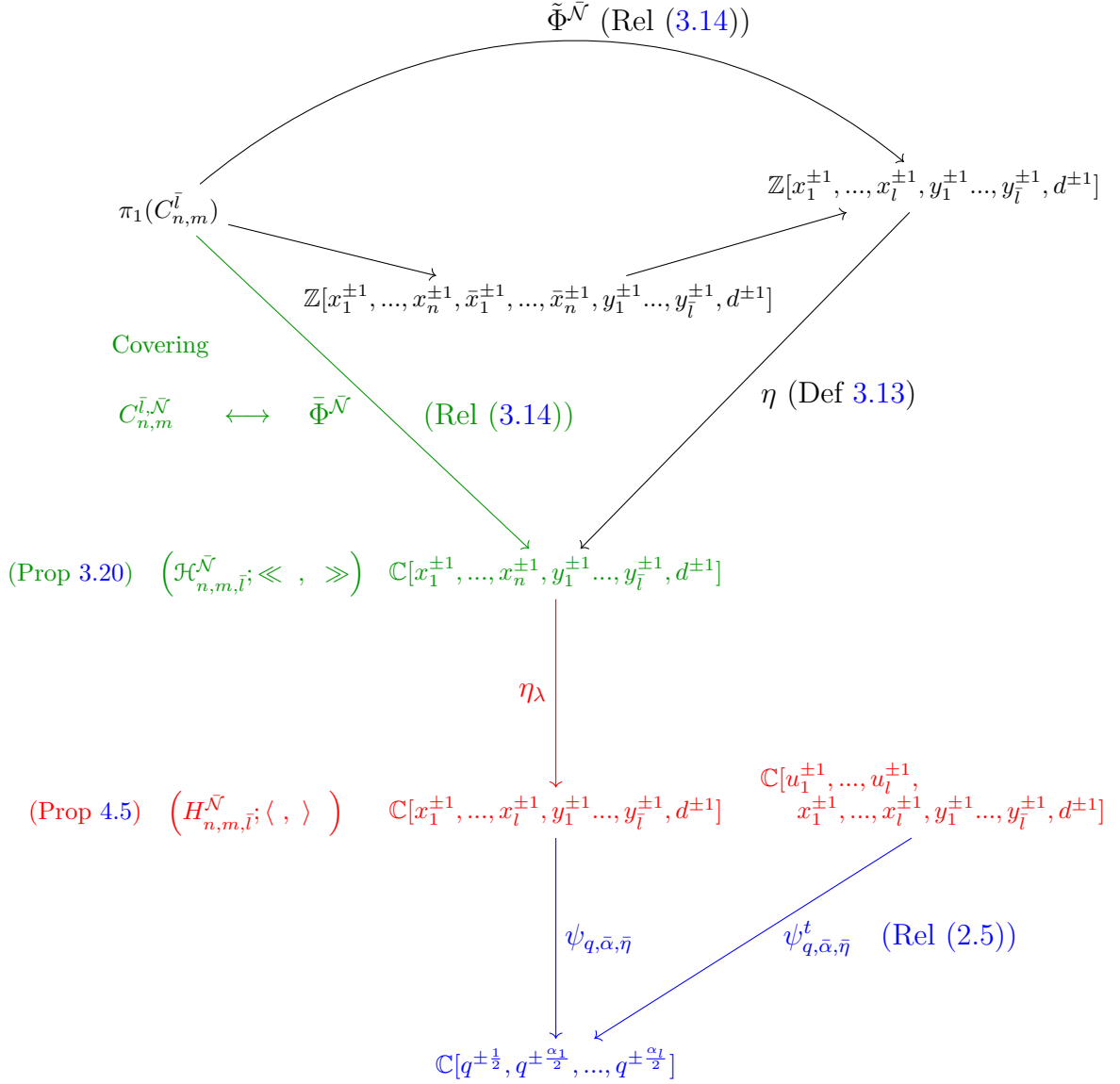
\begin{figure}[H] 
\begin{center}
\begin{tikzpicture}
[x=1.2mm,y=1.6mm]

\node (b1)               at (5,50)    {$\pi_1(\Clm)$};
\node (b20) [color=dgreen] at (4,38)   {Covering};
\node (b200) [color=dgreen] at (8,32)   {$\ClmN$ $ \ \ \ \ \ \longleftrightarrow$};
\node (b2) [color=dgreen] at (50,18)   {$\C[x_1^{\pm 1},...,x_n^{\pm 1},y_1^{\pm 1}..., y_{\bar{l}}^{\pm 1}, d^{\pm 1}]$};
\node (b22) [color=dgreen] at (7,18)   {(Prop \ref{P:3'''}) \ \ $\left(\mathscr H^{\bar{\cN}}_{n,m,{\bar{l}}}; \ll ~,~ \gg  \right)$};
\node (b3) [color=red]   at (50,-3)   {$\C[x_1^{\pm 1},...,x_l^{\pm 1},y_1^{\pm 1}..., y_{\bar{l}}^{\pm 1}, d^{\pm 1}]$};
\node (b33) [color=red] at (7,-3)   {(Prop \ref{P:3'}) \ \ $\left( H^{\bar{\cN}}_{n,m,{\bar{l}}}; \left\langle ~,~ \right\rangle \ \ \right)$};
\node (b3') [color=red]   at (83,0)   {$\C[u_1^{\pm 1},...,u_l^{\pm 1},$};
\node (b33') [color=red]   at (96,-3)   {$x_1^{\pm 1},...,x_l^{\pm 1},y_1^{\pm 1}..., y_{\bar{l}}^{\pm 1}, d^{\pm 1}]$};
\node (b4) [color=blue]  at (50,-25)   {$\C[q^{\pm \frac{1}{2}},q^{\pm \frac{\alpha_1}{2}},...,q^{\pm \frac{\alpha_l}{2}}]$};
\node (b1') [color=black]  at (48,42)   {$\Z[x_1^{\pm 1},...,x_n^{\pm 1},\bar{x}_1^{\pm 1},...,\bar{x}_n^{\pm 1},y_1^{\pm 1}..., y_{\bar{l}}^{\pm 1}, d^{\pm 1}]$};
\node (b1'') [color=black]  at (94,52)   {$\Z[x_1^{\pm 1},...,x_l^{\pm 1},y_1^{\pm 1}..., y_{\bar{l}}^{\pm 1}, d^{\pm 1}]$};

\draw[->,color=dgreen]   (b1)      to node[yshift=-3mm,xshift=11mm,font=\large]{$\bar{\Phi}^{\bar{\cN}} \hspace{10mm} (\text{Rel } \eqref{sgr})$}                           (b2);
\draw[->,color=red]             (b2)      to node[left,font=\large]{$\eta_{\lambda} $}   (b3);
\draw[->,color=black]             (b1')      to node[left,font=\large]{}   (b1'');
\draw[->,color=blue]             (b3)      to node[right,font=\large]{$\psi_{q,\bar{\alpha},\bar{\eta}}$}   (b4);
\draw[->,color=blue]             (b33')      to node[right,font=\large,xshift=3mm]{$\psi^t_{q,\bar{\alpha},\bar{\eta}}$\ \ \ (Rel \eqref{not})}   (b4);
\draw[->,color=black]   (b1)      to node[right,xshift=-2mm,yshift=3mm,font=\large]{}                        (b1');
\draw[->,color=black]   (b1)      to [out=40,in=140] node[right,xshift=-2mm,yshift=3mm,font=\large]{$\tilde{\Phi}^{\bar{\cN}}$ (Rel \eqref{sgr})}                        (b1'');
\draw[->,color=black]   (b1'')      to node[right,yshift=-2mm,xshift=1mm,font=\large]{$\eta$ (Def \ref{groupring'})}                        (b2);
\end{tikzpicture}
\end{center}
\vspace{-3mm}
\caption{\normalsize Specialisations of coefficients: General set-up}\label{diags}
\end{figure}

\subsection{General procedure for encoding homology classes in the base configuration space}
\begin{notation}[Geometric supports and paths to the base points]\label{paths}

\

In all constructions from now on we will use homology classes in a covering of the configuration space that are prescribed by the following data:

\begin{itemize}
\item[•] A {\em geometric support}, meaning a {\em set of arcs in the punctured disc} or a set of {\em circles in the punctured disc} on which we consider unordered configurations of a prescribed number of particles. The image of the product of these arcs or configurations on circles in the configuration space, gives us a submanifold $F$ which has half of the dimension of the configuration space. 
\item[•] A collection of {\em paths connected to the base point}, starting in the base points from the punctured disc and ending on these curves or circles. Then, the collection of these paths gives a path in the configuration space, starting in $\bf d$ and ending on the submanifold $F$. 
\end{itemize}

\

\

Let us suppose that the submanifold $F$ has a well defined lift to the covering.
Then, we lift the path to a path in the covering space, starting from $\tilde{\bf{d}}$ and after that we lift the submanifold through the end point of this path. The detailed construction of homology classes via this dictionary is presented in \cite[Section 5]{Crsym}. 
\end{notation}
We will see explicit examples of such classes in the next sections.

\subsection{Lift of submanifolds with support on configuration spaces on circles}
\label{S:4}

In this part we show that if we consider a submanifold which has as geometric support a collection of configuration spaces on symmetric circles, whose multiplicities are prescribed by the multi-level $\bar{\cN}$, then this submanifold has a well-defined lift in the covering associated to the level $\bar{\cN}$ local system.
More specifically, let us fix the following parameters:
$$ n \rightarrow n; \ \ \ m\rightarrow 1+\sum_{i=1}^{n} \cN_i; \ \ \ {\bar{l}}; \ \ \ \bar{\cN}.$$ 
This means that we work in the configuration space of $1+\sum_{i=1}^{n} \cN_i$ particles in the $(2n+2)$-punctured disc:
$$\Conf_{1+\sum_{i=1}^{n} \cN_i}\left(\mathbb D_{2n+2}\right)$$
After this, let us consider the local system $\Phi^{\bar{\cN}}$ associated to the these parameters.

We will use the homology groups associated to this data:
$$\mathscr H^{\bar{\cN}}_{n, 1+\sum_{i=1}^{n} \cN_i,\bar{l}} \ \ \ \ \ \ \ \ \ \ \ \ \ \ \ \text{ and }\ \ \ \ \ \ \ \mathscr H^{\bar{\cN},\partial}_{n,1+\sum_{i=1}^{n} \cN_i,\bar{l}}.$$
\begin{rmk} The following property holds for any $\bar{l}$, even if for the later sections we use $\bar{l}=1$.
\end{rmk}
\begin{defn} (Level $\bar{\cN}$ submanifold based on ovals)\label{gendsupp}\\
Let $L(\bar{\cN})$ be the submanifold in the configuration space given by the product of configuration spaces on symmetric ovals, with multiplicities as in Figure \ref{Pictureci} and the circle which goes around the puncture labeled by $0$:
 \begin{figure}[H]
\centering
$$\normalsize \tilde{L}(\bar{\cN})$$
$$\hspace{5mm}\downarrow \text{ lifts }$$
\vspace{-2mm}

\includegraphics[scale=0.5]{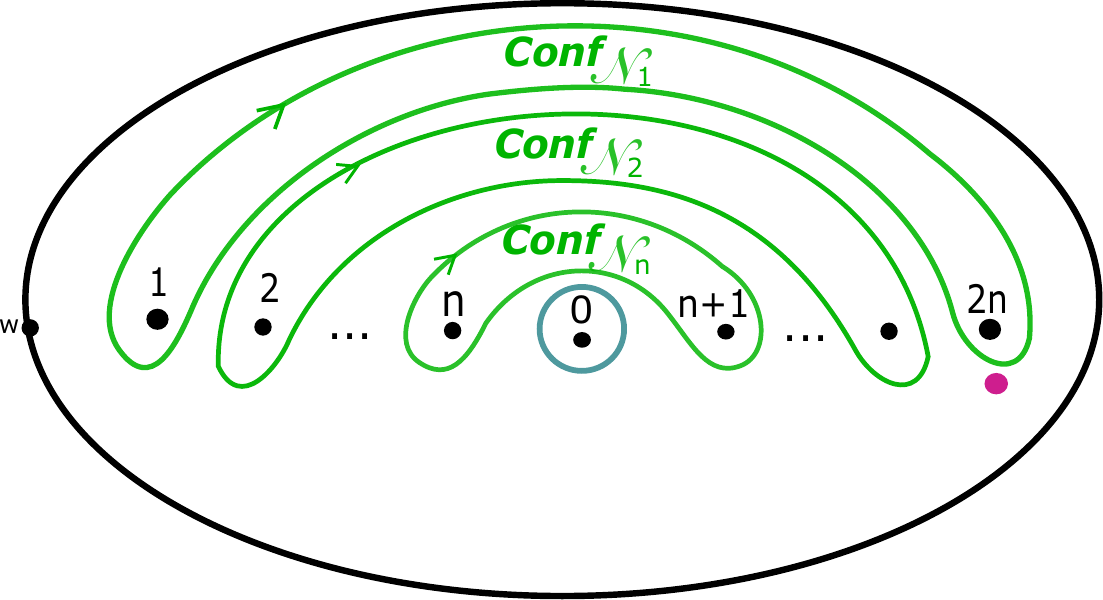}
\vspace*{-20mm}

$$\normalsize L(\bar{\cN})$$
\vspace*{1mm}

\vspace*{-8mm}

$\hspace{115mm} \Conf_{1+\sum_{i=1}^{n} \cN_i}\left(\mathbb D_{2n+2}\right)$

\vspace*{1mm}

\caption{\normalsize Lifting Lagrangians supported by configurations on ovals}\label{Pictureci}
\end{figure}
\end{defn}

\begin{lem}[Lifting submanifolds given by configurations on ovals]\label{gendsuppl} The geometric support based on ovals with multiplicities given by the level $\bar{\cN}$, shown in the Figure \ref{Pictureci}, leads to a well-defined class in the homology $\mathscr H^{\bar{\cN},\partial}_{n,1+\sum_{i=1}^{n} \cN_i,\bar{l}}$ of the level $\bar{\cN}$-covering space $\ClmN$.

\end{lem}
\begin{figure}[H]
\centering
\hspace{-5mm}\includegraphics[scale=0.37]{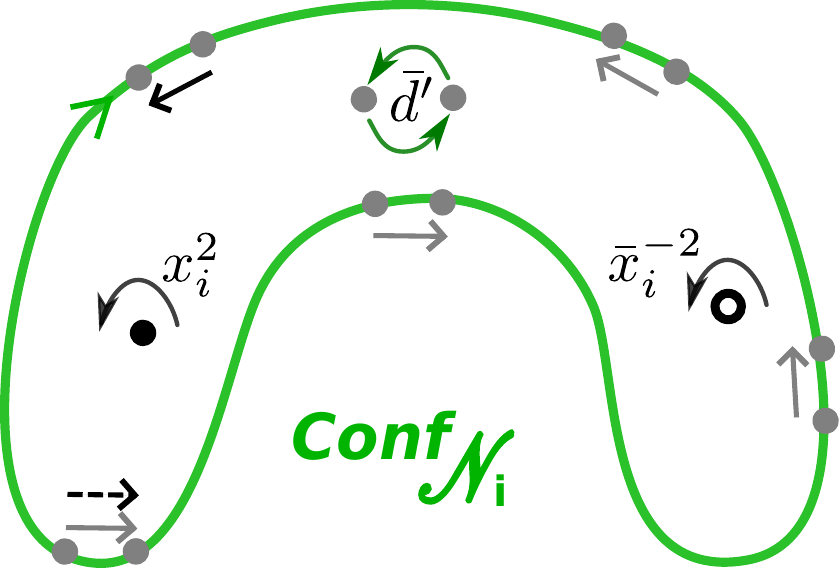}
\caption{\normalsize Monodromy of configurations on a circle}
\label{Picturemon}
\end{figure}
\begin{rmk}
In Figure \ref{Picturemon} we use the local system after we pass to the group ring, so all notations are multiplicative. For the purpose of clarity of signs, in the following proof we will use the additive notation, as below.
\end{rmk}
\begin{proof}
In order to prove that the geometric support given by configuration spaces on circles lifts to the covering, we will check that the total monodromy on each circle gets evaluated to zero by the local system $\bar{\Phi}^{\cN}$.

 More precisely, for a fixed $i\in \{1,...,n\}$, the local system $\Phi$ counts the monodromy around the symmetric points $(i,2n+1-i)$ with variables $2x_{i}$ and $2\bar{x}_{i}$. 
Now, let us look at a loop given by $\cN_i$ particles that go around a circle, as in Figure \ref{Picturemon}. The local system $\Phi$ (introduced in Definition \ref{localsystem0}) will count the total monodromy for this circle by the following expression:  
\begin{equation}\label{m1}
\begin{aligned}
&2\cN_i\cdot x_i-2\cN_i\cdot \bar{x}_i+\frac{\cN_i(\cN_i-1)}{2}(2\bar{d'})=\\
&=2\cN_i\cdot x_i-2\cN_i\cdot \bar{x}_i+\cN_i(\cN_i-1)\bar{d'}.
\end{aligned}
\end{equation} 
 
Here, $2\cN_i\cdot x_i$ comes from the rotation around the puncture $i$, $2\cN_i\cdot \bar{x}_i$ from the rotation around the symmetric puncture $2n+1-i$ and $\frac{\cN_i(\cN_i-1)}{2}(2\bar{d'})$ from the relative winding of the $\cN_i$ particles along the circle. 

We remark that this monodromy does not wanish, so the local system $\Phi$ is not enough to give us well defined homology classes.

Further, we look at the level $\bar{\cN}$ covering space and its homology groups. For this, we have to change the variables using the function $f^{\bar{\cN}}$ (using Definition \ref{N-chh} and Definition \ref{localsystem}). If we do so, we change the variables in the following way:

\begin{equation}
\begin{cases}
&x_i \rightarrow x_{i},\\    
&\bar{x_i} \rightarrow x_{i}+(\cN_i-1)d',\\ 
&\bar{d'}\rightarrow 2d'.
\end{cases}
\end{equation}
Then, the monodromy of the loop, expressed in relation \eqref{m1} gets evaluated to the following:
\begin{equation}
\begin{aligned}
&2\cN_i\cdot x_i+2\cN_i\cdot \bar{x}_i+\cN_i(\cN_i-1)\bar{d'} \rightarrow\\
 \rightarrow \ & 2\cN_i\cdot x_{i}-2\cN_i \cdot (x_{i}+(\cN_i-1)d')+\cN_i(\cN_i-1)2d'\\
& \ \ \ \ \ \ \ \ \ \ \ \ \ \ \ \ \ \ \ \ \ \ \ \ \ =0.
\end{aligned}
\end{equation} 

This means that the configuration space on a circle around these symmetric punctures lifts to a well-defined submanifold in the level $\bar{\cN}$ covering. 
On the other hand, the local system $\Phi$ is chosen to have trivial monodromy around the puncture labeled by $0$, so the circle which goes around this puncture has a well-defined lift.  

Then, we do this for all configuration spaces on the $n$ circles and the circle which goes around the puncture $0$ and remark that the total monodromy of such a loop vanishes. This shows that we have a well-defined lift that gives a submanifold $\tilde{L}(\bar{\cN})$ in the $\bar{\cN}-$covering space. Further on, if we fix a base point for lifting it, then it leads to a well-defined homology class in: 
$$\mathscr H^{\bar{\cN},\partial}_{n, 1+\sum_{i=1}^{n} \cN_i,\bar{l}}.$$This concludes the lifting property at the level $\cN$ covering space.
\end{proof}

\subsection{Braid group action}
We will use the set of braids which have $n+1+{\bar{l}}$ strands, whose induced action on the punctured disc preserves the colouring of the punctures (induced by $C$). We denote the set of such braids by $B^{C}_{n+{\bar{l}}+1}$.
\begin{prop}[\cite{CrM}] \label{colbr} 
The braid group action coming from the mapping class group action on the punctured disc induces an action on the homology of the level $\bar{\cN}$-covering, compatible with the deck transformations, as below:
$$B^{C}_{n+{\bar{l}}+1} \curvearrowright H^{\bar{\cN}}_{n,m,{\bar{l}}} \ \left(\text{as  module over the ring } \C[x_1^{\pm 1},...,x_l^{\pm 1},y_1^{\pm 1}..., y_{\bar{l}}^{\pm 1}, d^{\pm 1}]\right).$$ 
\end{prop}
\begin{prop}[\cite{CrM} Intersection pa1iring]\label{P:3'}
We have a specialised intersection pairing assoaciated to the above homology groups:
$$\left\langle ~,~ \right\rangle:  H^{\bar{\cN}}_{n,m,{\bar{l}}} \otimes H^{\bar{\cN},\partial}_{n,m,{\bar{l}}} \rightarrow \C[x_1^{\pm 1},...,x_l^{\pm 1},y_1^{\pm 1}..., y_{\bar{l}}^{\pm 1}, d^{\pm 1}].$$
The formula for the computation of $\left\langle ~,~ \right\rangle$ is the same as the one from Proposition \ref{P:3}, specialised via $f_C$:
$$ \left\langle ~,~ \right\rangle= \  \ll ~,~ \gg|_{f_C}.$$
\end{prop}

\section{Coloured Jones polynomials for framed links}\label{S:J}
In the next two sections, we construct two topological models: the first one for coloured Jones polynomials for framed links and the second one for coloured Alexander polynomials for framed links. 

These two models are different and each of these has its own geometric characteristics. More specifically, the homological set-ups are different and the local systems that we will use are different. In particular, for the case of link invariants, the homology classes for the two quantum invariants $J_{N_1,...,N_l}$ and $\Phi^{\cN}_{\la}$ cannot be seen in the same homology group which then gets specialised in two different manners in order to recover the two link invariants. 

 We start with a topological intersection model for coloured Jones polynomials for links, whose components are coloured with different colours. Such a formula is presented in \cite{WRT}, where the geometric supports for the second homology class are given by configuration spaces on figure eights. The advantage of this new model is that we construct homology classes using configuration spaces on circles. We will do this by choosing a more subtle local system than the one from \cite{WRT}. 

Let us consider a framed oriented link $L=K_1 \cup ...\cup K_l$ with $l$ components and framings $f_1,...,f_l\in \Z$. Then, let $\beta_n \in B_n$ be a braid representative such that $L= \widehat{\beta_n}$. 

We fix $N_1,...,N_l\in \N$ a set of colours and the associated coloured multi-index:
$$\bar{N}:=(N_1,...,N_l).$$
In the next part, we will choose the local system such that our geometric support given by configuration spaces on circles leads to a well defined homology class in the homology of that covering. 

\begin{notation}
Let $M\in \N$, and denote by $V_{M}$ the $M$-dimensional representation of the quantum group $U_q(sl(2))$ for generic $q$.

In the next part we colour the components of $L$ with the representations $V_{N_1},...,V_{N_l}$. Then, we denote by $J_{\bar{N}}(L,q)$ to be the coloured Jones polynomial of this framed link (as defined in \cite{RT}).
\end{notation}
\begin{defn}(Induced colorings)\label{colourings}\\
a) (Colourings associated to the braid)  The colouring of $L$ via the colours $\bar{N}$ gives a colouring of the strands of the braid, and we define the associated colours by:
$(C_1,...,C_n)$, as in Figure \ref{colouringbraid}:
\begin{figure}[H]
\centering
\includegraphics[scale=0.45]{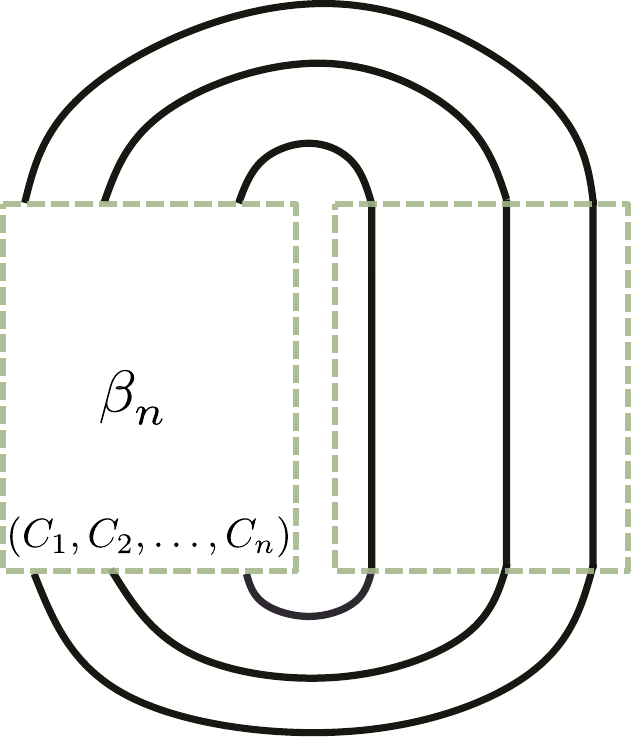}
\caption{\normalsize Colouring for the braid}
\label{colouringbraid}
\end{figure}
 
After this, we look at the link $L$ as being the closure of the braid together with $n$ straight strands. This gives a colouring of $2n$ points $C:\{1,...,2n\}\rightarrow \{1,...,l\}$. 
Let us denote by 
\begin{equation}\label{eq:N}
\ccN_i:=N_{C(i)}.
\end{equation}

We now have the following colours on the $2n$ points:
$$\bar{C}^{\bar{N}}:=(\ccN_1,...,\ccN_n,\ccN_n,...,\ccN_1).$$ 
\end{defn}
\begin{defn}(Set of states)\label{stJ} Let us define the following indexing set:
 
 $$C(\bar{N}):= \big\{ \bar{i}=(i_1,...,i_{n-1})\in \N^{n-1} \mid 0\leq i_k \leq \ccN_{k+1}-1, \  \forall k\in \{1,...,n-1\} \big\}.$$

\end{defn}
\begin{defn}[$\bar{N}$-coloured multi-level]
We consider the following indices:
\begin{equation}\label{eq:8'} 
\begin{cases}
&\cN_1=1,\\
&\cN_i=\ccN_i-1, i\in \{2,...,n\}.
\end{cases}
\end{equation}

Then, we define $\bar{N}$-coloured multi-level (associated to the coloured multi-index $\bar{N}$):

\begin{equation}
\bar{\cN}(\bar{N}):=(\cN_1,...,\cN_n) \left( =(1,\ccN_2-1,...,\ccN_n-1)\right).
\end{equation}
\end{defn}

\subsection{Homology classes}

Now, we use the induced colouring and the associated indexing set $C(\bar{N})$, in order to define the homology groups that we use for the topological model. More precisely, we fix the configuration space of $2+\sum_{i=2}^{n} (\ccN_i-1)$ points on the $(2n+2)$-punctured disc. After this, let us consider the local system $\Phi^{\bar{\cN}(\bar{N})}$ associated to the following parameters:
$$ n \rightarrow n; \ \ \ m\rightarrow 2+\sum_{i=2}^{n} (\ccN_i-1); \ \ \ {\bar{l}}\rightarrow 1; \ \ \ \bar{\cN}\rightarrow \bar{\cN}(\bar{N}).$$ 
We will use the homology groups associated to this data:
$$H^{\bar{\cN}(\bar{N})}_{n, 2+\sum_{i=2}^{n} (\ccN_i-1),1} \ \ \ \ \ \ \ \ \ \ \ \ \ \ \ \text{ and }\ \ \ \ \ \ \ H^{\bar{\cN}(\bar{N}),\partial}_{n,2+\sum_{i=2}^{n} (\ccN_i-1),1}.$$
For the next part we will erase the third component from the indices of the homology groups.

Now we have all the tools needed in order to introduce the homology classes that we use for the intersection model for coloured Jones polynomials. 

\begin{defn} (Homology classes)\\
Let $\bar{i}=(i_1,...,i_{n-1}) \in C(\bar{N})$ be a multi-index. We construct two homology classes associated to this multi-index, which are given by the geometric supports from Figure \ref{Picture0}:
 \begin{figure}[H]
\centering
$${\color{red} \FJ \in \HJ} \ \ \ \ \ \ \ \ \ \ \text{ and } \ \ \ \ \ \ \ \ \ \ \ \ \ {\color{dgreen} \LJ \in \HJd} .$$
$$\hspace{5mm}\downarrow \text{ lifts }$$
\vspace{-6mm}

\includegraphics[scale=0.4]{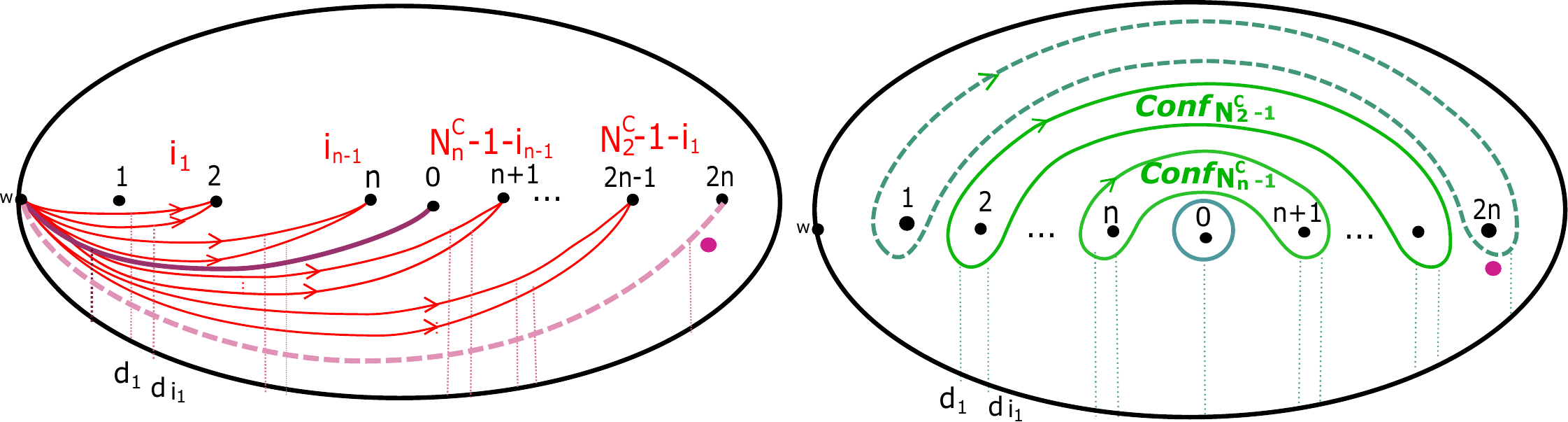}
\caption{ \normalsize Lagrangians for Coloured Jones link invariants coloured with diferent colours}
\label{Picture0}
\end{figure}
\end{defn}
\begin{prop}
The homology class $\LJ \in \HJd$ is well-defined.
\end{prop}
\begin{proof}
 The geometric support for the dual class, shown in the right hand side of Figure \ref{Picture0} is the same as the geometric support from Definition \ref{gendsupp}, associated to the multi-level $\cN(\bar{N})$, which we denoted by $$L(\cN(\bar{N})).$$ Further on, we showed in Lemma \ref{gendsuppl} that this submanifold has well-defined lifts in the covering. 
 
 In order to fix the lift, we use the procedure of lifting described in Notation \ref{paths}. More specifically, we use the path in the configuration space induced by the collection of straight vertical paths which connect the ovals to the boundary of the disc, shown in the right hand side of Figure \ref{Picture0} (see \cite{Crsym}, Section 5 for a similar detailed construction of this path in the configuration space). This is a path from the base point ${\bf d}$ to our submanifold $L(\cN(\bar{N}))$.  
 
The lift of this path through the base point $\tilde{{\bf d}}$ from the covering gives us a base point for lifting the submanifold, and we obtain a lift:
$$\tilde{L}(\cN(\bar{N}))_{\bar{i}}.$$
Following Lemma \ref{gendsuppl}, the submanifold $\tilde{L}(\cN(\bar{N}))_{\bar{i}}$ gives a well-defined homology class in the dual homology of the covering $\HJd$, which we denote:
$$\LJ \in \HJd.$$

\begin{figure}[H]
\centering
\hspace{-5mm}\includegraphics[scale=0.37]{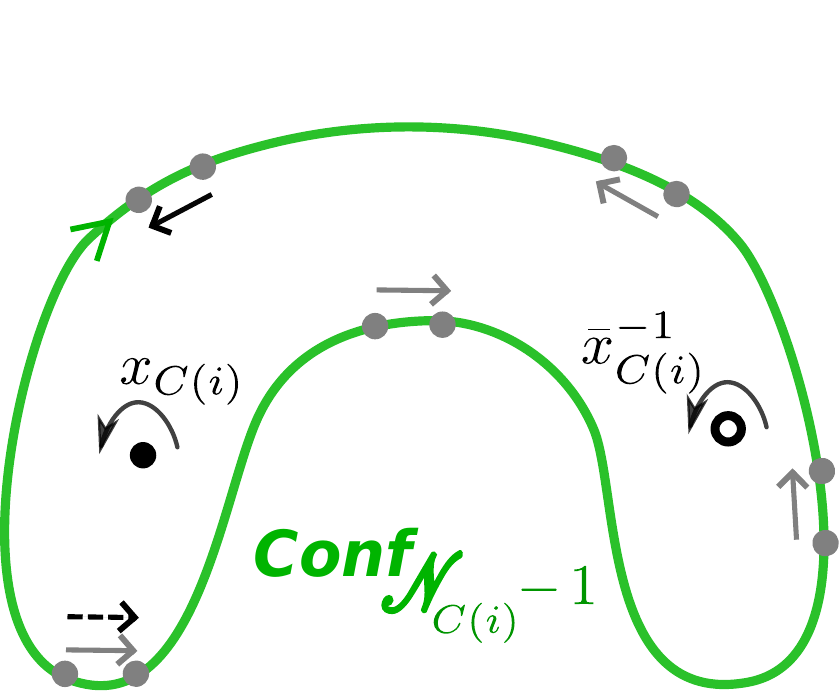}
\caption{\normalsize Monodromy of configurations on a circle}
\label{PicturemonJ}
\end{figure}
\end{proof}

In the next part, we use the specialisation of coefficients from Section \ref{S:not}, which we remind below.
\begin{defn}[Specialisations for link invariants] We use the following specialisation:
$$ \sJt: \C[u_1^{\pm 1},...,u_{l}^{\pm 1},x_1^{\pm 1},...,x_{l}^{\pm 1},y^{\pm 1},d^{\pm 1}] \rightarrow \C[q^{\pm 1}]$$
\begin{equation}\label{eq:8''''} 
\begin{cases}
&\sJt(u_j)=\left(\sJt(x_j)\right)^t=q^{N_i-1}\\
&\sJt(x_i)=q^{N_i-1}, \ i\in \{1,...,l\}\\
&\sJt(y)=[N^C_1]_{q},\\
&\sJt(d)= q^{-1}.
\end{cases}
\end{equation}
\end{defn}
\subsection{Intersection model}\label{modelJ}
In this part we present a topological model for the coloured Jones polynomial of a link coloured with the colours $N_1,...,N_l$. More precisely, we will see that it can be obtained from an intersection pairing which uses the classes $\FJ$ and $\LJ$ for all indices $\bar{i} \in C(\bar{N})$, as stated in Theorem \ref{THEOREMJ}. We remind the formula below.
\begin{thm}[Topological model for coloured Jones polynomials for coloured links via ovals]
\begin{equation}\label{THJ}
\begin{aligned}
 J_{N_1,...,N_l}(L,q)& =~  \left( \PJ \right)\Bigm| _{\sJt}.
\end{aligned}
\end{equation} 
where the intersection pairing is given by the formula from below:
\begin{equation}
\begin{cases}
& \PJ \in \Z[u_1^{\pm 1},...,u_l^{\pm 1},x_1^{\pm 1},...,x_l^{\pm 1},y^{\pm 1}, d^{\pm 1}]\\
& \PJ:=\prod_{i=1}^l u_{i}^{ \left(f_i-\sum_{j \neq {i}} lk_{i,j} \right)} \cdot \prod_{i=2}^{n} x^{-1}_{C(i)} \cdot \\
  & \ \ \ \ \ \ \ \ \ \ \ \ \ \ \ \ \ \sum_{\bar{i}\in C(\bar{N})} \left\langle(\beta_{n} \cup {\mathbb I}_{n+2} ) \ { \color{red} \FJ}, {\color{dgreen} \LJ}\right\rangle. 
   \end{cases}
\end{equation} 
\end{thm}
\begin{proof}Using the expression of the specialisation of coefficients from Section \ref{S:not}, we see that we have to prove the following formula:
\begin{equation}\label{TTHJ}
\begin{aligned}
J_{N_1,...,N_l}(L,q)& =~ q^{ \sum_{i=1}^{l}\left( f_i- \sum_{j \neq i} lk_{i,j}\right)(N_i-1)} \cdot \\
 & \cdot \left(\sum_{\bar{i}\in C(\bar{N})} \left( \prod_{i=2}^{n}x^{-1}_{C(i)} \right)\cdot  \left\langle(\beta_{n} \cup {\mathbb I}_{n+2} ) \ { \color{red} \FJ}, {\color{dgreen} \LJ}\right\rangle \right)\Bigm| _{\sJt}.
\end{aligned}
\end{equation} 

The proof of this intersection formula will be done in $4$ main steps, which are similar to the ones used in the model for coloured Jones polynomials presented in \cite{WRT}. However, the geometry of these classes is totally different than the one from \cite{WRT} which uses figure eights as geometric supports. Let us outline the main steps as below.

\

\

{\bf Step 1} 

Let us define two homology classes

$${\color{red} \FJ' \in \HJJ} \ \ \ \ \ \ \ \ \ \ \text{ and } \ \ \ \ \ \ \ \ \ \ \ \ \ {\color{dgreen} \LJ' \in \HJJd}$$

which are obtained from the geometric supports which are the same as the geometric supports of the classes $\FJ$ and $\LJ$, the only change being that we remove the 1-dimensional part  supported around the puncture of disc that is labeled by $0$ ( more specifically, we remove the purple segment and the blue circle), as in Figure \ref{FFJ1}:

\begin{figure}[H]
\centering
\vspace{-1mm}
\includegraphics[scale=0.4]{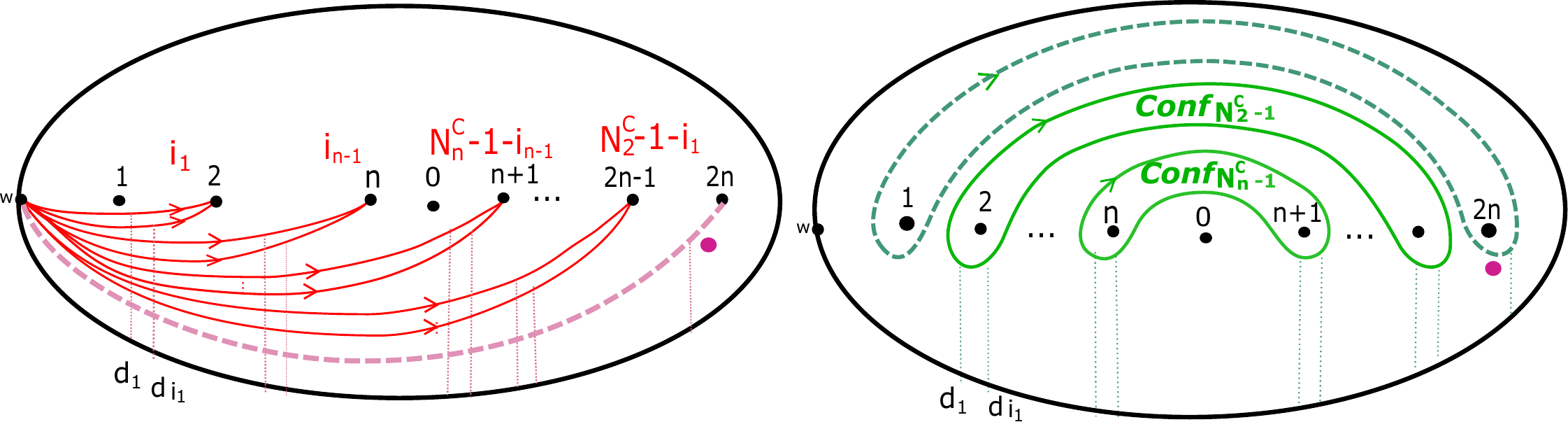}
\caption{\normalsize Removing the middle circle from the geometric support}\label{FFJ1}
\end{figure}

 Then we obtain the following relation between the intersection pairings (using $d^2$ for the contribution of the relative twisting, as in Remark \ref{rktw}):
$$ \left\langle(\beta_{n} \cup {\mathbb I}_{n+2} ) \ { \color{red} \FJ}, {\color{dgreen} \LJ}\right\rangle= d^{-2 \sum_{k=1}^{n-1}i_k}  \left\langle(\beta_{n} \cup {\mathbb I}_{n+2} ) \ { \color{red} \FJJ}, {\color{dgreen} \LJJ}\right\rangle $$
This means that we want to prove the following: 
\begin{equation}\label{d1}
\begin{aligned}
J_{N_1,...,N_l}(L,q)& =~ q^{ \sum_{i=1}^{l}\left( f_i- \sum_{j \neq i} lk_{i,j}\right)(N_i-1)} \cdot \\
 & \cdot \left(\sum_{\bar{i}\in C(\bar{N})} \prod_{i=2}^{n}x^{-1}_{C(i)} \cdot d^{-2 \sum_{k=1}^{n}i_k} \left\langle(\beta_{n} \cup {\mathbb I}_{n+2} ) \ { \color{red} \FJJ}, {\color{dgreen} \LJJ}\right\rangle \right)\Bigm| _{\sJt}.
\end{aligned}
\end{equation} 

{\bf Step 2}

 Next, we consider the homology classes given by the support of $\FJJ$ and $\LJJ$ where we remove the arc and circle that end or go around the last puncture, as in Figure \ref{Picture1}: 
\begin{figure}[H]
\centering
$${\color{red} \FJJJ \in \HJJJ} \ \ \ \ \ \ \ \ \ \ \text{ and } \ \ \ \ \ \ \ \ \ \ \ \ \ {\color{dgreen} \LJJJ \in \HJJJd} .$$
$$\hspace{5mm}\downarrow \text{ lifts }$$
\vspace{-6mm}

\includegraphics[scale=0.4]{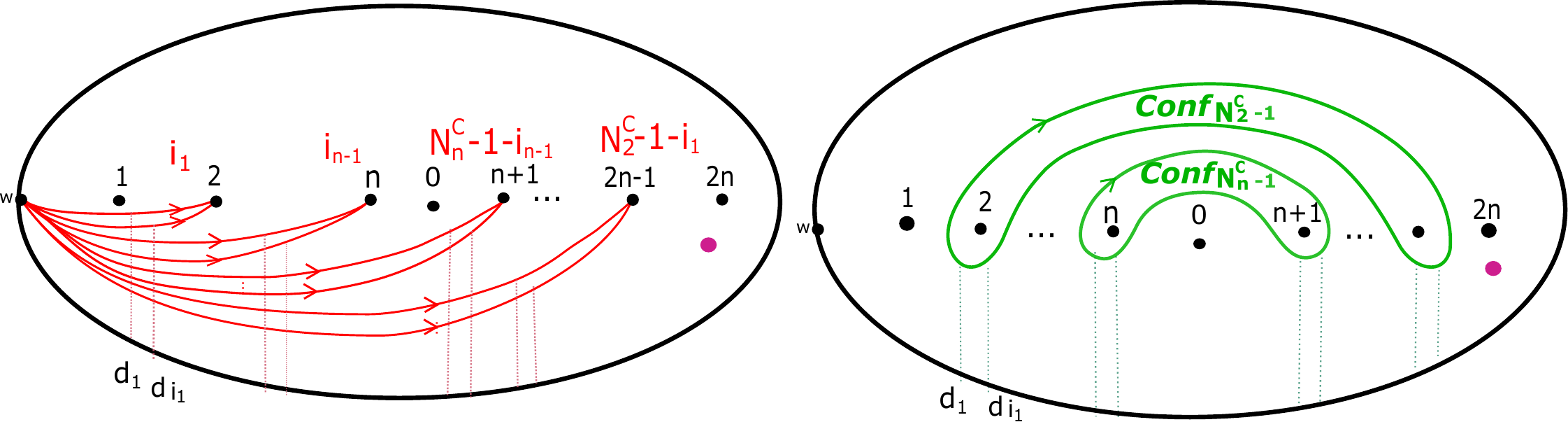}
\caption{\hspace*{-2mm} Removing the middle circle and the extremal circle from the geometric support}
\label{Picture1}
\end{figure}
Now, we compare two graded intersections: the one from Step 1, namely
$$\left\langle(\beta_{n} \cup {\mathbb I}_{n+2} ) \ { \color{red} \FJJ}, {\color{dgreen} \LJJ}\right\rangle$$ and the new pairing between the classes from Step 2, given by:

$$\left\langle(\beta_{n} \cup {\mathbb I}_{n+2} ) \ { \color{red} \FJJJ}, {\color{dgreen} \LJJJ}\right\rangle.$$
The only difference between these two intersections is that the first one comes with an extra coefficient associated to the intersection point between the extremal circle and the arc ending in the puncture labeled by $2n$, as in Figure \ref{compint}.
\begin{figure}[H]
\centering
$$\left\langle(\beta_{n} \cup {\mathbb I}_{n+2} ) \ { \color{red} \FJJ}, {\color{dgreen} \LJJ}\right\rangle  \ \ \ \ \ \ \ \ \ \ \ \ \ \ \text{ and } \ \ \ \ \ \ \ \ \ \  \ \ \ \ \ \left\langle(\beta_{n} \cup {\mathbb I}_{n+2} ) \ { \color{red} \FJJJ}, {\color{dgreen} \LJJJ}\right\rangle $$
\includegraphics[scale=0.4]{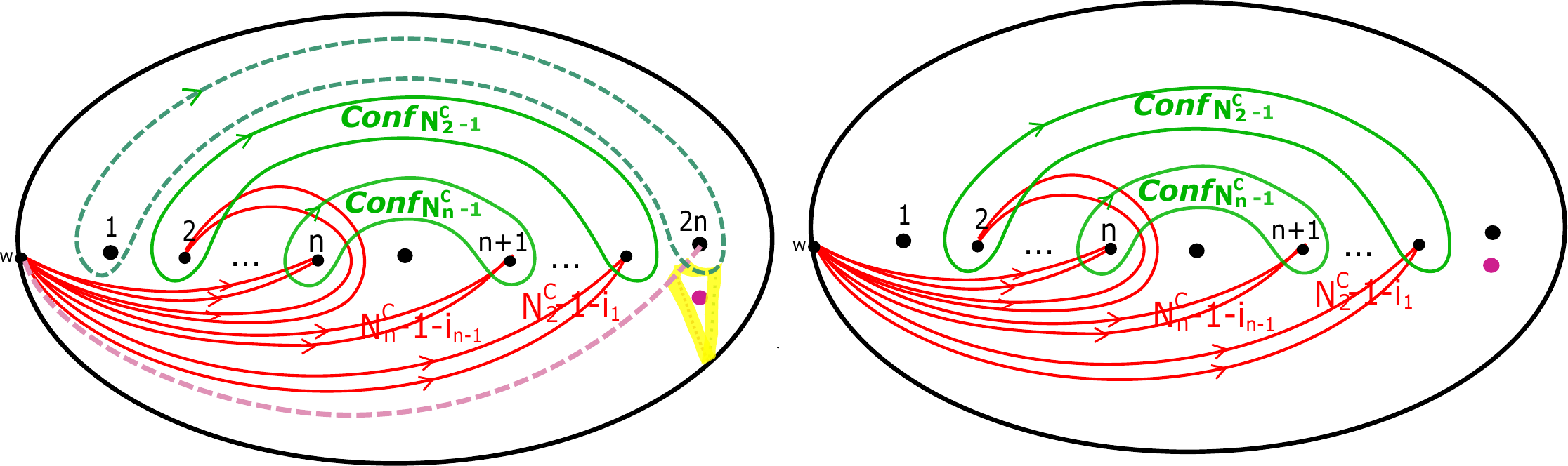}
\caption{\normalsize Removing the middle circle and the extremal circle from the geometric support}
\label{compint}
\end{figure}
This coefficient is the monodromy of the loop presented in the left hand side of Figure \ref{compint} (which goes around the $s$-puncture), which is $y$. So, we obtain the following relation between these two pairings: 
$$ \left\langle(\beta_{n} \cup {\mathbb I}_{n+2} ) \ { \color{red} \FJJ}, {\color{dgreen} \LJJ}\right\rangle= y  \left\langle(\beta_{n} \cup {\mathbb I}_{n+2} ) \ { \color{red} \FJJJ}, {\color{dgreen} \LJJJ}\right\rangle .$$
Now, we specialise these intersections using the change of coefficients $\sJt$. Using that $$\sJt(y)=[N^C_1]_q$$ we obtain the following relation:

$$ \left\langle(\beta_{n} \cup {\mathbb I}_{n+2} ) \ { \color{red} \FJJ}, {\color{dgreen} \LJJ}\right\rangle|_{\sJt}= [N^C_1]_{q} \cdot \left\langle(\beta_{n} \cup {\mathbb I}_{n+2} ) \ { \color{red} \FJJJ}, {\color{dgreen} \LJJJ}\right\rangle|_{\sJt}. $$

{\bf Step 3}

Now we use the definition of coloured Jones polynomials via the Reshetikhin-Turaev construction.  More precisely, the cups of the diagram on the algebraic side correspond exactly to the sum of the classes $\FJJJ$ over all choices of indices $\bar{i}\in C(\bar{N})$. 
Further on, the braid action on the quantum side and on the homological side correspond, using the identification due to Martel \cite{Martel}. 

{\bf Step 4}

On the last level of the braid closure, on the algebraic side, we evaluate the caps of the diagram. This means that we evaluate non-trivially just the components which are symmetric (excluding the first component, which is associated to the open strand).
On the homological side, this translates in the requirement that we evaluate non-trivially just the homology classes that are symmetric with respect to the puncture $0$, located in the middle of the disc. 

{\bf Step 5}

This condition tells us that the indices associated to the punctures $k$ and $2n+1-k$ should have the sum $\ccN_k-1$ (which is imposed by the colour of the link component associated to the puncture $k$). Geometrically we obtain this property by intersecting with the dual class $\LJJJ$. 

{\bf Step 6}

 Algebraically, we have to encode an additional piece of data, which is a coefficient corresponding to the caps of the diagram. This comes from the quantum trace from the representation side. The details of the argument for a single colour are discussed in details in \cite{Cr3} (Section 5 and Section 7). We refer to that argument and discuss below the difference that appears in this context, where we have different colours for our link.


The quantum trace encodes a coefficient that comes from the pivotal structure, which in this situation is given by the action of the element $K^{-1}$ of the quantum group (see \cite{Cr3} for the precise definition).

Let us fix a set of indices $i_1,...,i_{n-1}$ as above. The action of $K^{-1}$ on the associated monomial is the following: 
\begin{equation}\label{eq:scJ}
 q^{- \sum_{k=2}^{n}\left( (\ccN_k-1)-2i_{k-1}\right)}= \left( \prod_{k=2}^{n}q^{- (\ccN_k-1)} \right) \cdot  q^{ \ \sum_{k=1}^{n-1}2i_k}.
\end{equation}
Now we we would like to see this coefficient as coming from the specialisation of the variables that we use for the local system that gives our homology groups. For this, we remark the following:
\begin{equation}
\begin{aligned}
&\sJt(x_{C(i)})=q^{N_{C(i)}-1}=q^{N^C_i-1}, \forall i \in \{1,...,n\} \ (\text{following } \eqref{eq:N})\\
&\sJt(d^{-2})=q^{2}.
\end{aligned}
\end{equation}

This means that the coefficient from \eqref{eq:scJ} is precisely the specialisation:
\begin{equation}
\sJt \left(\left(\prod_{i=2}^{n}x^{-1}_{C(i)} \right) \cdot d^{-2 \sum_{k=1}^{n-1}i_k}\right).
\end{equation}
The last coefficient from the formula presented in \eqref{THJ} is given by the framing contribution of the components of our link $L$. This completes the intersection model for the coloured Jones polynomial $J_{N_1,...,N_l}(L,q)$ for the link coloured with colours $N_1,...,N_l$.
\end{proof}
\section{Coloured Alexander polynomials for framed links}\label{S:A}
In this part we prove a topological intersection model for the coloured Alexander polynomials for links.

This is the first topological model for these non-semisimple invariants for links whose components are coloured with different colours. For the case of links coloured with a single colour we showed a topological intersection model in \cite{Cr3} and \cite{Cr2}, where we used configuration spaces on figure eights and open circles respectively as the geometric supports for the second homology class. 

From the representation theory point of view, the change from the single colour case to the multi-colour case is totally non-trivial and requires the definition of so-called modified dimensions. In order to provide this model, we will encode geometrically this extra data, in our configuration space in the punctured disc.  

Even for the case where we colour by a single colour we obtain a new perspective, given by the fact that we have homology classes based on configuration spaces on circles rather than an open support given by the circles minus a puncture (as in \cite{Cr2}). This is achieved by our work on choosing a more subtle local system than the one from \cite{Cr2}. 

As in the previous section, we fix a framed oriented link $L=K_1 \cup ...\cup K_l$ with $l$ components, which has framings $f_1,...,f_l\in \Z$. Then we consider $\beta_n \in B_n$ to be a braid representative such that $L= \widehat{\beta_n}$.  Also, let $\cN \in \N$, $\cN \geq 2$. We  use the following notation for the set of generic colours: $\CCC:=(\C \setminus \Z) \cup \cN \Z$.

Let $\lambda_1,...,\lambda_l\in \CCC$ a set of colours for the strands of our link. The associated coloured multi-index is given by:
$$\bar{\lambda}:=(\lambda_1,...,\lambda_l).$$

\begin{notation}
For $\lambda \in \CCC$ let us denote by $U_{\lambda}$ the associated representation of the quantum group at roots of unity $U_{\xi_{\cN}}(sl(2))$  (here $\xi_{\cN}$ is our fixed primitive root of unity). 

Then, we colour the components of $L$ with the representations $U_{\lambda_1},...,U_{\lambda_l}$. Further, let us denote by $\Phi_{\cN}(L,\bar{\lambda})$ to be the coloured Alexander polynomial of the framed link $L$ (which was defined in \cite{ADO}).
\end{notation}

Now we want to define the local system such that our geometric support given by configuration spaces on circles leads to a well defined homology class in the homology of the associated covering space. 

\begin{defn}(Induced colorings)\label{colourings'}\\
a) (Colourings associated to the braid)  As in the case of the coloured Jones polynomials, the colours of the link $\bar{\lambda}$ induce a colouring of the strands of $\beta_n$. Let us denote the associated colours by:
$(C_1,...,C_n)$, as in Figure \ref{colouringbraid'}:
\begin{figure}[H]
\centering
\includegraphics[scale=0.45]{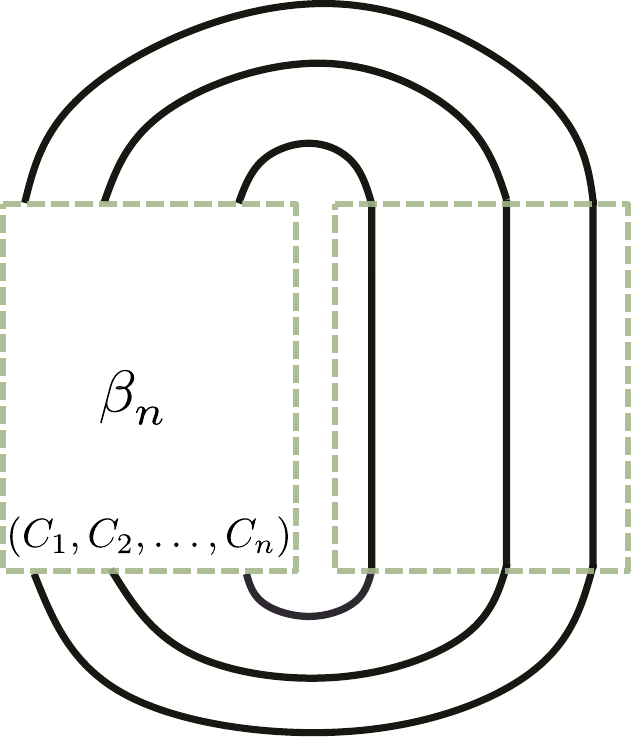}
\caption{\normalsize Colouring for the braid}
\label{colouringbraid'}
\end{figure}
 
Now, our link $L$ is the closure of $\beta_n$ together with $n$ straight strands. They induce an associated colouring of $2n$ points $C:\{1,...,2n\}\rightarrow \{1,...,l\}$. 
Let us denote by 
\begin{equation}\label{eq:N}
\ccl_i:=\lambda_{C(i)}.
\end{equation}

We consider the colouring of the $2n$ points:
$$\bar{C}^{\bar{\lambda}}:=(\ccl_1,...,\ccl_n,\ccl_n,...,\ccl_1).$$ 
\end{defn}
\begin{defn}(Set of states at level $\cN$) In this case, our indexing set will be given by the following set:
 $$C(\cN):=\{\bar{0},\dots,\overline{\cN-1}\}= \big\{ \bar{i}=(i_1,...,i_{n-1})\in \N^{n-1} \mid 0\leq i_k \leq \cN-1, \  \forall k\in \{1,...,n-1\} \big\}.$$ 
\end{defn}
We remark that in contrast to the case of coloured Jones polynomials, where the set of states depends on the choice of colours (as in Definition \ref{stJ}), here the set of states is intrinsec and it is determined just by the level and not by the particular choice of colours. 
\begin{defn}[$\cN$-coloured multi-level]
We consider the following indices:
\begin{equation}\label{eq:8'} 
\begin{cases}
&\cN_1=1,\\
&\cN_i=\cN-1, i\in \{2,...,n\}.
\end{cases}
\end{equation}

Then, we define $\cN$-coloured multi-level (which has $n$ components):
\begin{equation}
\bar{\cN}:=(1,\cN-1,...,\cN-1).
\end{equation}
\end{defn}

\subsection{Homology classes}

In the next part we define the homology classes that we work with for this second case: invariants at roots of unity. In contrast to the situation for coloured Jones polynomials, where we used the induced colouring in order to define the homology groups, here we need just the level $\cN$ in order to define these homologies. 

More specifically, let us consider the configuration space of $$2+(n-1) (\cN-1)$$ points on the $(2n+2)$-punctured disc.


 After this, we look at the local system $\Phi^{\bar{\cN}}$ associated to the following parameters:
$$ n \rightarrow n; \ \ \ m\rightarrow 2+(n-1) (\cN-1); \ \ \ {\bar{l}}\rightarrow 1; \ \ \ \bar{\cN}.$$ 
We fix the homology groups associated to these parameters, as below:
$$ \HA \ \ \ \ \ \ \ \ \ \ \ \ \ \ \ \text{ and }\ \ \ \ \ \ \ \ \ \ \ \ \ \ \HAd.$$
As in the previous section, we erase the third component associated to $\bar{l}$ from the indices of the homology groups.

Now we have all the tools needed in order to introduce the homology classes that we use for the intersection model for coloured Alexander polynomials. 

\begin{defn} (Homology classes for ADO invariants)\\
Let us fix a multi-index $\bar{i}=(i_1,...,i_{n}) \in C(\cN)$. We define the following homology classes associated to this multi-index, given by the geometric supports from Figure \ref{Picture0'}:
 \begin{figure}[H]
\centering
$${\color{red} \FA \in \HA} \ \ \ \ \ \ \ \ \ \ \text{ and } \ \ \ \ \ \ \ \ \ \ \ \ \ {\color{dgreen} \LA \in \HAd} .$$
$$\hspace{5mm}\downarrow \text{ lifts }$$
\vspace{-6mm}

\includegraphics[scale=0.4]{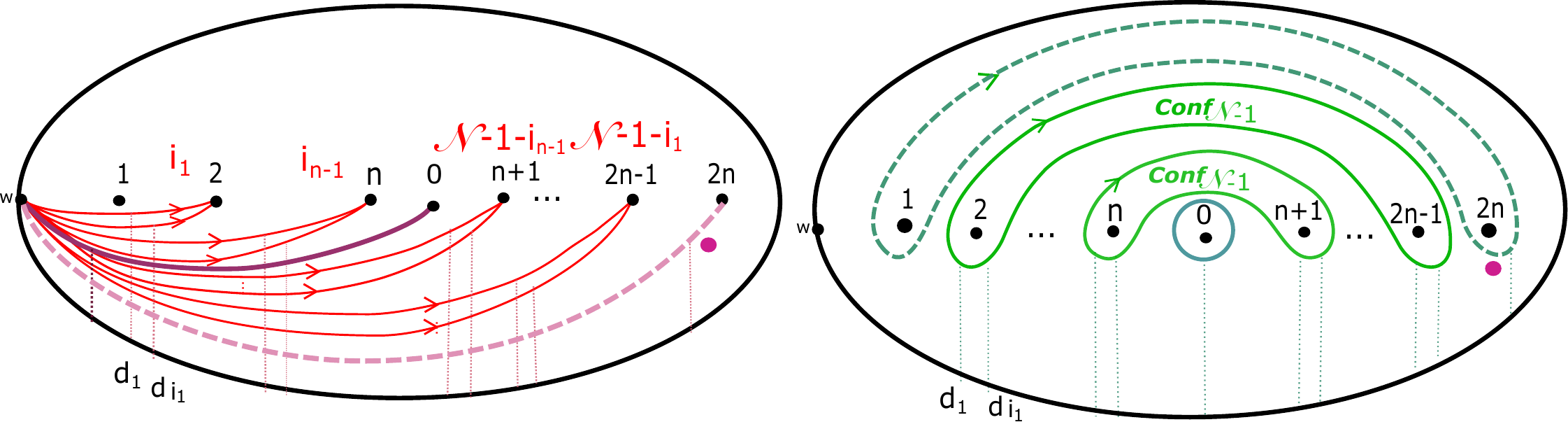}
\caption{\normalsize Lagrangians for Coloured Alexander link invariants coloured with diferent colours}
\label{Picture0'}
\end{figure}
\end{defn}
\begin{prop}
 The geometric support presented in Figure \ref{Picture0'}, induces a well-defined homology class in the covering, which we denoted by $\LA \in \HAd$.
\end{prop}
\begin{proof}
We remark that the geometric support for the dual class from Figure \ref{Picture0'} is the the geometric support from Definition \ref{gendsupp}, associated to the multi-level $\bar{\cN}.$
 
This was denoted by $L(\bar{\cN})$. Then, using Lemma \ref{gendsuppl} we know that this submanifold gives well-defined lifts in the associated covering space. 
 
As in the previous section, we use the procedure of lifting described in Notation \ref{paths}. This means that the collection of straight vertical paths which connect the ovals to the boundary of the disc gives us a path in the configuration space, from ${\bf d}$ to $L(\bar{\cN})$, shown in the right hand side of Figure \ref{Picture0'} (see \cite{Crsym}-Section 5 for the procedure of constructing such a path).  
  
The base point of the covering, $\tilde{{\bf d}}$, gives us a base point for lifting the submanifold, and we obtain a lift denoted by: $$\tilde{L}(\bar{\cN})_{\bar{i}}.$$

From Lemma \ref{gendsuppl}, this submanifold gives a well-defined homology class in the dual homology $\HAd$. We denote this class by:
$$\LA \in \HAd.$$

\begin{figure}[H]
\centering
\hspace{-5mm}\includegraphics[scale=0.37]{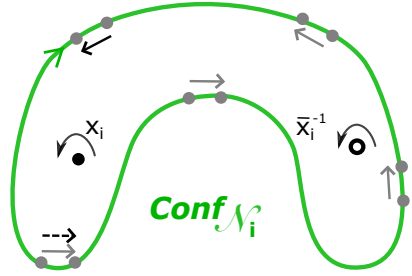}
\caption{\normalsize Monodromy of configurations on a circle for the coloured Alexander classes}
\label{PicturemonJ}
\end{figure}
This concludes the proof of the existence of the well-defined homology class.
\end{proof}
Now, let us fix the indices:
 \begin{equation}
\begin{cases}
\bar{\lambda}:=(\lambda_1,...,\lambda_l)\\
\overline{d(\lambda)}:=([\lambda_{C(1)}]_{\xi_{\cN}})\\
\end{cases}
\end{equation}

We then consider the specialisation of coefficients (which is associated to the multi-indices $\bar{\alpha}=\bar{\lambda}$ and $\eta=\overline{d(\lambda)}$), as in Definition \ref{pA2}:

$$ \sAt: \C[u_1^{\pm 1},...,u_{l}^{\pm 1},x_1^{\pm 1},...,x_{l}^{\pm 1},y^{\pm 1}, d^{\pm 1}] \rightarrow \C[\xi_{\cN}^{\pm \frac{1}{2}},\xi_{\cN}^{\pm \frac{\lambda_1}{2}},...,\xi_{\cN}^{\pm \frac{\lambda_l}{2}}]$$
\begin{equation}\label{eq:8''''} 
\begin{cases}
&\sAt(u_j)=\left(\sJt(x_j)\right)^{1-\cN}=\xi_{\cN}^{(1-\cN)\lambda_i}\\
&\sAt(x_i)=\xi_{\cN}^{\lambda_i}, \ i\in \{1,...,l\}\\
&\sAt(y)=([\lambda_{C(1)}]_{\xi_{\cN}}),\\
&\sAt(d)= \xi_{\cN}^{-1}.
\end{cases}
\end{equation}

\subsection{Intersection model} \label{modelA}

In this section we will prove the topological model for the coloured Alexander polynomials of a link coloured with the colours $\lambda_1,...,\lambda_l$. We show that this multivariable invariant can be read off from an intersection pairing via the homology classes $$\{ \FA, \LA\}  \text { for all indices } \bar{i} \in \{\bar{0},\dots,\overline{\cN-1}\}.$$ 
We recall the model presented in Theorem \ref{THEOREMA}.
\begin{thm}[Topological model for coloured Alexander polynomials for coloured links via ovals]\label{THEOREMAP}
We have the following topological model for a link coloured with colours $\bar{\lambda}$:
\begin{equation}
\begin{aligned}
\Phi^{\cN}_{\la}(L)& =~ \left(\PA \right)\Bigm| _{\sAt}.
\end{aligned}
\end{equation} 

Here, the intersection is given by:

\begin{equation}
\begin{cases}
& \PA \in \Z[u_1^{\pm 1},...,u_l^{\pm 1},x_1^{\pm 1},...,x_l^{\pm 1},y^{\pm 1}, d^{\pm 1}]\\
& \PA:=\prod_{i=1}^l u_{i}^{ \left(f_i-\sum_{j \neq {i}} lk_{i,j} \right)} \cdot \prod_{i=2}^{n} u^{-1}_{C(i)} \cdot \\
 & \ \ \ \ \ \ \ \ \ \ \ \ \ \ \ \ \ \ \ \sum_{\bar{i}\in\{\bar{0},\dots,\overline{\cN-1}\}}  \left\langle(\beta_{n} \cup {\mathbb I}_{n+2} ) \ { \color{red} \FA}, {\color{dgreen} \LA}\right\rangle. 
   \end{cases}
\end{equation}

\end{thm}

\begin{proof}
First of all, we look at the change of coefficients described in Section \ref{S:not}, and using this we want to show the following:
\begin{equation}\label{formA}
\begin{aligned}
\Phi^{\cN}_{\la}(L)& =~ {\xi_{\cN}}^{ \sum_{i=1}^{l}\left( f_i- \sum_{j \neq i} lk_{i,j}\right)\lambda_i(1-N_i)} \cdot \\
 & \cdot \left(\sum_{\bar{i}\in\{\bar{0},\dots,\overline{\cN-1}\}} \left( \prod_{i=2}^{n}u^{-1}_{C(i)} \right)\cdot  \left\langle(\beta_{n} \cup {\mathbb I}_{n+2} ) \ { \color{red} \mathscr F_{\bar{i},\cN}}, {\color{dgreen} \mathscr L_{\bar{i},\cN}}\right\rangle \right)\Bigm| _{\sAt}.
\end{aligned}
\end{equation}

We will prove this model in $6$ main parts. The first two steps are purely topological and for the rest we use the connection to representation theory. The strategy has some similarities to the one used for the model for coloured Jones polynomials from the previous section, but there are certain important topological differences, which we will emphasise in the next part.

{\bf Step 1} 

We consider the homology classes obtained from the geometric supports of the classes $\FJ$ and $\LJ$, but where we remove the 1-dimensional part supported around the puncture of the disc labeled by $0$ (this means that we remove the purple segment and the blue central circle), as in Figure \ref{FJ1}:

$${\color{red} \FA' \in \HAA} \ \ \ \ \ \ \ \ \ \ \text{ and } \ \ \ \ \ \ \ \ \ \ \ \ \ {\color{dgreen} \LA' \in \HAAd}$$
\begin{figure}[H]
\centering
\vspace{-1mm}
\includegraphics[scale=0.4]{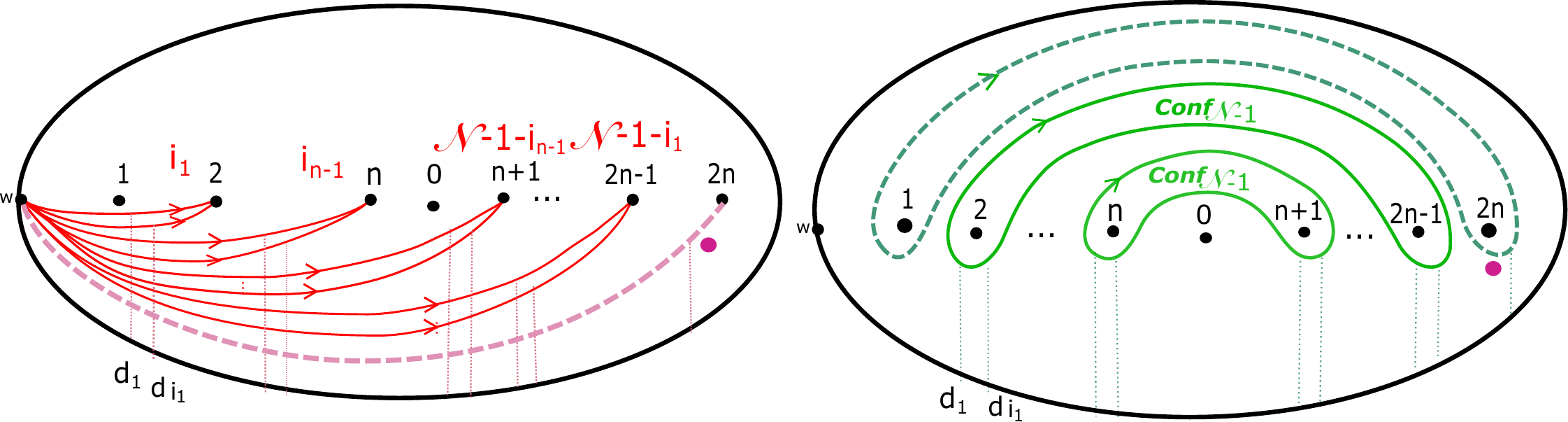}
\caption{\normalsize Removing the middle circle from the geometric support}\label{FJ1}
\end{figure}

Counting the contribution of the intersection point between the circle and the middle arc, this gives the following relation between the intersection pairings (using $d^2$ for the contribution of the relative twisting, as in Remark \ref{rktw}):
\begin{equation}\label{d2}
\left\langle(\beta_{n} \cup {\mathbb I}_{n+2} ) \ { \color{red} \FA}, {\color{dgreen} \LA}\right\rangle= d^{-2 \sum_{k=1}^{n-1}i_k}  \left\langle(\beta_{n} \cup {\mathbb I}_{n+2} ) \ { \color{red} \FAA}, {\color{dgreen} \LAA}\right\rangle 
\end{equation}
This means that we want to prove the following: 
\begin{equation}
\begin{aligned}
\Phi^{\cN}_{\la}(L)& =~ {\xi_{\cN}}^{ \sum_{i=1}^{l}\left( f_i- \sum_{j \neq i} lk_{i,j}\right)\lambda_i(1-N_i)} \cdot \\
 & \cdot \left(\sum_{\bar{i}\in\{\bar{0},\dots,\overline{\cN-1}\}} \left( \prod_{i=2}^{n}u^{-1}_{C(i)} d^{-2 \sum_{k=2}^{n}i_k}\right)\cdot  \left\langle(\beta_{n} \cup {\mathbb I}_{n+2} ) \ { \color{red} \FAA}, {\color{dgreen} \LAA}\right\rangle \right)\Bigm| _{\sAt}.
\end{aligned}
\end{equation} 

{\bf Step 2} 

Next, we consider the homology classes given by the support of $\FJJ$ and $\LJJ$ where we remove the arc and circle that end or go around the last puncture, as in Figure \ref{Picture1'}: 
\begin{figure}[H]
\centering
$${\color{red} \FAAA \in \HAAA} \ \ \ \ \ \ \ \ \ \ \text{ and } \ \ \ \ \ \ \ \ \ \ \ \ \ {\color{dgreen} \LAAA \in \HAAAd} .$$
$$\hspace{5mm}\downarrow \text{ lifts }$$
\vspace{-6mm}

\includegraphics[scale=0.4]{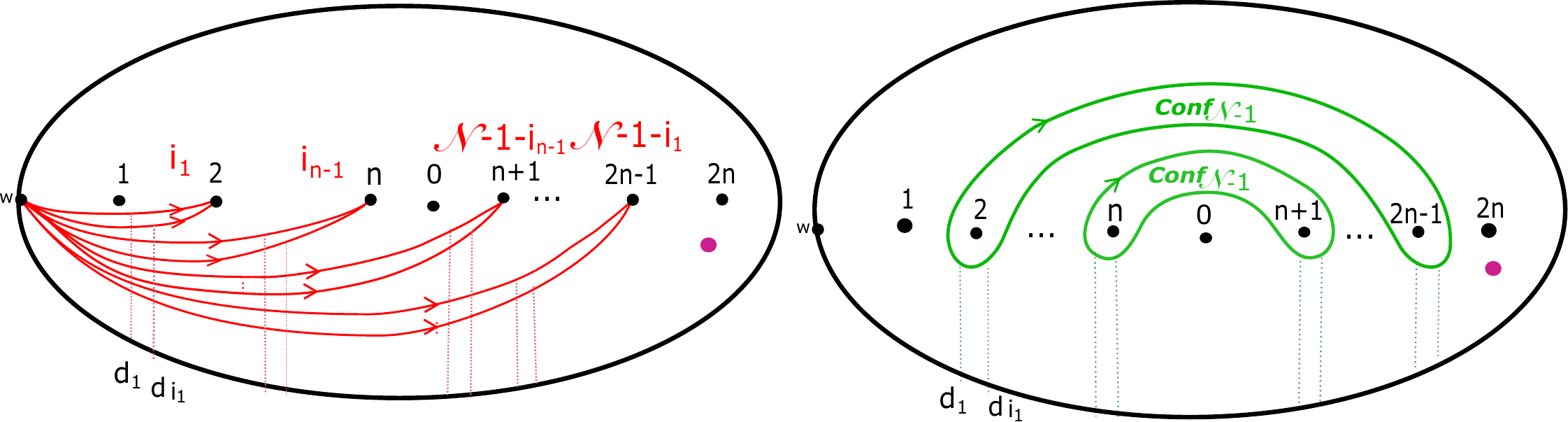}
\caption{\normalsize Removing the middle circle and the extremal circle from the geometric support}
\label{Picture1'}
\end{figure}
In this step we want to compare the graded intersection from Step I, given by
$$\left\langle(\beta_{n} \cup {\mathbb I}_{n+2} ) \ { \color{red} \FAA}, {\color{dgreen} \LAA}\right\rangle$$ and the new pairing between the classes from Figure \ref{Picture1'}, given by:

$$\left\langle(\beta_{n} \cup {\mathbb I}_{n+2} ) \ { \color{red} \FAAA}, {\color{dgreen} \LAAA}\right\rangle.$$
The unique difference between the intersections of the underlying pairs of submanifolds comes from the fact that the first pair of submanifolds has an extra intersection point between the extremal circle and the arc that ends in the puncture $2n$, as in Figure \ref{compintA}:
\begin{figure}[H]
\centering
$$\left\langle(\beta_{n} \cup {\mathbb I}_{n+2} ) \ { \color{red} \FAA}, {\color{dgreen} \LAA}\right\rangle  \ \ \ \ \ \ \ \ \ \ \ \ \ \ \text{ and } \ \ \ \ \ \ \ \ \ \  \ \ \ \ \ \left\langle(\beta_{n} \cup {\mathbb I}_{n+2} ) \ { \color{red} \FAAA}, {\color{dgreen} \LAAA}\right\rangle $$
\includegraphics[scale=0.4]{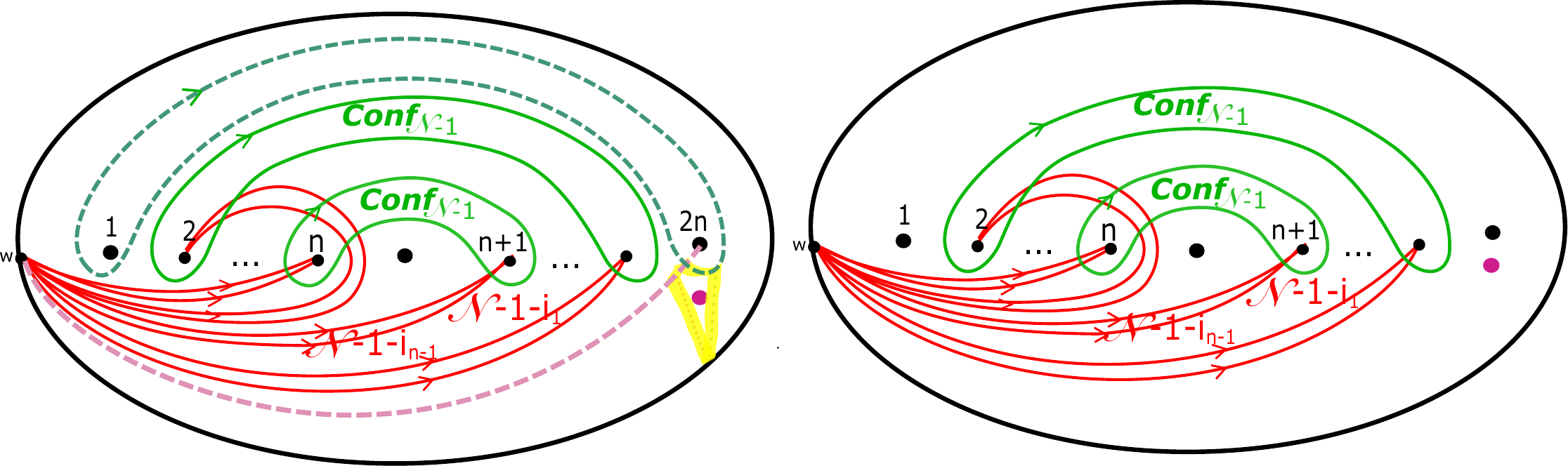}
\caption{\normalsize Removing the middle circle and the extremal circle for the coloured Alexander polynomials}
\label{compintA}
\end{figure}

That means that the associated intersection pairing has an extra coefficient coming from the contribution of this intersection point.

This is a computation which is very similar to the analog one presented in {\bf Step 2} for the case of the Coloured Jones polynomials (which we saw in the previous section).

 This coefficient comes from the monodromy of the loop around the $s$-puncture, presented in the left hand side of Figure \ref{compintA}, which gets evaluated to the variable $y$.

Then, we have the following relation between the two intersection pairings: 
$$ \left\langle(\beta_{n} \cup {\mathbb I}_{n+2} ) \ { \color{red} \FAA}, {\color{dgreen} \LAA}\right\rangle= y  \left\langle(\beta_{n} \cup {\mathbb I}_{n+2} ) \ { \color{red} \FAAA}, {\color{dgreen} \LAAA}\right\rangle $$
Then, let us specialise these intersections via the change of variables $\sAt$. We recall that: $$\sAt(y)=d(\lambda_{C(1)})=[\lambda_{C(1)}]_{\xi_{\cN}}.$$ So, we obtain the following relation:

$$ \left\langle(\beta_{n} \cup {\mathbb I}_{n+2} ) \ { \color{red} \FAA}, {\color{dgreen} \LAA}\right\rangle|_{\sAt}= [\lambda_{C(1)}]_{\xi_{\cN}} \cdot \left\langle(\beta_{n} \cup {\mathbb I}_{n+2} ) \ { \color{red} \FAAA}, {\color{dgreen} \LAAA}\right\rangle|_{\sAt} $$

{\bf Step 3} 

Now we turn our attention to the definition of the coloured Alexander invariants, coming from representation theory. We will use the version of the quantum group $U_q(sl(2))$ over two variables, which has the divided powers of one of the generators. We refer to \cite[Section 3]{Cr3} for the definition of this quantum group, its Verma modules and representation theory. 

We showed in \cite[Section 3]{Cr3} that we can specialise this representation theory at roots of unity and obtain the coloured Alexander invariant for a link whose components are coloured with the same colour.
We did this by first considering the subspace generated by the first $\cN$ weight vectors in the Verma module. Then, we noticed that when we specialise at the root of unity $\xi_{\cN}$, this subspace does not lead to a well-defined sub-representation over the quantum group $U_{\xi_{\cN}}(sl(2))$. However, the key point is that the tensor product of this specialised module is preserved by the specialised braid group action. 

This was shown in \cite[Lemma 3.1.7]{Cr3}, and it relies on the computation of the R-matrix action onto a tensor product of two vectors $v_i\otimes v_j$ from the Verma module together with the fact that when specialised at roots of unity we have a vanishing of certain quantum factorials. That ensures that if $i,j \leq \cN-1$, then the action of the specialised braiding on $v_i\otimes v_j$ will be a linear combination of $v_{i-n}\otimes v_{j+n}$ where all the indices remain bounded by $\cN$.

For our situation, namely for the coloured Alexander invariants for links, we need to work with different colours. This means that we have to consider tensor products of modules associated to different highest weights $\lambda_1$ and $\lambda_2$.

In a nutshell, the main property that changes and needs to be checked concerns the action of the braiding on such two Verma modules that have weights $\lambda_1$ and $\lambda_2$. 

 We notice that also for this situation we can generalise the argument from  \cite[Lemma 3.1.7]{Cr3}. A direct computation of the coefficients coming from the R-matrix shows that we have the property that the tensor power of finite dimensional subspaces of dimension $\cN$ associated to Verma modules of different weights are preserved by the braid group action, {\bf when specialised to the root of unity $\xi_{\cN}$}.

This means that we can obtain the coloured Alexander polynomial of the link coloured with different colours via the Reshetikhin-Turaev construction applied to this version of the quantum group. More precisely, looking on the bottom of the diagram, we have that the cups evaluated on the algebraic side correspond exactly to the sum classes $\FAAA$ over all choices of indices $\bar{i}\in \{\bar{0},...,\overline{\cN-1}\}$. 
After that, we use the property that the braid action on the quantum side and on the homological side are isomorphic, via the identification showed by Martel in \cite{Martel}. 

{\bf Step 4}

The upper part of the diagram, given by the set of caps is associated on the algebraic side to the evaluations. This means that we evaluate non-trivially just the tensor powers of vectors from the specialised Verma module which satisfy the following two requirements:
\begin{itemize}
\item[•] components are all strictly less than $\cN$ 
\item[•] the indices of the components are symmetric with respect to the colour $\cN$, meaning that they have the sum $\cN-1$ (if we remove the first component, because this is associated to the open strand).
\end{itemize}

Now, we translate the meaning of these evaluation conditions on the homological side. 

We notice that this condition gives us the requirement that we have to evaluate non-trivially just the homology classes which:
\begin{itemize}
\item[•] have all multiplicities of the multi-indices strictly less than $\cN$ 
\item[•] the number of arcs that end in symmetric punctures with respect to $0$ (which is in the middle of the punctured disc) adds up to $\cN-1$, excluding the components associated to the puncture labelled by $1$ (they are associated to the open strand).
\end{itemize}

{\bf Step 5}

Looking more closely at this requirement, we see that it enforces that we evaluate non-trivially just the classes whose indices associated to the p-punctures $k$ and $2n+1-k$ have the sum $\cN-1$. The main point is that we are able to obtain this condition in a geometric way, by intersecting with the dual class $\LAAA$.

{\bf Step 6}

 Algebraically, we have to encode an additional piece of data, which is a coefficient corresponding to the caps of the diagram. On the representation theory side, this is provided by the quantum trace. The details of the argument for the case where we have just a single colour are discussed in \cite[(Section 5, Section 8)]{Cr3}. We will proceed with an analog argument, so we present below the main differences that appears in this new situation, given by the choice of different colours for our link.

The quantum trace is a trace that is twisted by certain coefficients that come from the pivotal structure. In this case, the pivotal structure is provided by the action of the element $K^{\cN-1}$ of the quantum group (in \cite{Cr3} we present in details the precise definition of this quantum trace).

Now, we fix a set of indices $i_1,...,i_{n-1}$ which are strictly bounded by $\cN$. On the algebraic side, the action of the element $K^{\cN-1}$ on the associated monomial has the following coefficient: 
\begin{equation}\label{eq:scA}
 {\xi_{\cN}}^{(\cN-1) \sum_{k=2}^{n}\left( (\lambda_{C(k)})-2i_{k-1}\right)}= \left( \prod_{k=2}^{n}{\xi_{\cN}}^{ (\cN-1)\lambda_{C(k)}} \right) \cdot  {\xi_{\cN}}^{ \ \sum_{k=1}^{n-1}2i_k}.
\end{equation}
In the following part, we encode this coefficient using the specialisation of variables which are used for the local system. 

We recall the following property:
\begin{equation}
\begin{aligned}
&\sAt(u_{C(i)})={\xi_{\cN}}^{(1-\cN)\lambda_{C(i)}}, \forall i \in \{1,...,n\} \ (\text{following } \eqref{eq:N})\\
&\sAt(d^{-2})=\xi_{\cN}^{2}.
\end{aligned}
\end{equation}
Using this, we obtain that the coefficient from \eqref{eq:scA} is precisely the following specialisation:
\begin{equation}
\sAt \left(\left(\prod_{i=2}^{n}u^{-1}_{C(i)} \right) \cdot d^{-2 \sum_{k=1}^{n-1}i_k}\right).
\end{equation}

It remains one last coefficient in the formula presented in \eqref{formA}. This comes from the framing contribution of the components of the link $L$.

Putting all of these together, we conclude the intersection model for the coloured Alexander polynomial of the link $L$ coloured with colours $\lambda_1,...,\lambda_l$.
\end{proof}

\section{Unifying all Coloured Jones polynomials with colours bounded by $\cN$}\label{SAu}

In this part we put together the provious models for the coloured Jones polynomials for knots and we will show that if we fix $\cN$, then the intersection at level $\cN$ recovers all coloured Jones polynomials of levels less than $\cN$, as presented in Theorem \ref{THEOREMAU} which we remind below.

\

\begin{thm}[Unifying coloured Jones polynomials of bounded level]\label{THEOREMAU'}

\

Let us fix a level $\cN\in \N$. Then the state sum of Lagrangian intersection in the {\bf configuration space of $(n-1)(\cN-1)+2$} points in the punctured disc:
\begin{equation}
\PA \in \Z[u^{\pm 1},x^{\pm 1},y^{\pm 1}, d^{\pm 1}]
\end{equation} 
(given by the set of Lagrangian intersections $\{\langle (\beta_{n} \cup {\mathbb I}_{n+2} ) \ { \color{red} \FA}, {\color{dgreen} \LA} \rangle\}_{\bar{i}=(i_1,...,i_{n}) \in \{\bar{0},\dots,\overline{\cN-1}}$) 

recovers all coloured Jones polynomials of colours $\cM\leq \cN$, as below:
\begin{equation}
\begin{aligned}
J_{\cM}(K)& =~ \PA \Bigm| _{\psi_{\cM}}.
\end{aligned}
\end{equation} 
\end{thm}
\begin{proof}
Let us start by recalling the definition of the intersection form at level $\cN$:
\begin{equation}\label{f1}
\begin{aligned}
& \PA=u^{f} \cdot u^{-(n-1)} \cdot \sum_{\bar{i}\in\{\bar{0},\dots,\overline{\cN-1}\}}  \left\langle(\beta_{n} \cup {\mathbb I}_{n+2} ) \ { \color{red} \FA}, {\color{dgreen} \LA}\right\rangle. 
 \end{aligned}
\end{equation} 
Similarly, the intersection form at the inferior level, $\cM$, has the formula:
\begin{equation}\label{f2}
\begin{aligned}
& \PAM=u^{f} \cdot u^{-(n-1)} \cdot \sum_{\bar{i}\in\{\bar{0},\dots,\overline{\cM-1}\}}  \left\langle(\beta_{n} \cup {\mathbb I}_{n+2} ) \ { \color{red} \FAM}, {\color{dgreen} \LAM}\right\rangle. \end{aligned}
\end{equation} 
Following the topological model presented in Theorem \ref{THEOREMA} and Theorem \ref{THEOREMJ}, we have that $\PAM$ and $\PJ$ recover the $\cM^{th}$  coloured Alexander and coloured Jones polynomials, through the specialisation of coefficients $\sAtu$ and $\psi_{\cM}$, respectively as below: 
\begin{equation}
\begin{aligned}
\Phi^{\cM}_{\la}(K)& =~ \left(\PAM \right)\Bigm| _{\sAtu}\\
J_{\cM}(K)& =~ \left(J\PAM \right)\Bigm| _{\psi_{\cM}}.
\end{aligned}
\end{equation} 
We remark that for the case of knots $K=\hat{\beta_n}$, the two associated intersections actually coincide, and so we have:
\begin{equation}
\PAM=J\PAM, \forall \cM \in \N.
\end{equation}
Putting together these equations, we see that it is enough to prove that the state sums of Lagrangian intersections $\PA$ and $\PAM$ become equal when specialised through $\psi_{\cM}$:
\begin{equation}\label{f3J}
\begin{aligned}
\left(\PAM \right)\Bigm| _{\sAtuJ} =~ \left(\PA \right)\Bigm| _{\sAtuJ}
\end{aligned}
\end{equation} 
for all levels $\cM$ bounded by $\cN$.

Let us prove relation \eqref{f3J}, which will show that we recover all coloured Jones invariants of level bounded by $\cN$. Using the formulas from equation \eqref{f1} and \eqref{f2}, we want to prove the following relation:
\begin{equation}
\begin{aligned}
 &\sum_{\bar{i}\in\{\bar{0},\dots,\overline{\cM-1}\}}  \left\langle(\beta_{n} \cup {\mathbb I}_{n+2} ) \ { \color{red} \FAM}, {\color{dgreen} \LAM}\right\rangle| _{\psi_{\cM}}=\\
 &\sum_{\bar{i}\in\{\bar{0},\dots,\overline{\cN-1}\}}  \left\langle(\beta_{n} \cup {\mathbb I}_{n+2} ) \ { \color{red} \FA}, {\color{dgreen} \LA}\right\rangle| _{\psi_{\cN}}. 
 \end{aligned}
\end{equation} 
Looking at the second sum we notice that we have two types of terms, namely the ones associated to indices less than $\overline{\cM}$ and the ones which are in between $\overline{\cM}$ and $\overline{\cN-1}$, so we want to show the following relation:
\begin{equation}\label{topr}
\begin{aligned}
\sum_{\bar{i}\in\{\bar{0},\dots,\overline{\cM-1}\}}  \left\langle(\beta_{n} \cup {\mathbb I}_{n+2} ) \ { \color{red} \FAM}, {\color{dgreen} \LAM}\right \rangle & | _{\psi_{\cM}}=
\\=& \sum_{\bar{i}\in\{\bar{0},\dots,\overline{\cM-1}\}}  \left\langle(\beta_{n} \cup {\mathbb I}_{n+2} ) \ { \color{red} \FA}, {\color{dgreen} \LA}\right\rangle| _{\psi_{\cM}}+\\
 & +\sum_{\bar{i}\in\{\bar{\cM},\dots,\overline{\cN-1}\}}  \left\langle(\beta_{n} \cup {\mathbb I}_{n+2} ) \ { \color{red} \FA}, {\color{dgreen} \LA}\right\rangle| _{\psi_{\cM}}.  
 \end{aligned}
\end{equation} 
We will do this in two parts. First we show that the terms from the left and right hand side of the above equation that are associated to an index $\bar{i}\in\{\bar{0},\dots,\overline{\cM-1}\}$ become equal when specialised at level $\cM$. More specifically we will prove the following relation:
\begin{equation}\label{w1}
\begin{aligned}
 \left\langle(\beta_{n} \cup {\mathbb I}_{n+2} ) \ { \color{red} \FAM}, {\color{dgreen} \LAM}\right \rangle  | _{\psi_{\cM}}= \left\langle(\beta_{n} \cup {\mathbb I}_{n+2} ) \ { \color{red} \FA}, {\color{dgreen} \LA}\right\rangle & | _{\psi_{\cM}},\\
 & \forall \bar{i}\in\{\bar{0},\dots,\overline{\cM-1}\}.  
 \end{aligned}
\end{equation} 
Secondly, the terms associated to higher indices $\bar{i}\in\{\bar{\cM},\dots,\overline{\cN-1}\}$ vanish when specialised at level $\cM$. More specifically we will prove the following relation:
\begin{equation}\label{w2}
\begin{aligned}
\left\langle(\beta_{n} \cup {\mathbb I}_{n+2} ) \ { \color{red} \FA}, {\color{dgreen} \LA}\right\rangle & | _{\psi_{\cM}}=0, \ \forall \bar{i}\in\{\overline{\cM},\dots,\overline{\cN-1}\}.  
 \end{aligned}
\end{equation} 
{\bf Step I - Classes that are associated to a multi-index bounded by $\cM$}

\

In order to prove the first part, namely relation \eqref{w1}, we will actually show a slightly stronger property, namely that the classes corresponding to the indices $(\bar{i},\cM)$ and $(\bar{i},\cN)$ give the same intersection as long as the index is bounded by $\bar{\cM}$, as below.
\begin{lem}[Equality of intersections associated to indices less than $\cM$]\label{pr1}

\

The following intersection pairings, between classes associated to indices $\bar{i}$ bounded by $\bar{\cM}$ give the same result, even before specialisation:
\begin{equation}
\begin{aligned}
 \left\langle(\beta_{n} \cup {\mathbb I}_{n+2} ) \ { \color{red} \FAM}, {\color{dgreen} \LAM}\right \rangle= \left\langle(\beta_{n} \cup {\mathbb I}_{n+2} ) \ { \color{red} \FA}, {\color{dgreen} \LA}\right\rangle , \forall \bar{i}\in\{\bar{0},\dots,\overline{\cM-1}\}.  
 \end{aligned}
\end{equation}
\end{lem}

\begin{proof}
Looking at the classes $\FAM$ and $\FA$ we remark that they have the same geometric support in the left hand side of the puncture disc. The only difference is that the class $\FA$comes from a support that has more segments ending in the right hand side of the disc. 

Now, we have to act with the braid $\beta_{n} \cup {\mathbb I}_{n+1}$. This action does not change the support associated to the right hand side of the disc. 

So, we remark that the classes $\langle( \beta_{n} \cup {\mathbb I}_{n+1}\rangle) \FAM$ and $\langle( \beta_{n} \cup {\mathbb I}_{n+1}\rangle) \FA$ have the same geometric support in the left hand side of the disc, the only difference coming from the number of arcs ending in the right hand side of the disc. 

Now, we have to intersect with the dual classes, which are $\LAM$ and $\LA$. Looking at the geometric supports that give the submanifolds in the base configuration space, we remark that we have a bijective correspondence between the associated sets of intersection points, as below:

$$(\beta_{n} \cup {\mathbb I}_{n+2} )  { \color{red} \FFAM}\cap {\color{dgreen} \LLAM} \longleftrightarrow^f (\beta_{n} \cup {\mathbb I}_{n+2} )  { \color{red} \FFA}\cap {\color{dgreen} \LLA} $$
$$\bar{x} \ \ \ \longleftrightarrow \ \ \  f(\bar{x}).$$
In order to construct this bijection, let us consider an intersection point $\bar{x} \in (\beta_{n} \cup {\mathbb I}_{n+2} )  { \color{red} \FFAM}\cap {\color{dgreen} \LLAM}$. We define $f(\bar{x})$ to be the point in the configuration space coming from $\bar{x}$ where we add $(\cN-M)(n-1)$ particles. More specifically, we add to $\bar{x}$ the set of points situated in the right hand side of the disc, at the intersection between the extra arcs ending in a puncture labeled by $\{n+1,...,2n-1\}$ (from the support of $(\beta_{n} \cup {\mathbb I}_{n+2} )  { \color{red} \FFA}$) and the circles that give $\LLA$.

In order to compute the intersections we use the formula from \eqref{eq:1}. Having in mind that our submanifolds come from products of arcs, we can compute the intersections by looking at the product of signs at the local orientations in the punctured disc, associated to $\bar{x}$ and $f(\bar{x})$.
\begin{equation} 
\begin{aligned}
\ll (\beta_{n} \cup {\mathbb I}_{n+2} )  { \color{red} \FAM}, {\color{dgreen} \LAM} \gg & =  \sum_{\bar{x} \in (\beta_{n} \cup {\mathbb I}_{n+2} )  { \color{red} \FFAM}\cap {\color{dgreen} \LLAM}}  \alpha_{\bar{x}} \cdot \bar{\Phi}^{\bar{\cM}}(l_{\bar{x}}) \\
\ll (\beta_{n} \cup {\mathbb I}_{n+2} )  { \color{red} \FA}, {\color{dgreen} \LA} \gg& = \sum_{\bar{y} \in (\beta_{n} \cup {\mathbb I}_{n+2} )  { \color{red} \FFA}\cap {\color{dgreen} \LA}}  \alpha_{\bar{y}} \cdot \bar{\Phi}^{\bar{\cN}}(l_{\bar{y}})=\\ 
&= \sum_{\bar{x} \in (\beta_{n} \cup {\mathbb I}_{n+2} )  { \color{red} \FFAM}\cap {\color{dgreen} \LLAM}}  \alpha_{f(\bar{x})} \cdot \bar{\Phi}^{\bar{\cN}}(l_{f(\bar{x}})).
\end{aligned} 
\end{equation}
Now, looking at the components of the intersection points $\bar{x}$ and $f(\bar{x})$, we see that the points in the left hand side of the disc coincide, so they have the same local signs and all the components from the right hand side have positive orientation, so they do not contribute to the sign cunt. This shows that:
\begin{equation}
\alpha_{\bar{x}}=\alpha_{f(\bar{x})}.
\end{equation}
Next, let us look at the monomials that are associated to these intersection points. The loop in the configuration space which is associated to $f(\bar{x})$, $l_{f(\bar{x})}$, can be obtained from the loop associated to $\bar{x}$ union with another set of loops that pass through the extra intersection points in the right hand side of the disc (which we used in order to construct the bijection $f$). 

However, we remark that these extra loops are evaluated trivially by the local system $\bar{\phi}^{\bar{\cN}}$. This is due to the fact that they do not twist and do not go around any puncture, and so they contribute with coefficients which are all 1.

The last remark is that the two local systems $\bar{\Phi}^{\bar{\cM}}$ and $\bar{\Phi}^{\bar{\cN}}$ evaluate in the same way the loops associated to the punctures in the left hand side of the disc, and also the relative winding. So, the contribution of the components of the intersection points from the left hand side of the disc is the same, when evaluated through these two local systems. 

Putting everything together, we conclude that:
\begin{equation}
\bar{\Phi}^{\bar{\cM}}(\bar{x})=\bar{\Phi}^{\bar{\cN}}(f(\bar{x})), \ \ \ \ \forall \bar{x} \in (\beta_{n} \cup {\mathbb I}_{n+2} )  { \color{red} \FFAM}\cap {\color{dgreen} \LLAM}. 
\end{equation}

This shows that the two intersection pairings lead to the same result:
\begin{equation}
\ll (\beta_{n} \cup {\mathbb I}_{n+2} )  { \color{red} \FAM}, {\color{dgreen} \LAM} \gg=\ll (\beta_{n} \cup {\mathbb I}_{n+2} )  { \color{red} \FA}, {\color{dgreen} \LA} \gg.
\end{equation}
In order to get to the pairings that we are interested in, we still have to change the coefficients via the colouring $f_C$, following relation \eqref{eq:8}. Specialising the previous relation we get:
\begin{equation}
\begin{aligned}
&\langle (\beta_{n} \cup {\mathbb I}_{n+2} )  { \color{red} \FAM}, {\color{dgreen} \LAM} \rangle=\\
&=\ll (\beta_{n} \cup {\mathbb I}_{n+2} )  { \color{red} \FAM}, {\color{dgreen} \LAM} \gg|_{f_C}=\\
&=\ll (\beta_{n} \cup {\mathbb I}_{n+2} )  { \color{red} \FA}, {\color{dgreen} \LA} \gg|_{f_C} =\\
& =\langle (\beta_{n} \cup {\mathbb I}_{n+2} )  { \color{red} \FA}, {\color{dgreen} \LA} \rangle.
\end{aligned}
\end{equation}

\end{proof}
{\bf Step II- Classes associated to higher indices, specialised at level $\cM$}

\

Now, we will investigate what happens with the classes associated to multi-indices bigger than $\cM$, when specialised via the change of coefficients at level $\cM$. We will prove that in this situation the specialisation vanishes, which is what we wanted in relation \eqref{w2}.
\begin{lem}[Vanishing of the intersection pairing at inferior levels]\label{pr2}
\begin{equation}
\begin{aligned}
&\left\langle(\beta_{n} \cup {\mathbb I}_{n+2} ) \ { \color{red} \FA}, {\color{dgreen} \LA}\right\rangle  | _{\psi_{\cM}}=0, \forall \bar{i}\in\{\overline{\cM},\dots,\overline{\cN-1}\}.  
\end{aligned}
\end{equation}
\end{lem}
\begin{proof}
We will prove this relation in a topological way, showing that something special happens with the geometry of the above class when we specialise it to a root of unity or at natural parameters, if we are in the case of knot closures. 

First let us look at the action $(\beta_{n} \cup {\mathbb I}_{n+2} ) \ { \color{red} \FA}$. 

We will use the structure of the homology group $\HA$. Following \cite{CrM}, this homology group has a basis that is generated by the following set of elements: 
\begin{equation}
\mathscr{B}=\{ \mathscr U'_{j_0,...,j_{2n+1}} \mid j_0,...,j_{2n+1} \in \N, j_0+...+j_{2n+1}=(n-1)(\cN-1)+2\} 
\end{equation}

where $\mathscr U'_{j_0,...,j_{2n+1}}$ is the homology class given by the geometric support from Figure \ref{Statesum'}: 
\begin{figure}[H]
\centering
\includegraphics[scale=0.4]{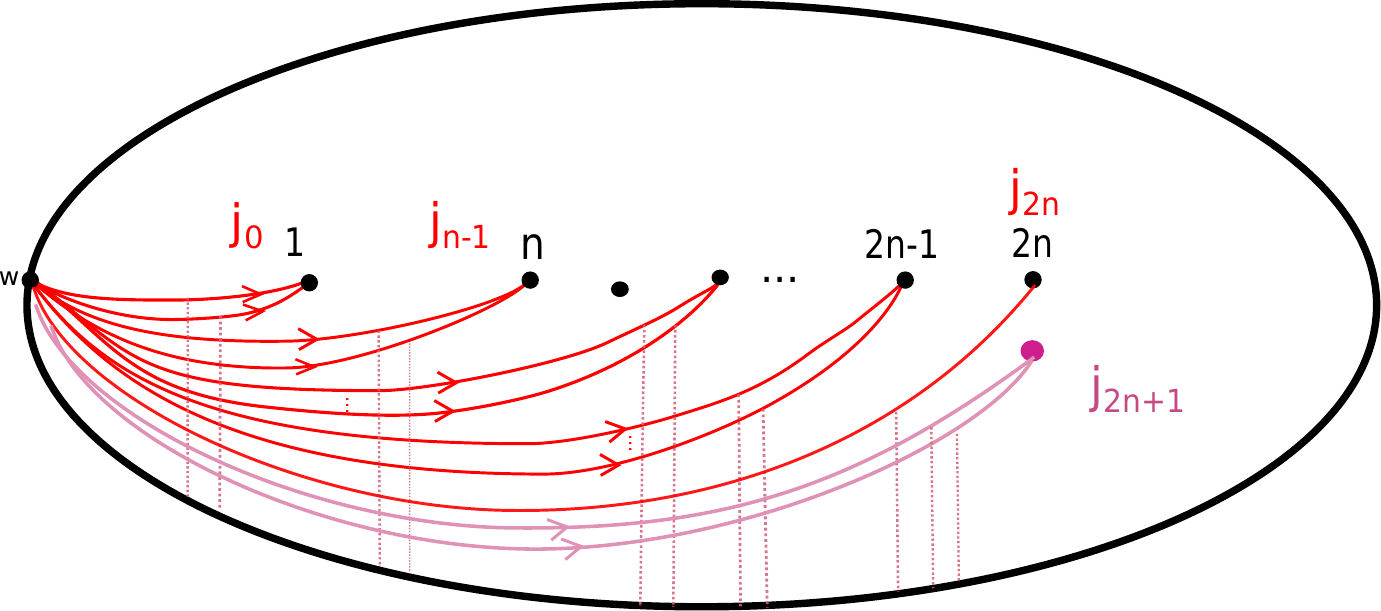}
\caption{Basis for homology}\label{Statesum'}
\vspace{2mm}
\end{figure}

This means that after we act with the braid, there exists a collection of coefficients $\alpha_{\bar{j}}\in \Z[x^{\pm1}, d^{\pm1}]$ such that:
\begin{equation}\label{coeff2}
(\beta_{n} \cup {\mathbb I}_{n} ){ \FA}=\sum_{\substack{\bar{j}=(j_0,...,j_{2n+1}) \\j_0+...+j_{2n+1}=(n-1)(\cN-1)+2}}\alpha_{\bar{j}} \cdot \mathscr U'_{\bar{j}}.
\end{equation}

We remark that the action modifies the geometric support of the class just around the punctures from the left hand side of the disc, since on the right hand side we have the trivial braid.  This means that 
\begin{equation}\label{coeff2}
(\beta_{n} \cup {\mathbb I}_{n}){ \FA}=\sum_{\substack{\tilde{j}=(j_0,...,j_{n-1})}}\alpha_{\bar{j}} \cdot \mathscr F_{\tilde{j},\tilde{i}}.
\end{equation}
where we denote $$F_{\tilde{j},\tilde{i}}:=\mathscr U'_{\tilde{j},\tilde{i}}$$ for $\tilde{i}=(1,\cN-1-j_{n-1},...,\cN-1-j_{1},1,0).$
Next, we look at the intersection with the dual class. We have the following property of the intersection form:
\begin{equation}\label{prop}
\left\langle(\beta_{n} \cup {\mathbb I}_{n+2} ) \ { \color{red} \FA}, {\color{dgreen} \LA}\right\rangle=0 \text{ if } (j_0,j_1,...,j_{n})\neq(0,i_1,...,i_{n-1}).
\end{equation}
This comes from a similar argument as the one presented in Step $4$ and Step $5$ from the proof of Theorem \ref{THEOREMAP}.
So our intersection form sees just the coefficient associated to the index $(0,i_1,...,i_{n})$: $\alpha_{\bar{i}}$ from formula \eqref{coeff2}.

For the next part, we follow a parallel phenomenon which happens on the algebraic side. The idea is that thanks to the fact that we work with knots, the induced braid action on the set of punctures from the left hand side of the disc permutes all these punctures between themselves by a cyclic permutation of maximum length:
$\sigma$ (its order is $n$). 

We look at the support of the homology class $\FA$ and we start with the first puncture whose associated index is $i_0=0$. After we act with the braid, this will contribute to the coefficients associated to the puncture $\sigma(1)$. Let us see what happens if $\sigma(1) \geq \cM$. This means that through our braid action we increase this index to jump over $\cM$ at some point. Now, we use the form of the braid action: the homological braid action corresponds to an $R$-matrix that uses powers of generators $E$ and $F^{(n)}$, via the identification between quantum and homological representations from \cite{Martel}, Lemma 6.37. Roughly speaking $E$ always decreases the index and the generators $F^{(n)}$ when acted on a component $v_i$ with index $i$ contribute to the coefficient by the formula:
\begin{equation}
F^{(n)}v_i \simeq  {n+i \brack i}_d \prod _{k=0}^{n-1} (xd^k-1) v_{i+n}.
\end{equation}
Here $\simeq$ means up to an invertible factor.

This means that if we arrive through our braid action in the puncture $\sigma(1)$ at an index such that $i_{\sigma(1)}> \cM$, since we started from $0$ at some point we acted with 
\begin{equation}
F^{(n)}v_i \simeq  {n+i \brack i}_d \prod _{k=0}^{n-1} (xd^k-1) v_{i+n}
\end{equation}
for $i<\cM$ and $i+n\geq \cM$.
\begin{rmk}(Specialisations associated to coloured Jones polynomials give vanishing coefficients)
\begin{equation}
\begin{aligned}
\psi_{\cM} \left( {n+i \brack i}_d \prod _{k=0}^{n-1} (xd^k-1)\right) =0
\end{aligned}
\end{equation}
for $i<\cM$ and $i+n\geq \cM$.
This means that if $i_{\sigma(1)}> \cM$ we have: $\alpha_{\bar{i}}=0$. Inductively, since $\sigma$ is a cycle, if we have a component $i_k>\cM$ then $\alpha_{\bar{i}}=0$, so our intersection sees just the classes with multi-indixes $\bar{i}$ that have all components less than $\cM$, when specialised through the specialisations at level $\cM$. This concludes the proof of the annulation of the specialised intersection form. 

\end{rmk}

\end{proof}

Now, we put together these properties in order to finish the proof of the globalisation property. Following Lemma \ref{pr1} and \ref{pr2}, we have:
\begin{equation}
\begin{aligned}
& \sum_{\bar{i}\in\{\bar{0},\dots,\overline{\cM-1}\}}  \left\langle(\beta_{n} \cup {\mathbb I}_{n+2} ) \ { \color{red} \FA}, {\color{dgreen} \LA}\right\rangle| _{\psi_{\cM}}+\\
 & +\sum_{\bar{i}\in\{\bar{\cM},\dots,\overline{\cN-1}\}}  \left\langle(\beta_{n} \cup {\mathbb I}_{n+2} ) \ { \color{red} \FA}, {\color{dgreen} \LA}\right\rangle| _{\psi_{\cM}}=\\
 & =\sum_{\bar{i}\in\{\bar{0},\dots,\overline{\cM-1}\}}  \left\langle(\beta_{n} \cup {\mathbb I}_{n+2} ) \ { \color{red} \FAM}, {\color{dgreen} \LAM}\right \rangle  | _{\psi_{\cM}}.
 \end{aligned}
\end{equation} 
This shows the relation that we wanted, presented in equation \eqref{topr} and so relation \eqref{f3J} holds. This concludes the proof that our intersection pairing recovers all coloured Jones polynomials for levels bounded by $\cM$, as stated in Theorem \ref{THEOREMAU'}. 

\end{proof}

\subsection{Graded intersection in the ring with integer coefficients}
 \begin{lem}
For any braid (in general, not just for braids that give knot closures), the graded intersections $\PA$ and $\PJ$ take values in the Laurent polynomial ring with integer coefficients:
$$\PA, \PJ \in \Z[u_1^{\pm 1},...,u_l^{\pm 1},x_1^{\pm 1},...,x_l^{\pm 1},y^{\pm 1}, d^{\pm 1}].$$
\begin{proof}
We remind the formula for this intersection form:
\begin{equation}
\begin{aligned}
\PA&:=\prod_{i=1}^l u_{i}^{ \left(f_i-\sum_{j \neq {i}} lk_{i,j} \right)} \cdot \prod_{i=2}^{n} u^{-1}_{C(i)} \cdot \\
 & \ \ \ \ \ \ \ \ \ \ \ \ \ \ \ \ \ \ \ \sum_{\bar{i}\in\{\bar{0},\dots,\overline{\cN-1}\}}  \left\langle(\beta_{n} \cup {\mathbb I}_{n+2} ) \ { \color{red} \FA}, {\color{dgreen} \LA}\right\rangle. 
 \end{aligned}
\end{equation} 
Since the homology classes $\FA$ and $\LA$ belong to the homology groups $\HA$ and $\HAd$ that are modules over  $\C[x_1^{\pm 1},...,x_l^{\pm 1},y^{\pm 1}, d^{\pm 1}]$, we have that a priori 
$$\PA \in \C[u_1^{\pm 1},...,u_l^{\pm 1},x_1^{\pm 1},...,x_l^{\pm 1},y^{\pm 1}, d^{\pm 1}] .$$
We will show that for our particular homology classes, we have the property that actually their intersection pairing has all coefficients that are integers:
\begin{equation}\label{integ}
\left\langle(\beta_{n} \cup {\mathbb I}_{n+2} ) \ { \color{red} \FA}, {\color{dgreen} \LA}\right\rangle \in \Z[x_1^{\pm 1},...,x_l^{\pm 1},y^{\pm 1}, d^{\pm 1}].
\end{equation}
As we have seen in Proposition \ref{P:3'}, we have 
$$ \left\langle ~,~ \right\rangle= \  \ll ~,~ \gg|_{f_C}$$
and the formula for $\left\langle ~,~ \right\rangle$ is the same as the one from Proposition \ref{P:3}, specialised via $f_C$.
 
 We will use Remark \ref{orientd}, and then specialise the coefficients through $f_C$.
For our homology classes, their intersection will be paramatrised by the intersection points between their associated submanifolds in the base configuration space, graded by $\bar{\psi}^{\cN}$ . 
So it is enough to show that $\tilde{\psi}^{\cN}$ evaluated on the loops that are associated to our intersection points takes integer coefficients.
Following relation \eqref{sor} we remind that 
$$\bar{\psi}^{\cN}=\eta \circ \tilde{\psi}^{\cN}.$$
The monodromies of the local system $\tilde{\psi}^{\cN}$ are:
\begin{itemize}
\item $x_i^2$ around the $p$-punctures from the left hand side of the disc, ($i \in \{1,...,n\}$)
\item $x_i^2 \cdot (d')^{2(\cN_i-1)}$ around the $p$-punctures punctures from the right hand side of the disc ($i \in \{n+1,...,2n\}$)
\item $y_1,...,y_l$ around the $s$-punctures
\item $(d')^2$ for the relative twisting.
\end{itemize} 
So $\tilde{\psi}^{\cN}$ evaluates classes of loops with even powers of $d'$. This means that when composing with the change of coefficients $\eta$, given by:
\begin{equation} 
\begin{aligned}
&\eta: \Z[x_1^{\pm 1},...,x_n^{\pm 1},y^{\pm1}, d'^{\pm 1}]\rightarrow \C[x_1^{\pm 1},...,x_n^{\pm 1},y^{\pm 1}, d^{\pm 1}]\\
&\begin{cases}
\eta(x_i)=x_i\\
\eta(y)=y\\
\eta(d')=id.
\end{cases}
\end{aligned}
\end{equation} as in \eqref{groupring'}, we obtain integer coefficients. 

This shows that \eqref{integ} holds and so our intersection form has integer coefficients: 
$$\PA \in \Z[u_1^{\pm 1},...,u_l^{\pm 1},x_1^{\pm 1},...,x_l^{\pm 1},y^{\pm1}, d^{\pm 1}].$$

The same argument applies also for the intersection form:
$$\PJ \in \Z[u_1^{\pm 1},...,u_l^{\pm 1},x_1^{\pm 1},...,x_l^{\pm 1},y^{\pm1}, d^{\pm 1}].$$
\end{proof}

 \end{lem}

\section{$\cN^{th}$ Unified Jones invariant} \label{S:UN}In this section our aim is to construct a knot invariant $\PAAJ$ starting from the graded intersection in the confogiration space $\Conf_{2+(n-1)(\cN-1)}\left(\mathbb D_{2n+2}\right)$: $$\PA.$$ 
This invariant is defined in a fixed configuration space  will unify all coloured Jones polynomials up to level $\cN$. 

The first part for this construction concerns the definition the appropiate ring of coefficients where these invariants will be defined. Secondly, we will show that in such a universal ring, we have indeed a well defined knot invariant. 

\subsection{Set-up and notations}

\begin{defn}[Rings for the universal invariant] Let us define the following rings:
 \begin{equation}
 \begin{cases}
 \LL:=\Z[x^{\pm 1},d^{\pm 1}]\\
\LNJ=\Z[d^{\pm 1}] .
 \end{cases}
 \end{equation}
\end{defn}
From now on, since we work with knots, we will set $y=1$ and $u=x$, as explained in Convention \ref{convention}. Then our intersection $\PA$ recovers the normalised versions of the $\cN^{th}$ coloured Jones and Alexander polynomials. We define $J_{\cN}$ to be the ${\cN}^{th}$ normalised coloured Jones polynomial and we will look at it in the variable $d$ (which corresponds to $q^{-1}$ from the previous sections).

So we have our intersection $$\PA \in \LL=\Z[x^{\pm 1},d^{\pm 1}].$$
\begin{defn}[Specialisation for the universal invariant] 

\

We change the ring of coefficients as above, and let us define the associated globalised specialisation of coefficients as below:
$$ \snJ: \LL=\Z[x^{\pm 1},d^{\pm 1}] \rightarrow \LNJ$$
\begin{equation}
\snJ(x)= d^{1-\cN}.
\end{equation}
\end{defn}
\subsection{Product up to a finite level} 
\begin{defn}[Product up to level $\cN$] For a fixed level $\cN$, let us define the product of the rings for smaller levels, as:
 \begin{equation}
\begin{aligned}
 \LNtJ:=\prod_{\cM \leq \cN}\LMJ.
 \end{aligned}
 \end{equation}
 and $\PMNJ:\LNt \rightarrow \LMJ$ the projection onto the $\cM^{th}$ component.
 
 Also, let us denote by $\sntJ: \LL \rightarrow \LNtJ$ the product of specialisations:
\begin{equation}
 \sntJ:=\prod_{\cM \leq \cN}\smJ.
 \end{equation} 
\end{defn}
\begin{defn}[Coloured invariants up to a fixed level]
We denote the product of the invariants up to level $\cN$ as below:
 \begin{equation}
 \begin{aligned}
& \ANtJ:=\prod_{\cM \leq \cN}J_{\cM}(K) \in \LNtJ. 
\end{aligned}
 \end{equation}

\end{defn}
We will construct our universal ring using this sequence of morphisms $\{\sntJ \mid \cN \in \N\}$. More specifically, we will use the sequence of kernels associated to these maps, as follows. 
\begin{defn}[Kernels and quotient rings] Let us denote by:
\begin{equation}\label{ker}
\begin{aligned}
&\tilde{I}^{J}_{\cN}:=\text{Ker}\left(\sntJ \right) \subseteq \LL.
\end{aligned}
\end{equation}
We also denote the quotient rings associated to these ideals by:
\begin{equation}
\LLNJ:=\LL / \tilde{I}^{J}_{\cN}.
\end{equation}
\end{defn}
We remind that we have the product of the rings for smaller levels:
 \begin{equation}
  \LNtJ:=\prod_{\cM \leq \cN}\LMJ.
 \end{equation}
 \begin{defn}[Induced maps on quotient rings]
This means that we have the following diagram with quotient maps, which we denote $\snbbJ$ and $g^J_{\cN}$, as below:
\begin{figure}[H] \label{L:special}
\begin{center}
\begin{tikzpicture}
[x=1mm,y=1mm]

\node (b1)               at (-20,0)    {$\LL$};
\node (a1)               at (-32,0)    {$\PA \in$};


\node (b2) [color=dgreen] at (-80,0)   {$\LNtJ$};
\node (b3) [color=red] at (-40,-20)   {$\LLNJ$};
\node (b4) [color=red] at (-85,-20)   {$\LLhJ$};
\node (c1) [color=dgreen] at (-75,20)   {$\LNJ$};
\node (a2')               at (-53,-20)    {${\color{red} \PAAJ} \in$};

\draw[->,color=black]   (b1)      to node[left,xshift=-3mm,yshift=0mm,font=\large]{            $g^J_{\cN}$} (b3);
\draw[->,color=black]   (a1)      to  node[right,xshift=-2mm,yshift=3mm,font=\large]{$\sntJ$}                        (b2);
\draw[->,color=black]   (b3)      to node[left,yshift=-2mm,xshift=1mm,font=\large]{$\snbbJ$}                        (b2);
\draw[->,color=black, dashed]   (b4)      to  node[right,xshift=-2mm,yshift=-3mm,font=\large]{}                        (a2');
\draw[->,color=black,dashed]   (b4)      to  node[right,xshift=0mm,yshift=0mm,font=\large]{$\snbJ$}                        (b2);
\draw[->,color=black,dashed]   (b2)      to  node[right,xshift=0mm,yshift=0mm,font=\large]{$\PNNJ$}                        (c1);
\draw[->,color=black,dashed]   (b4)   [out=150,in=130]   to  node[xshift=-7mm,yshift=0mm,font=\large] {Def. $\ref{pl}$ \ \ \ \ \  $\usnJ$}                        (c1);
\end{tikzpicture}
\end{center}
\vspace{-3mm}
\caption{\normalsize Quotients of the intersection form}\label{diagqu}
\end{figure}

We also define the projection map:

\begin{equation}\label{quotmn}
\begin{aligned}
&\snbJt: \LLNJ \rightarrow \LMJ\\
&\smnbJt:= \PMNJ \circ \snbbJ.
\end{aligned}
\end{equation}
 \end{defn}
\begin{rmk}

By the definition of the product rings $\LNtJ$, we see that we have a sequence of nested ideals:
\begin{equation}
\begin{aligned}
&\cdots \supseteq \tilde{I}^{J}_{\cN} \supseteq \tilde{I}^{J}_{\cN+1} \supseteq \cdots 
\end{aligned}
\end{equation}
\end{rmk}
\begin{defn}[Sequence of quotient rings]
In this manner, we obtain a sequence of associated quotient rings, with maps between them:
\begin{equation}
\begin{aligned}
 l^{J}_{\cN} \hspace{10mm} l^{J}_{\cN+1}  \\
\cdots \LLNJ \leftarrow \LLNNJ \leftarrow \cdots 
\end{aligned}
\end{equation}
\end{defn}
\begin{defn}[Limit ring]\label{eq10:limit} We define the projective limit of this sequence of rings and denote it as follows:
\begin{equation}
\LLhJ:= \underset{\longleftarrow}{\mathrm{lim}} \ \LLNJ.
\end{equation}
\end{defn}
\begin{defn}\label{pl}
Then, we have induced maps:
 \begin{equation}
 \snbJ: \LLhJ \rightarrow \LNtJ.
  \end{equation}
Further, composing with the projection onto the $\cN^{th}$ component we obtain the following map:
\begin{equation}
\usnJ: \LLhJ \rightarrow \LNJ,
  \end{equation}
defined as $$\usnJ:=\PNNJ \circ \snbJ .$$
\end{defn}
\begin{defn}[Projection maps for the coloured Jones coefficient rings] We have projection maps which go from the coefficient rings at consecutive levels, as follows:
 \begin{equation}\label{eq10:2}
 \PPNJ: \LNNtJ \rightarrow \LNtJ.
  \end{equation}
 \end{defn}
Now we are going to define step by step the universal invariant, coming from the geometric data given by our intersection pairing.
\subsection{$\cN^{th}$ Unified Jones invariant}
 First, we define a set of knot invariants that will be used to build the globalised coloured Jones invariant. We recall that we have the graded intersection:
 $$\PA \in \LL$$
 and it recovers the coloured Jones polynomials, as below:
 \begin{equation}\label{rem1}
 \snJ\left( \PA\right)=\ANJ.
 \end{equation}
Taking the product component-wise, we conclude that:
\begin{equation}\label{eq:10.1}
 \sntJ\left( \PA\right)=\ANtJ.
 \end{equation}
\begin{defn}[Level $\cN$ Jones-quotient] \label{NJO}
Let us consider the intersection form obtained from $\PA$ by taking the quotient through the projection $g^J_{\cN}$ (as in Figure \ref{diagqu}):
\begin{equation}
\begin{aligned}
&\PAAJ \in \LLNJ\\
& \PAA=g_{\cN}(\PA).
\end{aligned}
\end{equation}
\end{defn}  
Now we will prove Theorem \ref{INV} which we remind below.
\begin{thm}[$\cN^{th}$ unified Jones invariant]\label{levN}
The intersection $\PAAJ$ is a well-defined knot invariant unifying all coloured Jones polynomials up to level $\cN$:
$$  \PAAJ|_{\smnbJt}=\AMJ, \forall \cM \leq \cN.$$
\end{thm}
\begin{proof}
On one hand we have we have:
$$ \sntJ\left( \PA\right)=\ANtJ$$
for $L=\hat{\beta_n}$. On the other hand, following Definition \ref{NJO}, the projection $g^J_{\cN}$ is defined from the map $\sntJ$ and we have:
\begin{equation}
 \PAA=g_{\cN}(\PA).
 \end{equation}
Putting these properties together, it follows that 
$$ \snbb\left( \PAA\right)=\ANtJ.$$
On the other hand, from Definition \ref{ker} of the quotient rings, we know that $\snbb$ is injective. Also, $\ANtJ$ is given by the product of all coloured Jones invariants up to level $\cN$, so it is a knot invariant. 

So, we conclude that $\PAA$ is mapped to the knot invariant $\ANtJ$, through the injective function $\snbb$. This means that $\PAA$ is a well-defined knot invariant. 

This knot invariant recovers all coloured Jones polynomials up to level $\cN$, using the definition of the quotient morphisms from \eqref{quotmn}. This concludes the proof of the structure of this knot invariant. 

\end{proof}

\section{Universal Coloured Jones Invariant} \label{S:JU}
In this part our aim is to unify and show that we can define an invariant out of the sequence: $$\{\PAA\mid \cN \in \N\}.$$
 \subsection{Definition of the universal invariant} \label{univ}

 First we recall the Definition \ref{eq10:limit} which contains the formula for the appropiate ring of coefficients where this universal invariant will be defined
$$\LLhJ:= \underset{\longleftarrow}{\mathrm{lim}} \ \LLNJ.$$

Next we show that in such a universal ring, we have indeed a well defined knot invariant which is a universal coloured Jones invariant, recovering all $J_{\cN}(K)$ for all colours $\cN \in \N$.
This will be constructed from the sequence of intersections up to level $\cN$.
 \begin{defn}[Unification of all coloured Jones knot invariants] Let us consider the projective limit of the graded intersection pairings $\PAAJ$ over all levels, and denote it as follows:
\begin{equation}
\IJJ:= \underset{\longleftarrow}{\mathrm{lim}} \ \PAAJ \in \LLhJ.
\end{equation}
\end{defn}

  \begin{defn}[Limit ring of coefficients and limit invariant]
 Now, let us consider the product of the rings where the coloured Jones invariants belong to, where we put no condition about the level:
  \begin{equation}
 \LaJ:=\prod_{\cM \in \N}\LMJ
 \end{equation}

 \begin{equation}
 \AaJ:=\prod_{\cM \leq \cN}\AMJ \in \LaJ. 
 \end{equation}
Using the definition from above, let us denote the projection map:
 \begin{equation}
\PhM: \LaJ \rightarrow \LMJ.
  \end{equation}
We remark that this means that projecting onto the $\cM^{th}$ component we obtain the $\cM^{th}$ coloured Jones invariant:
\begin{equation}\label{eq10:5}
\PhM(\AaJ)=\AMJ.
\end{equation}

\end{defn}
Putting together all these tools now we are ready to show Theorem \ref{THEOREMUnivJ}, which we recall as follows.
\begin{thm}[Universal coloured Jones Invariant]\label{proofADO} The limit of the invariants via the graded intersections $\IJJ$ is a knot invariant recovering all coloured Jones polynomials:
\begin{equation} 
J_{\cN}(K)=  \IJJ|_{\usnJ}.
\end{equation}
\end{thm}

\begin{proof}

\begin{figure}[H]
\begin{equation}
\centering
\begin{split}
\begin{tikzpicture}
[x=1mm,y=1.6mm]
\node (ll) at (-30,15) {$\LL$};
\node (tl) at (0,62) {$\LLhJ$};
\node (tr) at (60,62) {$\LaJ$};
\node (ml) at (0,30) {$\LLNNJ$};
\node (mr) at (60,30) {$\LNNtJ$};
\node (bl) at (0,0) {$\LLNJ$};
\node (br) at (60,0) {$\LNtJ$};
\node (rr) at (90,15) {$\Lambda_{\cM}$};
\node (lle) at (ll.south) [anchor=north,color=red] {\rotatebox{90}{$\subseteq$}};
\node (llee) at (lle.south) [anchor=north,color=red] {$\{\PA\}$};
\node (tle) at (tl.south) [anchor=north,color=red] {$\IJJ$};
\node (tre) at (tr.south) [anchor=north,color=red] {$\AaJ$};
\node (mle) at (ml.south east) [anchor=north west,color=red] {$\PAANJ$};
\node (ble) at (bl.north east) [anchor=south west,color=red] {$\PAAJ$};
\node (mre) at (mr.south west) [anchor=north east,color=blue] {$\ANNtJ$};
\node (bre) at (br.north west) [anchor=south east,color=blue] {$\ANtJ$};
\node (rre) at (rr.north) [anchor=south,color=blue] {$\AMJ$};
\node (trd) at (62,60.5) [anchor=west] {$\coloneqq \displaystyle\prod_{\cM \in \N} \LMJ$};
\draw[->] (tl) to node[above,font=\small,color=gray]{$\suJ$} (tr);
\draw[->,densely dashed] (ll) to node[above,font=\small]{$g^J_{\cN +1}$} (ml);
\draw[->,densely dashed] (ll) to node[below,font=\small]{$g^J_{\cN}$} (bl);
\draw[->] (mr) to node[below,font=\small]{$\PMNNJ$} (rr);
\draw[->](br) to node[below,font=\small]{$\PMNJ$} (rr);
\inclusion{above}{$\snnbbJ$}{(ml)}{(mr)}
\inclusion{below}{$\snbbJ$}{(bl)}{(br)}
\draw[->] (ml) to node[left,font=\small]{$\ell_{\cN}$} (bl);
\draw[->] (mr) to node[right,font=\small]{$\PPNJ$} (br);
\draw[->] (ll) to[out=-71,in=210] node[below,font=\small]{$\sntJ$} (br);
\draw[->] (ll) to[out=70,in=150] node[above,font=\small]{$\snntJ$} (mr);
\draw[|->,color=red] (13,22) to node[right,font=\small,circle,draw,inner sep=1pt]{3} (13,8);
\draw[|->,color=blue] (47,22) to node[left,font=\small,circle,draw,inner sep=1pt]{1} (47,8);
\draw[|->,color=gray] (mle) to node[above,font=\small,circle,draw,inner sep=1pt]{2} (mre);
\draw[|->,color=gray] (ble) to node[below,font=\small,circle,draw,inner sep=1pt]{2} (bre);
\node [color=dgreen] at (28,15) {$\equiv$};
\node [color=dgreen,font=\small,circle,draw,inner sep=1pt] at (32,15) {4};
\draw[|->,densely dashed,color=red] (tle) to (tre);
\draw[|->,color=red] (tle.south east) to node[above,font=\small]{$\usmJ$} (rre);
\draw[|->,color=gray] (tre) to node[right,font=\small]{$\hat{p}^J_{\cM}$} (rre);
\draw[->,dotted,color=gray] (tle) to (ml);
\draw[->,dotted,color=gray] (tre) to (mr);
\draw[->,color=blue] (tle) to node[left,font=\small]{$\snbbbJ$} (mr);
\node[draw,rectangle,inner sep=3pt] at (-25,62) {Limit};
\node[draw,rectangle,inner sep=3pt,color=red] at (-29,57) {Universal Jones invariant};
\end{tikzpicture}
\end{split}
\end{equation}

\caption{\large The universal ring and universal invariant as projective limits}
\end{figure}

In order to have a well-defined limit, we should prove that:
\begin{equation}
l^J_{\cN}\left(\PAANJ \right)=\PAAJ.
\end{equation}

\begin{rmk}\label{rk1} Following the definition of the projection maps for the coefficient rings from relation \eqref{eq10:2}, we have the property:
\begin{equation}
\PPNJ \left(\ANNtJ \right)=\ANtJ.
\end{equation}
\end{rmk}
\begin{lem}\label{lemma2}
The knot invariant at level $\cN$ recovers all coloured Jones invariants with levels smaller than $\cN$, which means that we have the following relation:
\begin{equation}
\snbbJ \left(\PAAJ \right)=\ANtJ.
\end{equation}
 \end{lem}
\begin{proof}
This property will follow from Theorem \ref{THEOREMAU} together with the definition of the quotient maps, as we will see below. We notice that it is enough to prove that this relation holds when composed with the set of projections $\PMNJ$ for all $\cM \leq \cN$.
So, we want to show:
\begin{equation}
\PMNJ \circ \snbb \left(\PAAJ \right)=\PMNJ \circ \ANtJ \left( =\AMJ \right),  \ \ \ \forall \cM \leq \cN.
\end{equation}
By construction we have that:
\begin{equation}
\PAAJ =g^J_{\cN} \left(\PA \right).
\end{equation}
This means that we want to show:
\begin{equation}\label{eq10:3}
\PMNJ \circ \snbbJ \circ g^J_{\cN} \left(\PA \right) =\ANJ,  \ \ \ \forall \cM \leq \cN.
\end{equation}
On the other hand, we have:
$\snbbJ \circ g^J_{\cN}=\sntJ$ and $\PMNJ \circ \sntJ=\smJ$. This means that relation \eqref{eq10:3} is equivalent to:
\begin{equation}
\smJ \left(\PA \right) =\ANJ,  \ \ \ \forall \cM \leq \cN.
\end{equation}
This is precisely the statement of the unification result up to level $\cN$ from the Unification Theorem \ref{THEOREMAU}, which concludes the proof of this Lemma.
\end{proof}

\begin{lem}[Well behaviour of the intersection pairings when changing the level]\label{lemma3} When changing the level from $\cN$ to $\cN+1$, the intersection pairings recover one another through the induced map $l_{\cN}$ as below:
\begin{equation}
l^J_{\cN} \left(\PAANJ \right)=\PAAJ.
\end{equation}
 \end{lem}
\begin{proof}
From the construction of our quotients, which use the kernels of specialisation maps for all levels bounded by a fixed colour, we have that the maps $\snbbJ$ and $\snnbbJ$ are injective. Following Lemma \ref{lemma2}, we have the property that:
$$\snbbJ \left(\PAAJ \right)=\ANtJ.$$
This means that our statement is equivalent to: $$\PPNJ \left(\ANNtJ \right)=\ANtJ.$$
This is precisely the property from Remark \ref{rk1}, and so the statement holds.
\end{proof}
As a conclusion after this discussion, we have that:
\begin{equation}
l^J_{\cN} \left(\PAANJ \right)=\PAAJ.
\end{equation}
This shows that there exists a well-defined element in the projective limit, which is a knot invariant:
$$\IJJ \in \LLhJ.$$
The last part that we need to prove is that this universal invariant recoveres all coloured Jones invariants, for any level, as stated in relation \eqref{eq10:4}:
\begin{equation}
\ANJ=  \IJJ|_{\usnJ}.
\end{equation}

\begin{lem}[Commutation of the squares from the diagram]
All squares associated to consecutive levels $\cN$ and $\cN+1$ commute.
\end{lem}
\begin{proof}
This comes from the commutativity when we project onto each factor and by the definition of the quotient maps.   
\end{proof}
This shows that there exists a well-defined ring homomorphism between the limits, which we denote as below:
  \begin{equation}
\suJ: \LLhJ \rightarrow \LaJ.
  \end{equation}

Following the commutativity of the squares at all levels, we have that:

\begin{equation}
\suJ \left( \IJJ \right)=\AaJ.
\end{equation}
Also, using the commutativity of the diagrams together with Definition \ref{eq10:limit} we obtain that:
\begin{equation}
\PhNJ \circ \suJ=\PNNJ \circ \snbJ=\usnJ.
\end{equation}
Now, following relation \eqref{eq10:5}, we have that:
\begin{equation}
\PhNJ(\AaJ)=\ANJ.
\end{equation}
Using the previous three relations, we conclude that:
\begin{equation}
\begin{aligned}
\IJJ|_{\usn}&=\usnJ \left(\IJJ \right)=\\
&=\PhNJ \circ \suJ\left(\IJJ \right)=\\
&=\PhNJ(\AaJ)=\\
&=\ANJ, \ \ \ \ \forall \cN \in \N, \cN\geq 2.
\end{aligned}
\end{equation}
This concludes our construction and shows that we have a universal coloured Jones invariant for links $\IJJ$, constructed from graded intersections in configuration spaces, which recovers the coloured Jones invariants at all levels, through specialisation of coefficients.
\end{proof}

\section{Relation between our universal invariant and Habiro's invariant \cite{H2},\cite{W}} \label{S:rel}

So, we have the universal invariant $\IJJ \in \LLhJ$ in the universal ring which is the projective limit of quotients of
$\LL=\Z[x^{\pm 1},d^{\pm 1}].$
\begin{prop}[Structure of the quotient ideals] The sequence of nested ideals in $\LL$ that we used for the universal invariant have the following form:\\\begin{equation}
\begin{aligned}
&\tilde{I}_{\cN}^{J}:=\langle \prod_{i=1}^{\cN} (xd^{i-1}-1)\rangle \subseteq \LL.\\
\end{aligned}
\end{equation}
\end{prop}
This comes from the definition of the specialisation maps from the relations \eqref{u1J} and  \eqref{ker}. 
Next, we know that they lead to sequences of associated quotient rings and projective limits:
\begin{equation}
\begin{aligned}
 \phantom{A} \hspace*{-80mm}  \ l^J_{\cN} \hspace{10mm} l^{J}_{\cN+1}  \hspace{36mm}\\
\cdots \LLNJ \leftarrow \LLNNJ \leftarrow \cdots  \ \ \ \ \ \ \ \ \ \ \ \ \LLhJ:= \underset{\longleftarrow}{ \mathrm{lim}} \ \LLNJ
\end{aligned}
\end{equation}
Now we will explore the correlation to the work of Habiro and Willetts. In \cite{W}[{Definition 14, Lemma 15}], Willetts considered the sequence of ideals and quotients:
\begin{equation}
\begin{aligned}
&\tilde{I}_{\cN}^{W}:=\langle \prod_{i=1}^{k-1} (xd^{i-1}-1) \prod_{i=k-1}^{\cN-1} (d^i-1) \mid 1\leq k \leq \cN-1 \rangle \subset \LL\\
&\LLNNW=\LL / \tilde{I}_{\cN}^{W}
\end{aligned}
\end{equation}
They lead to the following projective limit:
\begin{equation}
\begin{aligned}
&\cdots \LLNW \leftarrow \LLNNW \leftarrow \cdots \ \ \ \ \ \ \ \ \ \ \ \ \ \LLhW:= \underset{\longleftarrow}{ \mathrm{lim}} \ \LLNW.
\end{aligned}
\end{equation}

\begin{rmk}[Semi-simple versus non semi-simple]
For the case of the topological model for the coloured Jones polynomials, we would like to emphasise that the variable $u$ does not play any important role, as it gets anyway specialised in the same way as $x$, so we can work without it. 


However, for the topological model for coloured Alexander case, the variable $u$ is important and it plays a central role in each specialisation at each level $\cN$, so we cannot separate it directly as for the case of coloured Jones universal invariants, where the representation theory is semi-simple.
\end{rmk}
\begin{defn}[Ideals and limit rings]
As a summary, we have two sequences of ideals that lead to two universal rings, as below:

\begin{equation}
\begin{aligned}
&\LLhJ= \underset{\longleftarrow}{\mathrm{lim}} \ \LLNJ =\underset{\longleftarrow}{ \mathrm{lim }} \ \LL /\tilde{I}_{\cN}^{J} & \tilde{I}_{\cN}^{J}:=\langle \prod_{i=1}^{\cN} (xd^{i-1}-1) \rangle \subset \LL.\\
& \LLhW=\underset{\longleftarrow}{ \mathrm{lim }} \ \LL /\tilde{I}_{\cN}^{W} &\tilde{I}_{\cN}^{W}:=\langle \prod_{i=1}^{k-1} (xd^{i-1}-1) \prod_{i=k-1}^{\cN-1} (d^i-1) \mid 1\leq k \leq \cN-1 \rangle \subset \LL.\\
\end{aligned}
\end{equation}
\end{defn}
Willetts showed the connection between Habiro's Universal invariant seen in the ring $\LLhW$ and fractions between ADO invariants and the Alexander polynomial.
\begin{thm}[Habiro's invariant \cite{H2} and non semi-simple invariants \cite{W}] \label{Will}There exists a well-defined universal invariant in the limit ring $\IW \in \LLhW$ such that:
\begin{equation}
\begin{aligned}
&\IW \Bigm| _{d=\xi_{\cN}^{-1}}  =~ \frac{\Phi^{\cN}(K,x)}{\Delta(K,x^{2\cN})}\\
&\IW \Bigm| _{x=d^{\cN-1}}  =~ J_{\cN}(K). 
\end{aligned}
\end{equation} 
\end{thm}
\begin{thm}[Projection onto Habiro's invariant \cite{H2},\cite{W}]\label{RELATION}There is a natural map between the universal rings $\LLhJ$ and $\LLhW$ which sends our universal Jones invariant to Habiro's universal invariant used by Willetts:
$$\pi: \LLhJ \rightarrow \  \LLhW$$
\begin{equation}
\pi(\IJJ)=\IW. 
\end{equation} 
\end{thm}
In order to prove this Theorem, we start with a discussion concerning relations between the three type of rings and quotients that we have so far. 
\begin{rmk}[Quotients and specialisation maps]\label{qmps}
1) Let $R$ be a ring and $A,B \subseteq R$. There exists a well-defined quotient map $R / A \rightarrow R / B$ if and only if $A \subseteq B$.\\
2) Let $R,R'$ be two rings, $A \subseteq R$ ideal and $f: R \rightarrow R'$ a ring homomorphism. Then there exists a well defined induced map $\bar{f}:R/ A \rightarrow R'$ if and only if $A \subseteq Ker(f)$.
 
\end{rmk}
\begin{rmk}[Specialisation maps from the level $\cN$ quotients]
We have in mind specialisations that are induced from the quotients:
$$\LLNJ, \ \LLNW.$$ 
We have three types of ideals $\tilde{I}_{\cN}^{J}, \tilde{I}_{\cN}^{W}$ and the ones associated to the ADO invariants. We notice that:\\
- the ideal generated by $(xd^{\cM-1}-1)$ describes specialisations related to coloured Jones polynomials at levels $\cM \leq \cN$ [{\em Natural parameters}]\\
- the ideal generated by $(d^{\cM}-1)$ describes specialisations related to coloured Alexander polynomials  at levels $\cM \leq \cN$ [{\em Roots of unity}].
\end{rmk}
\begin{prop}[No specialisations from Willetts's quotients to rings for coloured Jones polynomials]
 Since the ideals $\tilde{I}_{\cN}^{W}$ are larger than the kernel of the above specialisations, this means that there is no well defined specialisation map from the intermadiary rings used by Willetts to either generic parameters or roots of unity:
 $$\nexists \ \psi_A^W: \LLNW \rightarrow \C(\xi_{\cN}^{\pm1})$$
  $$\nexists \ \psi^W: \LLNW \rightarrow \LNJ.$$
\end{prop}
This comes directly from Remark \ref{qmps}.
\begin{rmk}[No unification at intermediary levels]
So the question of having a knot invariant at Willetts's intermediary levels that sees either coloured Jones invariants or coloured Alexander invariants is impossible to even state. 

\end{rmk}

\begin{prop}[Relation between the two quotients and associated universal rings]
We have $\tilde{I}_{\cN}^{J} \subseteq \tilde{I}_{\cN}^{W} \subseteq \Z[x^{\pm1}, d^{\pm 1}]$. This means that at the level of quotient rings we have:
\begin{equation}\label{levelNrings}
\begin{aligned}
& \hspace{30mm}\pi_{\cN}\\
&\Z[x^{\pm1}, d^{\pm 1}] \rightarrow \ \LLNJ \doublearrow \ \LLNW
\end{aligned}
\end{equation}
and at the limit 
\begin{equation}\label{limits}
\begin{aligned}
& \hspace{30mm}\pi\\
&\Z[x^{\pm1}, d^{\pm 1}] \rightarrow \ \LLhJ \rightarrow \ \LLhW. \\
\end{aligned}
\end{equation}

\end{prop}
\begin{proof}
The relation presented in \eqref{levelNrings} comes from the fact that $$\tilde{I}^J_{\cN} \subseteq \tilde{I}^W_{\cN}. $$ Then, the map $\pi: \LLhJ \rightarrow \  \LLhW$
at the limit comes from the property that we have an induced map between the limit rings.  
\end{proof}
Now we put these elements together and show the proof of Theorem \ref{RELATION} as follows.
\begin{proof}
 Our universal invariant is a limit of the invariants $\LLNJ$: $$\LLhJ:= \underset{\longleftarrow}{\mathrm{lim}} \ \LLNJ.$$
Here, we go back and we use the connection to representation theory. We use the representation theory of $U_q(sl(2))$ over $\Z[x^{\pm1},d^{\pm1}]$. There, there is a Verma module of infinite dimension $\hat{V}$, generated by an infinite family of vectors $\{v_0, v_1,...\}$. Then, for two fixed parameters $n,m$ there is the associated weight space:

(see \cite{Cr1}[Section 2.2, Section 2.3])
$$\hat{V}_{n,m}=\langle  v_{e_1}\otimes... \otimes v_{e_n}| e_1+...+e_{n}=m\rangle.
$$
Further, if we fix a level $\cN$, the $\cN^{th}$ finite weight space over two parameters is given by:
$$V^{\cN}_{n,m}\subseteq \hat{V}_{n,m}$$ 
$$V^{\cN}_{n,m}=\langle  v_{e_1}\otimes... \otimes v_{e_n}| e_1+...+e_{n}=m, 0 \leq e_1,...,e_n \leq \cN\rangle.
$$

Since $\LLNJ$ is equipped with a projection map from $\LL$, the weight spaces $V^{\cN}_{n,m}$ can be specialised to weight spaces over this quotient ring. Let us denote them by $V^{\cN}_{n,m}\mid_{\LLNJ}$.
There is an induced $\beta_n$ action on these, which we denote by $\varphi^{\cN}(\beta_n)$. In the next part we denote by $qtr_{2,..,n}$ the partial quantum trace of a braid action (we refer to \cite{M2} for the precise formula).

Our invariant $\IJJ$ for a knot $K=\hat{\beta_n}$ corresponds to a sum of quantum partial traces of the induced $\beta_n$ actions $\varphi^{\cN,m}(\beta_n)$ on all the weight spaces over the quotient ring $\LLNJ$:  $$\{V^{\cN}_{n,m}\mid_{\LLNJ} \mid 0 \leq m \leq (n-1)(\cN-1)\}.$$ More precisely, we have:
\begin{equation}\label{form}
\PAAJ= \sum_{m=0}^{(n-1)(\cN-1)} qtr_{2,..,n}\left(\varphi^{\cN,m}(\beta_n)\right)(v_0).
\end{equation}

We have as well a well-defined induced action on the weight spaces $\hat{V}_{n,m}$. We denote it by $\hat{\varphi}^{m}(\beta_n)$.

\begin{lem}
If we fix $\cN \in \N$, then for any $m \in \N$, the partial trace of the generic braid action $\hat{\varphi}^{m}(\beta_n)$ specialised over the quotient ring $\LLNJ$, becomes equal to the partial trace of the action on the level $\cN$ weight space: 
\begin{equation}\label{annulspec}
qtr_{2,..,n}\left(\varphi^{\cN,m}(\beta_n)\right)(v_0)=\left(qtr_{2,..,n}\hat{\varphi}^{m}(\beta_n)\right)(v_0)\mid_{\LLNJ}, \forall m\in \N.
\end{equation}
\end{lem}
The proof of this fact follows in an analog manner as the proof of Lemma \ref{pr2}. The main idea is that if we have the partial trace with respect to $v_0$, and we work with braids that give knots by braid closure, then the extra vectors from $\hat{V}_{n,m}$ which do not belong to $V^{\cN}_{n,m}$ get coefficients that belong to the ideal $\tilde{I}^J_{\cN}$, and so they get traced by zero when looking in the quotient ring $\LLNJ$.

Folowing \cite{W}, the Habiro invariant \cite{H2} is the sum of traces of generic representations on weight spaces on the Verma module:
\begin{equation}
\IW= \sum_{m} qtr_{2,..,n}\left(\hat{\varphi}^{m}(\beta_n)\right)(v_0).
\end{equation}
Let us denote by $\IW_{\cN}$ the $\cN^{th}$ component of $\IW$ in $\LLNW$.
Then, $\IW_{\cN}$ is obtained by considering the image $\IW$ in the corresponding quotient ring:
$$\IW_{\cN}=\IW\mid_{\LLNW}= \sum_{m} qtr_{2,..,n}\left(\hat{\varphi}^{m}(\beta_n)\right)(v_0)\mid_{\LLNW}.$$
Since $\tilde{I}^J_{\cN} \subseteq \tilde{I}^W_{\cN}$ and using formulas \eqref{form} and \eqref{annulspec} for our invariant at level $\cN$ we have: 
\begin{equation}
\begin{aligned}
\IW_{\cN}&=\IW \mid_{\LLNW}=\\
&=\sum_{m} qtr_{2,..,n}\left(\hat{\varphi}^{m}(\beta_n)\right)(v_0)\mid_{\LLNW}=\\
&= \pi_{\cN}\left(\sum_{m} qtr_{2,..,n}\left(\hat{\varphi}^{m} (\beta_n)\right)(v_0)\mid_{\LLNJ}\right)=\\
&= \pi_{\cN}\left(\sum_{m} qtr_{2,..,n}\left(\varphi^{\cN,m}(\beta_n)\right)(v_0)\right)=\\
&= \pi_{\cN}\left(\PAAJ\right).
\end{aligned}
\end{equation}

This means that $ \pi_{\cN}\left(\PAAJ\right)=\IW_{\cN}$ for all levels $\cN$, so our universal geometric Jones universal invariant recovers Habiro's universal invariant \cite{H2}, if we project onto the ring $\LLhW$:
$$\pi: \LLhJ \rightarrow \  \LLhW$$
\begin{equation}
\pi(\IJJ)=\IW. 
\end{equation} 
which concludes Theorem \ref{RELATION}.
\end{proof}

\begin{figure}[H]
\begin{equation*}
\begin{aligned}
&\hspace{20mm}\LLhJ= \underset{\longleftarrow}{\mathrm{lim}} \ \LLNJ  \\
& \LLNJ= \Z[x^{\pm1}, d^{\pm 1}] /\langle \prod_{i=1}^{\cN} (xd^{i-1}-1) \rangle.
\end{aligned}
\end{equation*}
\begin{center}
\hspace*{-30mm}\begin{tikzpicture}
[x=1.2mm,y=1.4mm]
\node (Jl)               at (-80,10)    {$ \LLNJ \ni \PAAJ$};

\node (Wn)               at (-85,-30)    {$\LLNW$};

\node (Jn)               at (-70,0)    {$J_{\cN}(K)$};
\node (Jn-1)               at (-70,-10)    {$J_{\cN-1}(K)$};
\node (Jm)               at (-70,-20)    {$J_{\cM}(K)$};

\node (A'n)               at (-20,0)    {$\frac{\AN}{\Delta(x^{2\cN})}$};
\node (A'n-1)               at (-20,-10)    {$\frac{\Phi^{\cN-1}(K)}{\Delta(x^{2(\cN-1})}$};
\node (A'm)               at (-20,-20)    {$\frac{\AM}{\Delta(x^{2\cM})}$};

\node(IJ)[draw,rectangle,inner sep=3pt,color=red] at (-80,15) {$\cN^{th}$ unified Jones invariant};

\node(UJ)[draw,rectangle,inner sep=3pt,color=red] at (-60,30) {Universal Coloured Jones invariant};
\node[draw,rectangle,inner sep=3pt,color=blue] at (-40,-35) {Habiro  universal invariant \cite{H2},\cite{W}};

\node (J)               at (-50,25)    {$\IJJ \in \LLhJ$};
\node (W)               at (-50,-30)    {$\IW \in \LLhW$};

\draw[->,dashed]             (Wn)      to node[right,xshift=2mm,font=\large]{}   (Jn);
\draw[->,dashed]             (Wn)      to node[right,xshift=2mm,font=\large]{}   (Jn-1);
\draw[->,dashed]             (Wn)      to node[right,xshift=2mm,font=\large]{$\nexists$}   (Jm);

\draw[->]             (Jl)      to node[right,xshift=2mm,font=\large]{}   (Jn);
\draw[->]             (Jl)      to node[right,xshift=2mm,font=\large]{}   (Jn-1);
\draw[->]             (Jl)      to node[right,xshift=2mm,font=\large]{}   (Jm);

\draw[->]             (J)      to node[right,xshift=2mm,font=\large]{}   (Jn);
\draw[->]             (W)      to node[right,xshift=2mm,font=\large]{}   (Jn);

\draw[->]             (J)      to node[right,xshift=2mm,font=\large]{}   (A'n);
\draw[->]             (W)      to node[right,xshift=2mm,font=\large]{}   (A'n);


\draw[->]             (J)      to node[right,xshift=2mm,font=\large]{$\pi$}   (W);
\draw[->,dashed, red]             (IJ)      to node[right,xshift=2mm,font=\large]{$\underset{\longleftarrow}{\mathrm{lim}}$}   (UJ);

\end{tikzpicture}
\hspace*{-20mm}\begin{equation}
\begin{aligned}
& \LLhW=\underset{\longleftarrow}{ \mathrm{lim }} \ \LLNW  \\ 
&\hspace{-40mm}\LLNW= \Z[x^{\pm1}, d^{\pm 1}] / \langle \prod_{i=1}^{k-1} (xd^{i-1}-1) \prod_{i=k-1}^{\cN-1} (d^i-1) \mid 1\leq k \leq \cN-1 \rangle.\\
\end{aligned}
\end{equation}
\end{center}
\caption{The two universal invariants}
\end{figure}

 {\itshape Universit\'e Clermont Auvergne-LMBP, Campus des C\'ezeaux 3, place Vasarely, 63178 Aubi\`ere, France; Institute of Mathematics “Simion Stoilow” of the Romanian Academy, 21 Calea Grivitei Street, 010702 Bucharest, Romania.}

\

{\itshape cristina.anghel@uca.fr, cranghel@imar.ro}

\noindent \href{http://www.cristinaanghel.ro/}{www.cristinaanghel.ro}

\end{document}